\documentclass[11pt,twoside]{article} % Stat tech report

\RequirePackage[OT1]{fontenc}
\RequirePackage{amsthm,amsmath,amsfonts,amssymb,thmtools}
\RequirePackage[colorlinks]{hyperref}
\RequirePackage{epsfig, graphicx, color}
\RequirePackage{times}
\RequirePackage{latexsym}
\RequirePackage{psfrag}
\RequirePackage{color}
\usepackage{natbib}
\usepackage{xcolor}
\usepackage{tikz}
\usepackage{multirow}

\usepackage[linesnumbered,ruled,lined]{algorithm2e}
\usepackage{fullpage}
\usepackage{epsf,epsfig}
\usepackage{fancyheadings}
\usepackage{graphics,graphicx}
\usepackage{enumerate}
\usepackage{bbm}
\usepackage{mathtools}
\usepackage{tikz,subcaption,amsthm}

\usetikzlibrary{positioning,arrows,decorations,decorations.markings,decorations.pathreplacing,decorations.pathmorphing,shadows,positioning,arrows.meta,matrix,fit}
\usetikzlibrary{shapes,snakes}

\newtheorem{theorem}{Theorem}

\newtheorem{lemma}[theorem]{Lemma}

\theoremstyle{definition}
\newtheorem{definition}{Definition}
\newtheorem{example}[definition]{Example}

\newtheorem{proposition}{Proposition}

\theoremstyle{remark}
\newtheorem{remark}{Remark}

\usepackage{cleveref}

\newcommand{\R}{\mathbb{R}}
\newcommand{\PP}{\mathbb{P}}
\newcommand{\E}{\mathbb{E}}
\newcommand{\argmin}{{\arg\hspace*{-0.5mm}\min}}

\newcommand{\Q}{\mathbb{Q}}

\newcommand{\one}{\mathbbm{1}}

\def\st{\textit{s.t.}}

\def\cS{\mathcal{S}}

\def\cB{\mathcal{B}}
\def\cC{\mathcal{C}}

\def\cI{\mathcal{I}}
\def\cT{\mathcal{T}}

\def\pr{\mathbb{P}}

\def\argmin{{\arg\hspace*{-0.5mm}\min}}
\def\var{\mathrm{var}}

\def\bx{\boldsymbol X}
\def\by{\boldsymbol Y}
\def\tby{\tilde{\boldsymbol Y}}
\def\bz{\boldsymbol Z}

\def\dx{d_X}
\def\dy{d_Y}
\def\dz{d_Z}
\def\bs{\boldsymbol}

\def\htac{\hat{T}^\textrm{AC, cond}}
\def\htaci{\hat{T}^\textrm{AC}}
\def\tac{T^\textrm{AC, cond}}
\def\taci{T^\textrm{AC}}
\def\bsy{\bs{y}}
\def\bsz{\bs{z}}
\def\bsx{\bs{x}}
\def\bsa{\bs{a}}
\def\bsb{\bs{b}}
\def\asconv{\overset{\textrm{a.s.}}{\to}}

\newcommand\independent{\protect\mathpalette{\protect\independenT}{\perp}}
\def\independenT#1#2{\mathrel{\rlap{$#1#2$}\mkern2mu{#1#2}}}

\newcommand{\vertiii}[1]{{\left\vert\kern-0.25ex\left\vert\kern-0.25ex\left\vert #1 
    \right\vert\kern-0.25ex\right\vert\kern-0.25ex\right\vert}}

\makeatletter
\newcommand{\customlabel}[2]{%
	\protected@write\@auxout{}{%
		\string\newlabel{#1}{%
			{#2}{\thepage}{}{}{}{}%
		}%
	}%
}
\makeatother
    
\title{A multivariate extension of Azadkia-Chatterjee's rank coefficient}

\author{Wenjie Huang$^1$, Zonghan Li$^1$ and Yuhao Wang$^{1,2}$\thanks{The authors are listed in alphabetical order, correspondence should be addressed to Yuhao Wang: \href{mailto:yuhaow@tsinghua.edu.cn}{yuhaow@tsinghua.edu.cn}} \\ \\
$^1$ Institute for Interdisciplinary Information Sciences, Tsinghua University, Beijing, China\\
$^2$ Shanghai Qi Zhi Institute, Shanghai, China
}

\begin{document}

\maketitle

\begin{abstract}
The Azadkia-Chatterjee coefficient is a rank-based measure of dependence between a random variable $Y \in \R$ and a random vector $\bz \in \R^{\dz}$. This paper proposes a multivariate extension that measures the dependence between random vectors $\by \in \R^{\dy}$ and $\bz \in \R^{\dz}$, based on $n$ i.i.d. samples. The proposed coefficient converges almost surely to a limit with the following properties: i) it lies in $[0, 1]$; ii) it is equal to zero if and only if $\by$ and $\bz$ are independent; and iii) it is equal to one if and only if $\by$ is almost surely a function of $\bz$. Remarkably, the only assumption required by this convergence is that $\by$ is not almost surely a constant vector. We further prove that under the same mild condition and after a proper scaling, this coefficient converges in distribution to a standard normal random variable when $\by$ and $\bz$ are independent. This asymptotic normality result allows us to construct a Wald-type hypothesis test of independence based on this coefficient. To compute this coefficient, we propose a merge sort based algorithm that runs in $O(n (\log n)^{\dy})$. Finally, we show that it can be used to measure the conditional dependence between $\by$ and $\bz$ conditional on a third random vector $\bx$, and prove that the measure is monotonic with respect to the deviation from an independence distribution under certain model restrictions.
\end{abstract}

\section{Introduction}\label{sec:intro}

% \begin{enumerate}
%     \item Refer to criteria;

%     \item Two main approaches: rank based approach and kernel based approach

%     \item Rank based approach has many adavntanges;

%     \item However, it still has the limitation that it cannot be scaled to multivariate setting;

%     \item There are some works using optimal transport, however, they cannot have distribution-free estimation of coefficient;

%     \item  We develop a new approach that can achieve distribution-free estimation. Moreover, we prove its asymptotic normality. We provide additional analysis of its monotonicity properties. 
% \end{enumerate}

Measuring the dependence (and conditional dependence) of random variables is a fundamental task in statistics. Classical coefficients include Pearson's correlation coefficient, Spearman's $\rho$, and Kendall's $\tau$. Recently, there has been growing interest in developing new dependence measures tailored to specific applications; see~\citet{chatterjee2021new,josse2016measuring} for surveys. However, as argued by~\citet{chatterjee2021new}, most existing measures suffer from two drawbacks: they are built to test independence rather than quantify the strength of dependence, and they lack simple asymptotic null distributions, so computationally intensive procedures such as permutation or bootstrap tests are required to obtain p-values.

In light of these issues,~\citet{chatterjee2021new} recommended to propose a coefficient that is
\begin{quote}
    ``(a) as simple as the classical coefficients like Pearson’s correlation or Spearman’s correlation, and yet (b) consistently estimates some simple and interpretable measure of the degree of dependence between the variables, which is 0 if and only if the variables are independent and 1 if and only if one is a measurable function of the other, and (c) has a simple asymptotic theory under the hypothesis of independence, like the classical coefficients''
\end{quote}
These criteria, especially (b), seem very natural.  In fact, requirement similar to (b) was already discussed by \citet{renyi1959measures}: 
\begin{quote}
``It is natural to choose the range $[0, 1]$ and to make correspond the value $1$ to strict dependence and thus $0$ to independence.''
\end{quote}

In the same paper, \citet{chatterjee2021new} further showed that at least in the univariate case, a coefficient satisfying these criteria is possible. More specifically, they introduced a coefficient that converges almost surely to a measure of dependence between two random variables $Y, Z \in \R$ -- a measure previously examined by \citet{dette2013copula,gamboa2018sensitivity}. Remarkably, this measure satisfies (b), and the convergence of the coefficient in~\citet{chatterjee2021new} requires only that $Y$ is not almost surely a constant. Moreover, under the null hypothesis where $Y \independent Z$, they proved that this coefficient is asymptotically normal under the same mild condition. The proposed coefficient is a rank-based approach constructed using only the ranks of the i.i.d. samples of two random variables $Y$ and $Z$. 
This rank-based nature also confers two key advantages: robustness to outliers and contamination; and invariant under certain data transformations, such as linear transformations (see e.g.~\Cref{sec:uniac}). 

As an extension,~\citet{azadkia2021simple} relaxed the dimensionality restriction of $Z$, allowing it to be a random vector $\bz \in \R^{\dz}$. Moreover, they showed that the asymptotic limit of the proposed coefficient, denoted as $\taci(Y, \bz)$, has the property that, given a third vector $\bx \in \R^{\dx}$, $\taci(Y, (\bz, \bx))$ and $\taci(Y, \bx)$ satisfy
\begin{equation}\label{eq:thirdproperty}
\taci(Y, (\bz, \bx)) \ge \taci(Y, \bx) \;\&\; \taci(Y, (\bz, \bx)) = \taci(Y, \bx) \;\textrm{if.f.}\; Y \independent \bz \mid \bx.
\end{equation}
This allows us to use $\taci$ to measure not only dependence, but also \emph{conditional} dependence of $Y$ and $\bz$ given a third vector $\bx \in \R^{\dx}$. This new coefficient is now widely called the ``Azadkia–Chatterjee rank coefficient'' (AC coefficient). 
In a follow up study,~\citet{shi2024azadkia} proved the asymptotic normality of the new coefficient when $Y$ and $\bz$ are independent, under the additional assumption that $Y$ is continuous and $\bz$ is absolutely continuous. In another follow up work~\citep{lin2022limit}, the authors further analyzed the asymptotic normality of the AC coefficient when $Y$ and $\bz$ are not independent. \citet{ansari2022simple} further extended AC coefficient to the case where $\by \in \R^{\dy}$, subject to that each component of $\by$ is not almost surely a constant, as well as some additional regularity conditions of $\by$. While it represents an important extension, it remains unclear whether the estimation bias of their proposed coefficient is of order $o(1 / \sqrt{n})$ under the null hypothesis of independence. Consequently, their coefficient cannot be used directly to construct a valid Wald-type hypothesis test for independence, unlike classical correlation coefficients.

% \textcolor{red}{This paragraph requires more revisions.}
Alongside extensions of AC coefficients, another line of research has pursued an alternative strategy for developing \emph{rank-based} multivariate dependence coefficients: leveraging optimal transport techniques. This strategy has been the subject of intensive research over the past decade.
% In light of these works, a natural question is: \emph{can we construct a rank-based coefficient which allows $Y$ to be multivariate as well?} In the past decade, this question has been extensively studied in a separate line of research.
% In the past decade, this strategy has been extensively studied in a separate line of research.
The common approach in this line of work first employs optimal transport-based techniques to construct a so-called ``multivariate rank'' (or ``Monge-Kantorovich rank'')~\citep{Chernozhukov2017,hallin2021distribution}, and then uses this rank to replace the traditional rank for coefficient construction. Famous applications of this idea include~\citet{deb2023multivariate,deb2020measuring,deb2024distribution,shi2022universally,mordant2022measuring}; see e.g.~\citet{han2021extensions,chatterjee2024survey} for recent surveys. Under the null hypothesis that $\by$ and $\bz$ are independent, these rank-based coefficients are usually guaranteed to converge to an asymptotic distribution (not necessarily normal), so that one can construct asymptotically valid tests for testing independence based on these coefficients.

While this progress is encouraging, the asymptotic null distribution of these rank-based approaches are all derived under the extra assumption that both $\by$ and $\bz$ are \emph{absolutely continuous} with respect to the Lebesgue measure. Moreover, when $\by$ and $\bz$ are dependent, these coefficients either require extra assumptions on the joint distribution of $\by$ and $\bz$ in order to ensure its asymptotic convergence, or cannot satisfy all the desired properties in criterion (b). 
Finally, to the best of our knowledge, none of these coefficients satisfy~\eqref{eq:thirdproperty}, and thus cannot be easily extended to measure conditional dependence.

In light of these works, a natural question is: \emph{can we construct a rank-based multivariate dependence coefficient that satisfies all the criteria (a)--(c) advocated by~\citet{chatterjee2021new}, as well as~\eqref{eq:thirdproperty}, without invoking any distributional assumption?}
In this paper, we propose a new multivariate extension of Azadkia-Chatterjee's rank coefficient to overcome these barriers. We prove that it converges almost surely to a measure of dependence between $\by$ and $\bz$ as long as $\by$ is not almost surely a constant vector. Moreover, this measure inherits the criterion (b) advocated by~\citet{chatterjee2021new}, as well as~\eqref{eq:thirdproperty}, so that it can be used to measure the conditional dependence of $\by$ and $\bz$, given a third random vector $\bx$. We also derive the asymptotic null distribution of this coefficient under the hypothesis of independence. Remarkably, our asymptotic null analysis requires \emph{only} that $\by$ is not almost surely a constant, allowing us to yield a valid test based on this coefficient even when $\by$ and $\bz$ are \emph{not} absolutely continuous. This sheds light on the asymptotic analysis of~\citet{shi2024azadkia} even in the univariate $Y$ case. Finally, we further provide a merge sort based algorithm~\citep{knuth1997art} which can compute this coefficient in time complexity $O(n (\log n)^{\dy})$, and analyze the monotonicity properties of this dependence measure. 

The rest of this article is organized as follows. In~\Cref{sec:uniac}, we review previous results of Azadkia-Chatterjee coefficient in the univariate $Y$ case. In~\Cref{sec:multiac}, we provide a multivariate extension, and analyze its almost sure convergence property. In~\Cref{sec:normal}, we discuss its asymptotically normality under independence. In~\Cref{sec:algorithm}, we discuss how to compute this coefficient, and discuss the properties of this new coefficient. In~\Cref{sec:oracle,sec:converganaly}, we prove the key theoretical findings.

{\bf Notations.} Given two vectors $\bsa, \bsb \in \R^d$, we write $\bsa \wedge \bsb$ as their entrywise minimum $(\bsa_1 \wedge \bsb_1, \ldots, \bsa_d \wedge \bsb_d)$, and say that $\bsa \le \bsb$ if $\bsa_i \le \bsb_i$ holds for all $i \in [d]$. Write $\|\cdot\|$ as the $\ell_2$-norm. Given a sequence of random variables $A_n$ and a random variable $A$, we write $A_n \asconv A$ if $A_n$ converges to $A$ almost surely, and $A_n \overset{d}{\to} A$ if it converges in distribution.

\section{Azadkia-Chatterjee coefficient: preliminaries and property discussions}\label{sec:uniac}

Consider a random variable $Y \in \R$ and a random vector $\bz \in \R^{\dz}$, both defined on the same probability space, \citet{azadkia2021simple} proposed to measure the dependence between $Y$ and $\bz$ via:
\begin{equation}\label{eq:acmea}
    \taci(Y, \bz) := \frac{\int \var(\pr(Y \ge y \mid \bz)) d \mu_Y (y)}{\int \var(\one\{Y \ge y\}) d \mu_Y (y)},
\end{equation}
where $\mu_Y$ denotes the measure of the marginal distribution of $Y$. In the rest of this paper, we will also simplify it as $\taci$ when the context is clear. \citet{azadkia2021simple} claimed that, as long as $Y$ is not almost surely a constant, $\taci$ satisfies the following two properties:
\begin{enumerate}
\item[(P1)] $0 \le \taci(Y, \bz) \le 1$; moreover, $\taci(Y, \bz)$ is equal to $0$ if and only if $Y \independent \bz$; and is equal to $1$ if and only if $Y$ is almost surely a function of $\bz$;
\customlabel{p1}{(P1)}
    \item[(P2)] Consider in addition a random vector $\bx \in \R^{\dx}$ on the same probability space as $(Y, \bz)$, then $\taci(Y, (\bx, \bz)) \ge \taci(Y, \bx)$, where the equality holds if and only if $Y \independent \bz \mid \bx$.
  \customlabel{p2}{(P2)}
\end{enumerate}
\ref{p1} matches the criterion (b) mentioned in~\citet{chatterjee2021new}; \ref{p2} allows practitioners to use $\taci$ to construct a measure of conditional dependence of $Y$ and $\bz$, given a third random vector $\bx$. Moreover, since $\taci$ is constructed via the conditional CDF of $Y$, such structure further reveals the following property:
\begin{enumerate}
    \item[(P3)] For any strictly increasing function $f(\cdot)$ and any invertible $h(\cdot)$, $\taci(Y, \bz) = \taci(f(Y), h(\bz))$.
  \customlabel{p3}{(P3)}
\end{enumerate}
\ref{p3} implies,  for example, that $\taci(Y, \bz) = \taci(\alpha Y + \beta, \bz)$ for any $\alpha > 0, \beta \in \R$. Another classical example for strictly monotone function is the logarithmic transformation widely used in genetics studies. In light of~\ref{p2}, \citet{azadkia2021simple} proposed to measure the conditional dependence between $Y$ and $\bz$, conditional on $\bx$, via:
\begin{equation}
    \tac(Y, \bz \mid \bx) := \frac{\taci(Y, (\bz, \bx)) - \taci(Y, \bx)}{1 - \taci(Y, \bx)} \equiv \frac{\int \E(\var(\PP(Y \ge y \mid \bz, \bx) \mid \bx)) d \mu_Y(y)}{\int \E(\var(\one\{\by \ge \bsy\} \mid \bx)) d \mu_Y(y)}.
\end{equation}
Here the denominator $1 - \taci(Y, \bx)$ is used for normalization. Likewise, for simplicity we may write it as $\tac$ when the context is clear. Just like~\ref{p1}, provided $Y$ is not almost surely a function of $\bx$, we have $\tac \in [0, 1]$; $\tac$ is equal to $0$ if and only if $Y \independent \bz \mid \bx$, and is equal to $1$ if and only if $Y$ is almost surely a function of $(\bx, \bz)$.

Equipped with the new measures, the next question is to estimate them from finite samples. Let $\{Y_i, \bx_i, \bz_i\}_{i = 1}^n$ be i.i.d. copies of the triple $(Y, \bx, \bz)$, \citet{azadkia2021simple} proposed to estimate $\taci$ and $\tac$ via the following two coefficients:
\[
\htaci := \frac{\sum_{i = 1}^n (n \min\{R_i, R_{M_{\bz}(i)}\} - L_i^2)}{\sum_{i = 1}^n L_i (n - L_i)} \;\&\; \htac := \frac{\sum_{i=1}^n (\min\{R_i, R_{M_{(\bx, \bz)}(i)}\} - \min\{R_i, R_{M_{\bx}(i)}\})}{\sum_{i=1}^n (R_i - \min\{R_i, R_{M_{\bx}(i)}\})}.
\]
Here $M_{(\bx, \bz)}(i)$ stands for the index of the nearest neighbour of $(\bx_i, \bz_i)$ in the dataset $\{(\bx_j, \bz_j)\}_{j \neq i}$ in Euclidean distance, with random tie-breaking, and $M_{\bx}, M_{\bz}$ are defined analogously. $R_i$ is the number of $Y_j$'s whose value is no greater than $Y_i$, and $L_i$ is the number of $Y_j$'s whose value is no smaller than $Y_i$.

\citet{azadkia2021simple} also proved that they converge to $\taci$ and $\tac$ almost surely, under the mild condition that $Y$ is not almost surely a constant (for convergence of $\taci$) and $Y$ is not almost surely a function of $\bx$ (for $\tac$). In a follow up study, \citet{shi2024azadkia} further proved the asymptotic normality of $\taci$ when $Y$ and $\bz$ are independent and are absolutely continuous with respect to the Lebesgue measure.

\section{A multivariate extension}\label{sec:multiac}

In this section, we present an extension to the multivariate $\by$. To motivate this extension, first consider the following reformulation of the AC coefficient in the univariate case: 
\begin{equation}\label{eq:acnewcoeff}
\htaci := \frac{\sum_{i = 1}^n (n R(Y_i \wedge Y_{M_{\bz}(i)}) - L_i^2)}{\sum_{i=1}^n L_i (n - L_i)}
\end{equation}
where given a deterministic quantity $y \in \R$, we write $
R(y) := \sum_{i = 1}^n \one\{Y_i \le y\}$.
Informally then, $R(y)$ denotes the rank of $y$ in the dataset $\{Y_i\}_{i = 1}^n$. 
%where given index $(i,j)$, we redefine $R_{i,j} = \sum_{\ell=1}^n \one\{Y_\ell \le \min\{Y_i, Y_j\}\} = \sum_{\ell=1}^n \one\{Y_\ell \le Y_i\} \one\{Y_\ell \le Y_j\}$. Intuitively, it corresponds to the rank of $\min\{Y_i, Y_j\}$ in the dataset $\{Y_i\}_{i=1}^n$.
Motivated by such reformulation, it is straightforward to define $\htaci$ for a general $\dy$ via replacing the $R(\cdot)$ and $L_i$ used in~\eqref{eq:acnewcoeff} as
\begin{equation}\label{eq:rfunc}
R(\bs{y}) := \sum_{i=1}^n \one\{\by_i \le \bs{y}\} \quad\&\quad L_i := \sum_{j=1}^n \one\{\by_j \ge \by_i\},
\end{equation}
where recall that $\one\{\by_i \le \bs{y}\}$ is equal to $1$ when $\by_i$ is entry-wise no greater than $\bs{y}$.
The following result shows that such simple reformulation allows us to ensure the almost sure convergence of $\htaci$ to some dependence measure when $\dy > 1$:
\begin{proposition}\label{prop:multiac}
	Consider the $\htaci$ as in~\eqref{eq:acnewcoeff} with $R(\cdot)$ and $L_i$ as in~\eqref{eq:rfunc}. We have that $\htaci$ converges almost surely to the coefficient 
	\begin{equation}\label{eq:tac1}
	\frac{\int \var(\PP(\by \ge \bsy \mid \bz)) d \mu_{\by}(\bsy)}{\int \var(\one\{\by \ge \bsy\}) d \mu_{\by}(\bsy)}
	\end{equation}
	whenever $\int \var(\one\{\by \ge \bsy\}) d \mu_{\by}(\bsy) > 0$.
\end{proposition}

Just like the univariate case,~\eqref{eq:tac1} is equal to zero when $\by \independent \bz$. However, as we show in~\Cref{ex:naiveac}, such a relation is not ``if and only if''. 

\begin{figure}[t!]
    \centering
\begin{tikzpicture}[scale=4]

\draw[->] (-0.2,0) -- (1.2,0) node[right] {$y_1$};
\draw[->] (0,-0.2) -- (0,1.2) node[above] {$y_2$};

\fill[yellow] (0,1) -- (1,0) -- (1,1) -- cycle;

\draw[red, thick] (1,0) -- (0,1);

\node[below left] at (0,0) {0};
\node[below] at (1,0) {1};
\node[left] at (0,1) {1};

\end{tikzpicture}
    \caption{Density of random vector $\by$ in~\Cref{ex:naiveac}. The red area corresponds to the area where the distribution of $\by$ varies between $\bz = 1$ and $\bz = 0$, the yellow area corresponds to the area where distribution of $\by$ stays the same. The red and yellow area composites the support of $\by$. Apparently, for any $\bsy$ in the support of $\by$ (even around the boundary), we have $\pr(\by \ge \bsy \mid \bz = 0) = \pr(\by \ge \bsy \mid \bz = 1)$.}
    \label{fig:example}
\end{figure}
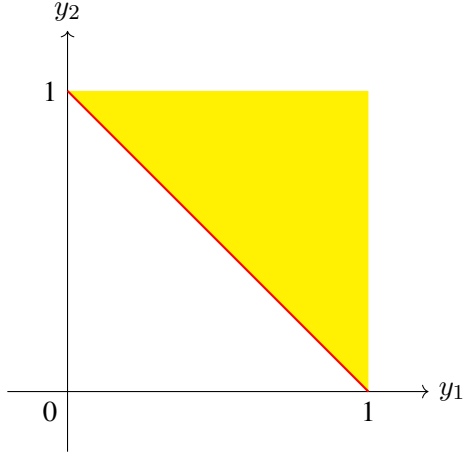

\begin{example}\label{ex:naiveac}
    Consider a $\by, \bz$, where $\bz \in \{0, 1\}$ is a uniform binary variable, $\by$ is a two-dimensional random vector whose support conditional on both $\bz = 1$ and $\bz = 0$ is equal to the triangle $\{(y_1, y_2): 0 \le y_1, y_2 \le 1 \;\&\; y_1 + y_2 \ge 1\}$. Moreover, we assume that for any set $\mathcal{B} \subseteq \{(y_1, y_2): 0 \le y_1, y_2 \le 1 \;\&\; y_1 + y_2 > 1\}$, 
    \[
    \pr(\by \in \mathcal{B} \mid \bz = 1) = \pr(\by \in \mathcal{B} \mid \bz = 0).
    \]
    At the same time, we assume that
    \[
    \forall \bsy \in \mathcal{L} := \{(y_1, y_2): 0 \le y_1, y_2 \le 1 \;\&\; y_1 + y_2 = 1\}, \pr(\by = \bsy \mid \bz = 1) =  \pr(\by = \bsy \mid \bz = 0) = 0,
    \]
    i.e., there is no point mass on the line segment $\mathcal{L}$, and there exists at least one $\mathcal{A} \subseteq \mathcal{L}$ such that $\pr(\by \in \mathcal{A} \mid \bz = 1) \neq \pr(\by \in \mathcal{A} \mid \bz = 0)$. See~\Cref{fig:example} for a illustration.

    Under these requirements, apparently $\by$ and $\bz$ are dependent. However, for any $\bsy$ in the support of $\mu_{\by}$, we can always have
    \[
    \pr(\by \ge \bsy \mid \bz = 1) = \pr(\by \ge \bsy \mid \bz = 0).
    \]
    Consequently,~\eqref{eq:tac1} is equal to zero. \hfill$\blacksquare$
\end{example}

This motivates us to consider a new coefficient that satisfies all the desirable properties in~\ref{p1}. As shown in~\Cref{ex:naiveac}, the reason the original coefficient fails to satisfy the requirements in~\ref{p1} is that its integration in the numerator is only over the support of $\mu_{\by}$. In~\Cref{ex:naiveac}, however, the two conditional CDFs $\pr(\by \ge \bsy \mid \bz = 1)$ and $\pr(\by \ge \bsy \mid \bz = 0)$ differ for some $\bsy$ outside the support of $\mu_{\by}$. This encourages us to define a coefficient that replaces $\mu_{\by}$ with a different measure on $\R^{\dy}$.

In this paper, we replace it by $\tilde{\mu}_{\by}$, a probability measure on $\R^{\dy}$ under which each component has the same marginal distribution as under $\mu_{\by}$, but all components are independent. More specifically, we redefine $\taci$ as
\begin{equation}\label{eq:multitaci}
\taci(\by, \bz) := \frac{\int \var(\PP(\by \ge \bsy \mid \bz)) d \tilde{\mu}_{\by}(\bsy)}{\int \var(\one\{\by \ge \bsy\}) d \tilde{\mu}_{\by}(\bsy)}.
\end{equation}
The following result shows that this new definition overcomes the limitations of~\eqref{eq:tac1}:
\begin{theorem}\label{thm:multiaccoef}
Suppose that $\by$ is not almost surely a constant, then the $\taci$ in~\eqref{eq:multitaci} satisfies: i) $\taci \in [0, 1]$; ii) $\taci = 0$ if and only if $\by \independent \bz$; and iii) $\taci = 1$ if and only if $\by$ is almost surely a function of $\bz$.
\end{theorem}
%
% \textcolor{red}{I got until here.}
%
Informally, \Cref{thm:multiaccoef} shows that the new $\taci$ satisfies the properties in~\ref{p1} for the multivariate $\by$.  
Using~\eqref{eq:multitaci}, we can analogously define the conditional coefficient $\tac$ as
\begin{equation}\label{eq:multitaccond}
\tac := \frac{\int \E(\var(\PP(\by \ge \bsy \mid \bz, \bx) \mid \bx)) d \tilde{\mu}_{\by}(\bsy)}{\int \E(\var(\one\{\by \ge \bsy\} \mid \bx)) d \tilde{\mu}_{\by}(\bsy)}.
\end{equation}
In~\Cref{thm:multiaccondcoef}, we show that the $\taci$ in~\eqref{eq:multitaci} satisfies~\ref{p2}, so that we can use this new $\tac$ to measure conditional dependence:
\begin{theorem}\label{thm:multiaccondcoef}
    Suppose that $\by$ is not almost surely a function of $\bx$, then the $\tac$ in~\eqref{eq:multitaccond} satisfies: i) $\tac \in [0, 1]$; ii) $\tac = 0$ if and only if $\by \independent \bz \mid \bx$; and iii) $\tac = 1$ if and only if $\by$ is almost surely a function of $(\bx, \bz)$.
\end{theorem}
Note that since $\tac$ defined in~\eqref{eq:multitaccond} can be equivalently expressed as $\frac{\taci(\by, (\bz, \bx)) - \taci(\by, \bx)}{1 - \taci(\by, \bx)}$,~\Cref{thm:multiaccondcoef} directly implies that $\taci$ satisfies~\ref{p2}. In addition, since $\taci$ is a CDF based measure, one can easily verify that it still satisfies~\ref{p3} with a slight modification:
\begin{enumerate}
    \item[(P3')] For any vector-valued function $f: \R^{\dy} \to \R^{\dy}$ of the form $f(\bs{a}) = (f_1(\bs{a}_1), \ldots, f_{\dy}(\bs{a}_{\dy}))$, where each $f_i : \R \to \R$ is a strictly increasing function, and any invertible $h(\cdot)$, $\taci(\by, \bz) = \taci(f(\by), h(\bz))$.
  \customlabel{p3'}{(P3')}
\end{enumerate}
A direct consequence of \ref{p3'} is that $\taci(\by, \bz) = \taci(\alpha \by + \bs{\beta}, \bz)$ for any scalar $\alpha > 0$ and $\bs{\beta} \in \R^{\dy}$, just like the univariate case.

To consistently estimate $\taci$ and $\tac$, we propose the following estimators, which, with slight abuse of notations, are still denoted by $\htaci, \htac$, respectively:
\begin{equation}\label{eq:multiacest}
\begin{aligned}
\htaci & := \frac{\sum_{i=1}^n (n \tilde{R}(\by_i \wedge \by_{M_{\bz}(i)}) - \check{L}_i^2)}{\sum_{i=1}^n (n - \check{L}_i) \check{L}_i} \\
\htac & := \frac{\sum_{i=1}^n (\tilde{R}(\by_i \wedge \by_{M_{(\bx, \bz)}(i)}) - \tilde{R}(\by_i \wedge \by_{M_{\bx}(i)}))}{\sum_{i=1}^n (\tilde{R}(\by_i) - \tilde{R}(\by_i \wedge \by_{M_{\bx}(i)}))},
\end{aligned}
\end{equation}
where provided $\dy$ permutations $\pi_1, \ldots, \pi_{\dy}: [n] \to [n]$ satisfying that for any $i$ and any $1 \le d_1 < d_2 \le \dy$, $\pi_{d_1}(i) \neq \pi_{d_2}(i)$, we write $\tilde{\by}_i := (Y_{\pi_1(i), 1},, \ldots, Y_{\pi_{\dy}(i), \dy})$ as a permuted vector and write $\tilde{R}(\cdot), \check{L}_i$ as:
\[
\tilde{R}(\bsy) := \sum_{i=1}^{n} \one\{\tilde{\by}_i \le \bsy\} \quad\&\quad \check{L}_i := \sum_{\ell = 1}^n \one \{\by_{\ell} \ge \tilde{\by}_i\}.
\]
In other words, we propose to first perform an entrywise permutation of the random vectors $\by_i$ to get $\{\tilde{\by}_i\}_{i=1}^n$, and then construct estimators based on the comparisons between $\{\by_i\}_{i = 1}^n$ and the new data $\{\tilde{\by}_i\}_{i = 1}^n$. The following theorem shows the almost sure convergence of the two estimators.

\begin{theorem}\label{thm:multiacest}
Consider the $\taci, \tac, \htaci, \htac$ defined in~\eqref{eq:multitaci}--\eqref{eq:multiacest}. 
Suppose that $\by$ is not almost surely a constant, then $\htaci\asconv \taci$; moreover, suppose that $\by$ is not almost surely a function of $\bx$, then $\htac \asconv \tac$.
\end{theorem}

\begin{remark}[rate of convergence]
    Following the proof of~\citet[Theorem~4.1]{azadkia2021simple}, it is easy to prove that the rate of convergence of $\htaci$ to $\taci$ is of order $n^{-\frac{1}{\max\{\dz, 2\}}}$, up to logarithmic factors. The dependence on $\dz$ is substantial as it relies on a nearest neighbor matching of $\bz_i$. This is faster than the $n^{-\frac{1}{\max\{\dz + \dy - 1, 2\}}}$ convergence rate for the estimators in~\citet{ansari2022simple} when $\dy \ge 2$.
\end{remark}

\subsection{Some intuitions for the proofs of~\Cref{thm:multiaccoef,thm:multiaccondcoef}}

In this subsection, we discuss some intuitions for the proofs of~\Cref{thm:multiaccoef,thm:multiaccondcoef}, more specifically, their ii) and iii). The ``if'' is straightforward, therefore, the ``only if'' constitutes the main challenge. The proofs rely repeatedly on the following lemma for the regular conditional probability\footnote{For details of regular conditional probability, see the beginning of~\Cref{sec:oracle}} of $\by$ given the other random vectors.
\begin{lemma}\label{lem:closedset}
Let $\bs{V}_1,\ldots,\bs{V}_m$ be any random vectors defined on the same probability space as $\by$. Let $B \subseteq \R^m$ be any closed set containing $\bs{0} \in \R^m$. Suppose that there exists a set $A \subseteq \R^{\dy}$ with $\tilde{\mu}_{\by}(A)=1$ such that for all $\bsy \in A$,
\[
\pr((\pr(\by \ge \bsy \mid \bs{V}_1), \ldots, \pr(\by \ge \bsy \mid \bs{V}_m)) \in B) = 1.
\]
Then, the above relation holds for all $\bsy \in \R^{\dy}$.
\end{lemma}

To see how~\Cref{lem:closedset} can be used to prove the theoretical claims, we take the ``only if'' part of~\Cref{thm:multiaccondcoef} ii) as an example. Based on some standard calculations, we can directly derive that $\tac = 0$ implies 
\[
\int \E[(\pr(\by \ge \bsy \mid \bx) - \pr(\by \ge \bsy \mid \bx, \bz))^2] d \tilde{\mu}_{\by}(\bsy) = 0,
\]
which means that there exists a set $A \subseteq \R^{\dy}$ with $\tilde{\mu}_{\by}(A) = 1$ such that for all $\bsy \in A$,
\[
\pr(\pr(\by \ge \bsy \mid \bx) - \pr(\by \ge \bsy \mid \bx, \bz) = 0) = 1.
\]
This allows us to apply~\Cref{lem:closedset} with $m = 2, \bs{V}_1 = \bx, \bs{V}_2 = (\bx, \bz)$ and $B = \{(p_1, p_2): p_1 - p_2 = 0\}$ to get that the above relation holds for all $\bsy \in \R^{\dy}$, which further implies the conditional independence relation.

\section{Asymptotic normality under independence}\label{sec:normal}

In this section, we discuss the asymptotic normality of $\htaci$ when $\by \independent \bz$. We first show that our new $\htaci$ is asymptotically normal. In order to formally describe this result, it will be convenient to define the following quantities:
\[
\Gamma_1 := \var(\tilde{F}(\by_1 \wedge \by_2)) - 2 \Gamma_2, \quad\&\quad \Gamma_2 := \E\left[\tilde{F}(\by_1 \wedge \by_2) \tilde{F}(\by_1 \wedge \by_3)\right] - \left(\E\left[\tilde{F}(\by_1 \wedge \by_2)\right]\right)^2,
\]
where for any $\bsy \in \R^{\dy}$, $\tilde{F}(\bsy) := \int \one\{\bsy' \le \bsy\} d \tilde{\mu}(\bsy')$. Given these quantities, we define the asymptotic variance of $\sqrt{n} \htaci$ as
\begin{equation}\label{eq:acvar}
\sigma_n^2 := \frac{ \Gamma_1 (1 + (n - 1) \pr(M_{\bz}(1) = 2, M_{\bz}(2) = 1)) + \Gamma_2 (n - 1) \pr(M_{\bz}(1) = M_{\bz}(2))}{\left(\int \var(\one\{\by \ge \bsy\}) d \tilde{\mu}_{\by}(\bsy)\right)^2}.
\end{equation}

Armed with the above definitions, we have
\begin{theorem}\label{thm:tacnormal}
Suppose that $\by$ is not almost surely a constant, then $\sigma_n > 0$ for all $n \ge 1$, and $0 < \liminf \sigma_n^2 \le \limsup \sigma_n^2 < \infty$. Suppose further that $\by$ and $\bz$ are independent, then 
\[
\frac{\sqrt{n} \htaci}{\sigma_n} \overset{d}{\to} \mathcal{N}(0, 1).
\]
\end{theorem}

% \begin{table}[htbp]
%   \centering
%   \caption{Overall caption for the two sub-tables}
%   \label{tab:both}
%   %
%   \begin{subtable}[t]{.45\linewidth}
%     \centering
%     \caption{First scenario ($d_y=2$)}
%     \begin{tabular}{cc}
%     $n$ & Rejection rate\\ 
%     50  & 0.06\\
%     200 & 0.05\\
%     \end{tabular}
%     \label{tab:small}
%   \end{subtable}
%   \hfill
%   %
%   \begin{subtable}[t]{.45\linewidth}
%     \centering
%     \caption{Second scenario ($d_y=10$)}
%     \begin{tabular}{cc}
%     $n$ & Rejection rate\\ 
%     50  & 0.08\\
%     200 & 0.05\\
%     \end{tabular}
%     \label{tab:large}
%   \end{subtable}
% \end{table}

When $\by$ is a scalar,~\citet{shi2024azadkia} proved asymptotic normality of $\htaci$ by requiring the joint distribution of $\by$ and $\bz$ to be absolutely continuous with respect to the Lebesgue measure. Our work extends this result by allowing their joint distribution to follow arbitrary distribution, thereby shedding light on the asymptotic normality of $\htaci$ even in the univariate case. Note that $\sigma_n^2$ may not have a limit, since $(n - 1) \pr(M_{\bz}(1) = 2, M_{\bz}(2) = 1)$ and $(n - 1) \pr(M_{\bz}(1) = M_{\bz}(2))$ are quantities depending on $n$. Instead, we can only prove that the two terms are of order $O(1)$, i.e., that $\htaci$'s convergence rate is of order $O(1 / \sqrt{n})$. When $\bz$ is absolutely continuous with respect to the Lebesgue measure on $\R^{\dz}$, $\sigma_n^2$ do have a finite positive limit, which is given by
\begin{equation}\label{eq:contsigma}
\sigma^2 := \frac{\Gamma_1 (1 + A_{\dz}) + \Gamma_2 B_{\dz}}{\left(\int \var(\one\{\by \ge \bsy\}) d \tilde{\mu}_{\by}(\bsy)\right)^2},
\end{equation}
where by writing $\lambda(\cdot)$ as the Lebesgue measure and $B(\bs{w}, r)$ as a Euclidean ball centered at $\bs{w}$ with radius $r$,
\[
\begin{aligned}
    A_d & := \left(2 - \frac{\int_0^{3/4} t^{(d-1)/2} (1-t)^{-1/2} d t}{\int_0^1 t^{(d-1)/2} (1-t)^{-1/2} dt}\right)^{-1},\\
    B_d &:= \iint_{\bs{w}_1, \bs{w}_2 \in \R^d: \max(\Vert \bs{w}_1 \Vert, \Vert \bs{w}_2 \Vert) \leq \Vert \bs{w}_1 - \bs{w}_2 \Vert} \exp(-\lambda(B(\bs{w}_1, \Vert \bs{w}_1 \Vert) \cup B (\bs{w}_2, \Vert \bs{w}_2 \Vert))) d \bs{w}_1 d \bs{w}_2.
\end{aligned}
\]
For more details, we refer the readers to ~\citet[Theorem~3.6 and~3.7] {shi2024azadkia}. In~\Cref{thm:limit}, we further prove that, when $\bz$ is a mix of absolutely continuous and discrete distributions, $\sigma_n^2$ is still convergent.
\begin{theorem}\label{thm:limit}
	Suppose that $\by$ is not almost surely a constant, and $\mu_{\bz}$, the distribution measure of $\bz$, follows the decomposition $\mu_{\bz} = (1 - \eta) \mu_{\bz, a} + \eta \mu_{\bz, d}$, where $\eta \in [0, 1]$, $\mu_{\bz, a}$ is absolutely continuous with respect to the Lebesgue measure on $\R^{\dz}$, and $\mu_{\bz, d}$ is a discrete measure, then
	\[
	\lim_{n \to \infty} \sigma_n^2 = \frac{\Gamma_1 (1 + (1 - \eta) A_{\dz}) + \Gamma_2 (\eta + (1 - \eta) B_{\dz})}{\left(\int \var(\one\{\by \ge \bsy\}) d \tilde{\mu}_{\by}(\bsy)\right)^2}.
	\]
\end{theorem}
Equipped with \Cref{thm:limit}, one might ask if $\sigma_n^2$ converges without any assumption on $\bz$. \Cref{thm:nolimit} provides a counterexample, showing that $\sigma_n^2$ can fail to converge.
\begin{theorem}\label{thm:nolimit}
Suppose that $\bz$ follows a Cantor distribution on $\R$ and $\by$ is not a constant almost surely, then $\sigma_n^2$ is not convergent as $n \to \infty$.
\end{theorem}
% A simple inspection of the proof of~\Cref{thm:nolimit} shows that as $n$ becomes sufficiently large, $\sigma_n^2$ can be approximated by
% \[
% \sigma_n^2 \approx \frac{\Gamma_1 (1 + A(n) / 2) + \Gamma_2 B(n) / 2}{\left(\int \var(\one\{\by \ge \bsy\}) d \tilde{\mu}_{\by}(\bsy)\right)^2},
% \]
% where 
% \[
% A(x) :=  \sum_{k \in \mathbb{Z}} 2^{-k} x e^{- 2^{-k} x} \quad\&\quad B(x) := \sum_{k \in \mathbb{Z}} \frac{2^{-2k} x^2}{e^{2^{-k} x} (e^{2^{-k} x} - 1)}.
% \]
% Our numerical analysis of $A(x)$ and $B(x)$ reveals that, although $\sigma_n^2$ is not convergent, it does not oscillate wildly as $n$ increases.

We now discuss the estimation of $\sigma_n^2$. We define 
$$\begin{aligned}
        \hat{\Gamma}_1&:= \frac{1}{n^2(n-1)}\sum_{i = 1}^{n - 1}\left(\tilde{R}(\by_i \wedge \by_{i + 1})\right)^2-\left(\frac{1}{n(n-1)}\sum_{i = 1}^{n - 1}\tilde{R}(\by_i \wedge \by_{i + 1})\right)^2-2\hat{\Gamma}_2;\\
        \hat{\Gamma}_2&:=\frac{1}{n^2(n-2)}\sum_{i = 1}^{n - 2}\tilde{R}(\by_i \wedge \by_{i + 1})\tilde{R}(\by_i \wedge \by_{i + 2})- \left(\frac{1}{n(n-1)}\sum_{i = 1}^{n - 1}\tilde{R}(\by_i \wedge \by_{i + 1})\right)^2.
\end{aligned}$$
These are empirical estimates of $\Gamma_1$ and $\Gamma_2$, respectively. Equipped with those estimates, we propose to estimate $\sigma_n^2$ via
\[
\hat{\sigma}_n^2 := \frac{\hat{\Gamma}_1 \left(1 + \frac{1}{n} \sum_{i = 1}^n \one\{M_{\bz}(M_{\bz}(i)) = i\}\right) + \frac{\hat{\Gamma}_2}{n} \sum_{i = 1}^n |M_{\bz}^{-1}(i)| (|M_{\bz}^{-1}(i)| - 1)}{\left(\frac{1}{n^3} \sum_{i=1}^n (n - \check{L}_i) \check{L}_i\right)^2},
\]
where $M_{\bz}^{-1}(i) := \{j: M_{\bz}(j) =i\}$, i.e., the set of $j$ whose nearest neighbour used in the coefficient construction is $i$.
Apparently, it is constructed via estimating each component of $\sigma_n^2$ separately. The following result shows that it is a consistent estimator of $\sigma_n^2$.
\begin{proposition}\label{prop:estvar}
    Suppose that $\by$ is not almost surely a constant, then
    \[
    \hat{\sigma}_n^2 / \sigma_n^2 \asconv 1.
    \]
\end{proposition}

\begin{table}[!t]
\centering
\caption{(a): Averaged percentage of rejections under different simulation setups at nominal levels of $\alpha = 5\%$ and $10\%$. Percentage signs ($\%$) are omitted. (b): Same as (a), but with the oracle $\sigma^2$ in place of the estimate $\hat{\sigma}_n^2$ when building the test. Because the terms $\Gamma_1, \Gamma_2$ and the denominator in~\eqref{eq:contsigma} cannot be expressed in closed form, we compute them numerically using a sample of $n = 10000$ i.i.d. observations.}
\label{tab:both}
\begin{subtable}[t]{\linewidth}
\centering
\caption{Results with $\hat{\sigma}_n^2$}
    \begin{tabular}{ccc|cc|cc|cc}
\hline\hline
\multirow{2}{*}{$\dy$} & \multirow{2}{*}{$\dy$} & \multirow{2}{*}{distribution type} & \multicolumn{2}{c|}{$n=50$} & \multicolumn{2}{c|}{$n=200$} & \multicolumn{2}{c}{$n=1000$} \\
\cline{4-9}
 & & & 5\% & 10\% & 5\% & 10\% & 5\% & 10\% \\
 \hline
2 & 2 & Gaussian & 15.68 & 20.81 & 7.34 & 12.45 & 5.14 & 10.51 \\
2 & 2 & $t_2$ & 16.01 & 21.51 & 7.55 & 12.91 & 5.21 & 10.44 \\
2 & 2 & $t_4$ & 15.64 & 20.87 & 7.64 & 12.90 & 5.49 & 10.41 \\
2 & 5 & Gaussian & 13.04 & 18.10 & 6.45 & 11.58 & 5.47 & 10.05 \\
2 & 5 & $t_2$ & 12.92 & 18.09 & 6.68 & 12.39 & 5.19 & 10.07 \\
2 & 5 & $t_4$ & 12.64 & 17.95 & 6.37 & 11.75 & 4.95 & 10.26 \\
5 & 2 & Gaussian & 19.38 & 24.13 & 10.18 & 15.30 & 6.55 & 11.73 \\
5 & 2 & $t_2$ & 18.46 & 24.01 & 10.21 & 16.78 & 6.33 & 11.67 \\
5 & 2 & $t_4$ & 19.01 & 23.99 & 9.85 & 15.93 & 6.15 & 11.30 \\
5 & 5 & Gaussian & 17.82 & 23.67 & 9.33 & 14.88 & 5.94 & 11.49 \\
5 & 5 & $t_2$ & 17.81 & 22.24 & 9.20 & 15.53 & 6.34 & 11.83 \\
5 & 5 & $t_4$ & 17.89 & 23.06 & 9.21 & 15.63 & 6.26 & 11.32 \\
\hline\hline
\end{tabular}
\label{tab:normal}
\end{subtable}
\hfill
%

% \vspace{0.1cm}
\begin{subtable}[t]{\linewidth}
\centering
\caption{Results with $\sigma^2$}
    \begin{tabular}{ccc|cc|cc|cc}
\hline\hline
\multirow{2}{*}{$\dy$} & \multirow{2}{*}{$\dz$} & \multirow{2}{*}{Distribution type} & \multicolumn{2}{c|}{$n=50$} & \multicolumn{2}{c|}{$n=200$} & \multicolumn{2}{c}{$n=1000$} \\
\cline{4-9}
 & & & 5\% & 10\% & 5\% & 10\% & 5\% & 10\% \\
 \hline
2 & 2 & Gaussian & 5.09 & 10.50 & 4.88 & 9.61 & 4.83 & 10.07 \\

2 & 2 & $t_2$ & 4.77 & 9.84 & 4.52 & 9.29 & 4.50 & 9.34 \\
2 & 2 & $t_4$ & 5.47 & 10.99 & 5.44 & 10.57 & 5.15 & 10.15 \\
2 & 5 & Gaussian & 6.18 & 11.93 & 6.09 & 11.52 & 5.66 & 11.02 \\
2 & 5 & $t_2$ & 5.67 & 11.22 & 5.33 & 10.57 & 5.62 & 10.94 \\
2 & 5 & $t_4$ & 5.91 & 11.56 & 5.40 & 10.31 & 5.11 & 10.39 \\
5 & 2 & Gaussian & 13.37 & 20.49 & 7.36 & 12.85 & 5.60 & 10.78 \\
5 & 2 & $t_2$ & 13.39 & 20.66 & 7.20 & 12.78 & 5.82 & 10.81 \\
5 & 2 & $t_4$ & 13.32 & 20.51 & 7.04 & 12.37 & 5.51 & 10.32 \\
5 & 5 & Gaussian & 12.88 & 20.08 & 7.01 & 12.53 & 4.99 & 9.85 \\
5 & 5 & $t_2$ & 12.85 & 19.43 & 6.91 & 12.01 & 5.30 & 10.32 \\
5 & 5 & $t_4$ & 13.42 & 20.21 & 7.14 & 12.62 & 5.44 & 10.46 \\
\hline\hline
\end{tabular}
\label{tab:normal_oracle}
\end{subtable}
\end{table}

To empirically understand the validity of testing the null hypothesis using $\htaci$ and $\hat{\sigma}_n$, we perform a numerical experiment. We vary $(\dy, \dz) = (2, 2), (5, 2), (2, 5), (5, 5)$. For each choice of $\dy, \dz$, we generate $\{(\by_i, \bz_i)\}_{i = 1}^n$ according to the model $\by = B_{\by} \by'$, $\bz = B_{\bz} \bz'$ where
\begin{itemize}
    \item  $\by' \in \R^{\dy}, \bz' \in \R^{\dz}$ are generated with i.i.d. entries from $\mathcal{N}(0, 1)$, $t_2$ and $t_4$ distributions;
    \item $B_{\by} \in \R^{\dy \times \dy}, B_{\bz} \in \R^{\dz \times \dz}$ are generated by first generating each entry according to i.i.d. $\mathcal{N}(0, 1)$, then normalizing each row such that its $\ell_2$-norm is equal to $1$.
\end{itemize}
For each configuration of $\dy$ and $\dz$, we randomly generate 20 matrices $B_{\by}$ and $B_{\bz}$, respectively. For each combination of $B_{\by}$, $B_{\bz}$, distribution type for $\by'$ and $\bz'$, and sample size $n$ (50, 200, or 1000), we perform 1000 Monte Carlo replications, each with $n$ i.i.d. samples, and then use these 1000 Monte Carlo replications to calculate the test's rejection rate with nominal levels $\alpha = 5\%, 10\%$ under this combination. Since each specific simulation setting (defined by $\dy$, $\dz$, distribution type, and $n$) is associated with 20 distinct $(B_{\by}, B_{\bz})$ pairs, we take the average of the 20 corresponding rejection rates and report them in~\Cref{tab:normal}. To better understand the extent to which miscoverage is due to the convergence of $\htaci$ versus the estimation error from $\hat{\sigma}_n^2$, we report in \Cref{tab:normal_oracle} the average rejection rates obtained using the oracle $\sigma^2$ in~\eqref{eq:contsigma} instead of its estimate $\hat{\sigma}_n^2$.

\Cref{tab:normal_oracle} shows that, when the test is constructed by the oracle $\sigma^2$, for every setting, the rejection rate converges to the nominal level as $n$ becomes sufficiently large, confirming the asymptotic validity of $\htaci$. For smaller $n$ (e.g., $n = 50$), the closeness of $\htaci$'s distribution to Gaussianity can be sensitive to the dimensions, and in particular to $\dy$. Finally, holding both dimension and sample size fixed, the rejection rates are remarkably similar across Gaussian, $t_2$ and $t_4$ distributions, indicating that $\htaci$ is strongly robust to the tail heaviness of the underlying distribution.

\Cref{tab:normal}, on the other hand, reflects a practical scenario where the test is constructed by the estimated $\hat{\sigma}_n^2$. The rejection rates are consistently worse than the results in \Cref{tab:normal_oracle}, especially for small sample sizes ($n = 50, 200$). When the sample size reaches $n=1000$, the test achieves relatively good performance, confirming the asymptotic validity of the proposed test. For $n=50$, the rejection rates are severely inflated across all settings, and in particular the settings with large $\dy$. Finally, the test maintains its robustness to the distribution type of $\by', \bz'$.

\section{Algorithmic implementation and additional discussions}\label{sec:algorithm}

In this section, we discuss the computation of $\htaci, \htac$ and $\hat{\sigma}_n^2$. Taking $\htaci$ as an example, the main bottleneck is to calculate $\sum_{i = 1}^n \tilde{R}(\by_i \wedge \by_{M_{\bz}(i)})$ and $\sum_{i = 1}^n \check{L}_i^2$. By simply using pairwise comparisons, they can be easily computed in $O(n^2)$, which means that $\htaci, \htac$ and $\hat{\sigma}_n$ can be computed with time complexity no greater than $O(n^2)$. We now develop a faster algorithm which can reduce the time complexity to $O(n (\log n)^{\dy})$; this amounts to solving the following computational problem: 
\begin{itemize}
    \item Suppose we are given two data sets $\{\bsa_i\}_{i=1}^{n_a}$ and $\{\bsb_i\}_{i=1}^{n_b}$, where $\bsa_i, \bsb_i \in \R^d$, how to calculate $c_1 := \sum_{i=1}^{n_a} \one\{\bsa_i \le \bsb_1\}, \ldots, c_{n_b} := \sum_{i=1}^{n_a} \one\{\bsa_i \le \bsb_{n_b}\}$.
\end{itemize}

\begin{algorithm}[!t]
	\DontPrintSemicolon
	\caption{\label{alg:2drank} $2$-dimensional multivariate rank construction}
	\KwIn{Two sets of $2$-dimensional vectors $\{\bsa\}_{i=1}^{n_a}, \{\bsb\}_{i=1}^{n_b}$.}
    \KwOut{Multivariate ranks $c_1 := \sum_{i=1}^{n_a} \one\{\bsa_i \le \bsb_1\}, \ldots, c_{n_b} := \sum_{i=1}^{n_a} \one\{\bsa_i \le \bsb_{n_b}\}$.}
	
	Sort $\{\bsa_i\}_{i=1}^{n_a}$ according to its first entry, such that $\bsa_{1, 1} \le \cdots \le \bsa_{n_a, 1}$\;
    
    For each $\bsb_i$, associate it with a $k_i$, which is the largest $i'$ satisfying $\bsa_{i', 1} \le \bsb_{i, 1}$; also, set $c_i$ as $1$ if $\bsa_{k_i} \le \bsb_i$, and $0$ otherwise. If $\bsb_{i,1}$ is smaller than all $\bsa_{i', 1}$, set $k_i = 0$ and $c_i = 0$\;

	\For{$j = 1, \ldots, J := \lceil \log_2 n_a \rceil$}{
    Set $J'_j := \lceil n_a / 2^{j - 1}\rceil$; split second coordinates of $\bsa_i$ into sequences $\bs{e}_1^j, \ldots, \bs{e}_{J'_j}^j$ and sort each $\bs{e}_{j'}^j$; then $\bs{e}_{j'}^j$ becomes a sorted sequence containing the second coordinates of vectors 
    \[
    \bsa_{2^{j - 1} \cdot (j' - 1) + 1}, \ldots, \bsa_{\min\{2^{j - 1} \cdot j', \; n_a\}}\footnotemark\;
    \]
        % {\footnotesize (To construct $\bs{e}_{j'}$'s, one can utilize $\bs{e}_{j'}$'s in the $(j - 1)$-th round for faster construction, this is a standard operation in merge sort algorithm)}
        % \footnotemark\;

        \For{$i = 1, \ldots, n_b$}{
        If $k_i \le 1$, skip. Otherwise, if there exists an even integer $\ell$ such that $$2^{j - 1} \cdot (\ell - 1) < k_i \le 2^{j - 1} \cdot \ell,$$ update $c_i$ by adding $c_i' := \sum_{x \in \bs{e}_{\ell - 1}^j}\one\{x \le \bsb_{i, 2}\}$, which is obtained by computing the rank of $\bsb_{i, 2}$ in the sorted sequence $\bs{e}_{\ell - 1}^j$; otherwise, leave $c_i$ unchanged\;
        }
	}
    \vspace{-0.05cm}
\end{algorithm}

\footnotetext{Note that when $j > 1$, the sequences $\bs{e}_1^j, \ldots, \bs{e}_{J'_j}^j$ can be constructed by applying a merge sort algorithm to the sequences from the $(j - 1)$-th iteration.}

When $d = 1$, this problem can be solved by a standard ranking algorithm. In the higher dimensional regime, calculating $c_i$'s becomes more challenging. In this paper, we propose Algorithms~\ref{alg:2drank} and~\ref{alg:drank} to solve this problem under the regimes $d = 2$ and $d \ge 3$, respectively. Algorithm~\ref{alg:2drank} begins by sorting the set $\{\bsa_i\}_{i=1}^{n_a}$ by its first coordinate (Step 1), so that $\bsa_{1, 1} \le \cdots \le \bsa_{n_a, 1}$. Next, it computes the rank $k_i$ of each scalar $\bsb_{i, 1}$ within the sorted list of first coordinates $\bsa_{i', 1}$ (Step 2). Given these ranks $k_i$, a naive approach to compute $c_i$ would be to sort the set $\{\bsa_{i', 2}, i' \le k_i\}$ and then find the rank of $\bsb_{i, 2}$ within them. 
This naive approach, however, has a time complexity of $O(n^2 \log n)$, which is even worse than a simple pairwise comparison. To improve efficiency, in Steps 3-8, we propose a new method that simultaneously computes all $c_i$ values by embedding this computation within a \emph{merge sort}~\citep{knuth1997art}
% \footnote{Here \emph{merge sort} refers to the so-called merge sort algorithm.} 
of the second coordinates of $\bsa_i$. We prove in~\Cref{thm:complexity} that this algorithm achieves a time complexity of $O(n (\log n)^2)$. Algorithm~\ref{alg:drank} is a divide-and-conquer procedure that computes the values $c_i$ by recursively reducing the dimension of the dataset until it becomes two-dimensional; and handles the final step by Algorithm~\ref{alg:2drank}. The following proposition shows the complexity and consistency of the two algorithms.

\begin{algorithm}[!t]
    \DontPrintSemicolon
	\caption{\label{alg:drank} $d$-dimensional cumulative rank construction}
    \KwIn{Two sets of $d$-dimensional vectors $\{\bsa_i\}_{i=1}^{n_a}, \{\bsb_i\}_{i=1}^{n_b}$ for some $d \ge 3$.}
    \KwOut{Multivariate ranks $c_1 := \sum_{i=1}^{n_a} \one\{\bsa_i \le \bsb_1\}, \ldots, c_{n_b} := \sum_{i=1}^{n_a} \one\{\bsa_i \le \bsb_{n_b}\}$.}
    
    Sort $\{\bsa_i\}_{i=1}^{n_a}$ according to its first entry, such that $\bsa_{1, 1} \le \ldots, \bsa_{n_a, 1}$\;
    
    For each $\bsb_i$, associate it with a $k_i$, which is the largest $i'$ satisfying $\bsa_{i', 1} \le \bsb_{i, 1}$; also, set $c_i$ as $1$ if $\bsa_{k_i} \le \bsb_i$, and $0$ otherwise. If $\bsb_{i,1}$ is smaller than all $\bsa_{i', 1}$, set $k_i = 0$ and $c_i = 0$\;

    \For{$j = 1, \ldots, J := \lceil \log_2 n_a \rceil$}{
		Set $J_j' := \lceil n_a / 2^{j - 1}\rceil$; split subvectors $\{\bsa_{i, \{2, \ldots, d\}}\}_{i = 1}^{n_a}$ into sets of vectors $\mathcal{A}_1^j, \ldots, \mathcal{A}_{J_j'}^j$; here 
        \[
        \mathcal{A}_{j'}^j := \left\{\bsa_{2^{j - 1} \cdot (j' - 1) + 1, \{2, \ldots, d\}}, \ldots, \bsa_{\min\{2^{j - 1} \cdot j',\; n_a\},\; \{2, \ldots, d\}}\right\}
        \]

        For each $\mathcal{A}_{j'}^j$, associate it with two empty sets $\mathcal{B}_{j'}^j$, $\cC_{j'}^j$\;
        % {\footnotesize (To construct $\bs{e}_{j'}$'s, one can utilize $\bs{e}_{j'}$'s in the $(j - 1)$-th round for faster construction, this is a standard operation in merge sort algorithm)}
        % \footnotemark\;

        \For{$i = 1, \ldots, n_b$}{
        If $k_i \le 1$, skip. Otherwise, if there exists some even number $\ell$ such that $$2^{j - 1} \cdot (\ell - 1) < k_i \le 2^{j - 1} \cdot \ell,$$ add $\bsb_{i, \{2, \ldots, d\}}$ to $\cB_{\ell - 1}^j$ and add $i$ to $\cC_{\ell - 1}^j$\;
        }
        For each pair $(\mathcal{A}_{j'}^j, \mathcal{B}_{j'}^j)$, if $d = 3$, apply Algorithm~\ref{alg:2drank} to this pair; otherwise, apply Algorithm~\ref{alg:drank}. This creates a set of ranks $\{c_i': i \in \cC_{j'}^j\}$; for each $i \in \cC_{j'}^j$, add $c_i$ by $c_i'$\;
    }
\end{algorithm}

\begin{theorem}\label{thm:complexity}
Let $n = \max\{n_a, n_b\}$. Algorithm~\ref{alg:2drank} has time complexity $O(n (\log n)^2)$; and Algorithm~\ref{alg:drank} has time complexity $O(n (\log n)^d)$. Moreover, the two algorithms are correct, namely the values of $c_i$ obtained by these algorithms are equal to $\sum_{i' = 1}^{n_a} \one\{\bsa_{i'} \le \bsb_i\}$.
\end{theorem}

Recall that the construction of nearest neighbor functions has time complexity $O(n \log n)$~\citep{friedman1977algorithm}. Then \Cref{thm:complexity} means that $\htaci, \htac$ and $\hat{\sigma}_n$ can be computed with complexity $O(n (\log n)^{\dy})$. When $\dy = 1$, it has been shown in~\citet{azadkia2021simple} that the complexity of $\htaci$ and $\htac$ are of order $O(n \log n)$, which is consistent with our results. 
% Analogously,~\Cref{thm:complexity} also shows that $\hat{\sigma}_n$ can be computed with time complexity $O(n (\log n)^{\dy})$.

We next discuss some properties of $\taci$. As we have shown in~\Cref{sec:multiac}, in the multivariate case $\taci$ still satisfies~\ref{p1}--\ref{p3'}. In the following two propositions, we show that $\taci$ has some additional monotonicity properties that were not discussed in previous literature. It should be noted that after the public release of this manuscript, we discovered that~\Cref{prop:additive} was also proposed independently in a contemporaneous work~\citep[Theorem~1.8]{ansari2025ordering}.
\begin{proposition}[Monotonicity under mixture models]\label{prop:mixture}
    Consider two distributions $\pr_0$ and $\pr_1$ of $(\by, \bz)$ such that the marginal distribution of $\by$ and $\bz$ are the same for $\pr_0$ and $\pr_1$; and moreover, $\by$ is almost surely a function of $\bz$ under $\pr_1$, but is independent of $\bz$ under $\pr_0$. Then if $(\by, \bz)$ follows the distribution $(1 - \eta) \pr_0 + \eta \pr_1$, for some $\eta \in [0, 1]$, we have $\taci = \eta^2$.
\end{proposition}

\begin{proposition}[Monotonicity under additive models]\label{prop:additive}
    Consider $Y = \eta h(\bz) + \varepsilon$, where $\varepsilon, h(\bz) \in \R$ are independent, continuous random variables, and $\eta \ge 0$. We have that $\taci$ is monotonically increasing with $\eta$. 
\end{proposition}

Noteworthy, \Cref{prop:additive} shows that, when $\dy = 1$, $\taci$ is monotone with an additive model. When $\dy > 1$, the analysis becomes significantly more challenging. In the univariate case, the denominator is a universal constant, meaning we only need to analyze how the numerator changes with $\beta$. When $\dy \ge 2$, both the numerator and denominator of $\taci$ depend on $\beta$, making $\taci$ very difficult to analyze theoretically. To address this theoretical gap, we demonstrate the monotonicity of $\taci$ through a numerical analysis. We conduct an experiment following similar simulation setup as the one in~\Cref{sec:normal}. The result indicates that $\taci$ is still monotonically increasing with $\eta$. For more details, we refer the readers to the Supplementary Material.

% \textcolor{red}{To do list for Wenjie:
% \begin{itemize}
%     % \item Provide a new proof of $\tilde{\Gamma} > 0$
%     \item Show that $\sigma_n^2$ is not convergent under Cantor.
% \end{itemize}
% }

\section{Oracle analysis}\label{sec:oracle}

This section is devoted to studying the properties of $\tac, \taci, \Gamma_1$ and $\Gamma_2$, all of which are defined under the oracle scenario in which the underlying distribution is known. More specifically, our goal is to prove~\Cref{thm:multiaccoef,thm:multiaccondcoef}, and the following proposition.

\begin{proposition}\label{thm:gamma}
    Consider the $\Gamma_1, \Gamma_2$ defined in~\Cref{thm:tacnormal}. We have that $\Gamma_1, \Gamma_2 \ge 0$. Moreover, if $\by$ is not almost surely a constant, then $\Gamma_1 > 0$.
\end{proposition}

The proofs of~\Cref{thm:multiaccoef,thm:multiaccondcoef} are direct consequences of~\Cref{lem:mea,lem:iff0,lem:iff1} in~\Cref{sec:oraclelem}. The proof of~\Cref{thm:gamma} is in~\Cref{sec:pfgamma}.

\noindent{\bf Notations.} Here we provide some notations that will be used \emph{only} in this section. Given a random vector $\bs{V} \in \R^d$, we write $\bs{V}_{d'}$ as the $d'$-th coordinate of $\bs{V}$. Given a set of indices $\cS \subseteq [d]$, we write $\bs{V}_{\cS}$ as the random subvector taking only the coordinates in $\cS$. Analogously, given a fixed $\bs{v} \in \R^d$, we write $\bs{v}_{d'}$ and $\bs{v}_{\cS}$ as its $d'$-th coordindate or its subvector. Write $\mu_{\by_j}$ as the marginal distribution measure of random variable $\by_j$.
We write $\mu_{\by\mid\bx,\bz}$ as the regular conditional distribution of $\by$ given $(\bx,\bz)$. We write $\mu_{\by\mid\bx}$ and $\mu_{\by\mid\bz}$ as the corresponding marginals of $\mu_{\by\mid\bx,\bz}$. Given an index $j \in [\dy]$, we write $\mu_{\by_j\mid\bx,\bz}$, $\mu_{\by_j\mid\bx}$ and $\mu_{\by_j\mid\bz}$ as the $j$-th marginal of $\mu_{\by\mid\bx,\bz}$, $\mu_{\by\mid\bx}$ and $\mu_{\by\mid\bz}$, respectively. For the existence and the properties of regular conditional distributions, see the beginning of~\citet[Supplementary Material]{chatterjee2021new}. Note that the regular conditional distribution $\mu_{\by\mid\bx,\bz}$ is a measure only almost surely, not for every realization. This, however, does not affect the correctness of the proofs in this section, as all the arguments therein hold up to a zero-measure set. Finally, provided $\mu_{\by\mid\bx,\bz}$, we define $\bs{W}:=(\bx,\bz)$, and then define $G(\bsy):=\pr(\by \geq \bsy)$, $G_{\bz}(\bsy) := \pr(\by \geq \bsy \mid \bz)$, $G_{\bx}(\bsy) := \pr(\by \geq \bsy \mid \bx)$ and $G_{\bs{W}}(\bsy) := \pr(\by \geq \bsy \mid \bx,\bz)$.

% write $\mu_{\by \mid \bz}, \mu_{\by_j \mid \bz}$ as the regular conditional distribution measure of $\by$ and $\by_j$ conditional on $\bz$, respectively. Finally, we define $\bs{W} := (\bx, \bz)$; define $G(\bsy) := \pr(\by \ge \bsy)$, $G_{\bz}(\bsy) := \pr(\by \ge \bsy \mid \bz)$ under $\mu_{\by \mid \bz}$'s version and define $G_{\bs{W}}(\bsy)$ analogously.

\subsection{Preliminary lemmas}\label{sec:oraclelem}

% \begin{lemma}\label{lem:closedset}
% Let $\bs{V}_1,\ldots,\bs{V}_m$ be any random vectors defined on the same probability space as $\by$. Let $B \subseteq \R^m$ be any closed set containing $\bs{0}$. Suppose that there exists a set $A \subseteq \R^{\dy}$ with $\tilde{\mu}_{\by}(A)=1$ such that for all $\bsy \in A$,
% \[
% \pr((G_{\bs{V}_1}(\bsy), \ldots, G_{\bs{V}_m}(\bsy)) \in B) = 1.
% \]
% Then, the above relation holds for all $\bsy \in \R^{\dy}$.

% Here, we write $\mu_{\by\mid\bs{V}_1,\ldots,\bs{V}_m}$ as the regular conditional distribution of $\by$ given $\bs{V}_1,\ldots,\bs{V}_m$, write $\mu_{\by\mid\bs{V}_i}$ as its corresponding $i$-th marginal, and define $G_{\bs{V}_i}(\bsy) := \pr(\by \geq \bsy \mid \bs{V}_i)$ provided $\mu_{\by\mid\bs{V}_i}$.
% \end{lemma}

We start by proving~\Cref{lem:closedset}, which will be repeatedly used in this section.

\begin{proof}[Proof of~\Cref{lem:closedset}]
We start by introducing some notations. We write $\mu_{\by\mid\bs{V}_1,\ldots,\bs{V}_m}$ as the regular conditional distribution of $\by$ given $\bs{V}_1,\ldots,\bs{V}_m$, write $\mu_{\by\mid\bs{V}_i}$ as its corresponding $i$-th marginal, and define $G_{\bs{V}_i}(\bsy) := \pr(\by \geq \bsy \mid \bs{V}_i)$, with the understanding that this conditional probability is computed using the measure $\mu_{\by\mid\bs{V}_i}$. Armed with these notations, then the problem becomes proving that if for all $\bsy \in A$,
\[
\pr((G_{\bs{V}_1}(\bsy), \ldots, G_{\bs{V}_m}(\bsy)) \in B) = 1,
\]
then the above relation holds for all $\bsy \in \R^{\dy}$.

Consider arbitrary $\bsy \in \R^{\dy}$. If $\tilde{\mu}_{\by}(\{\bsy\}) > 0$, then obviously $\bsy \in A$, so that this relation holds. Otherwise, we consider two cases.

In the first case, for any $\bsy' > \bsy$, i.e., $\bsy'$ is entry-wise strictly larger than $\bsy$, $\tilde{\mu}_{\by}(\{\tilde{\bsy}: \bsy \leq \tilde{\bsy} < \bsy'\})>0$. Define $\cS := \{j \in [\dy]: \mu_{\by_j}(\{\bsy_j\})>0\}$ and $H := \{\bsy': \bsy_{\cS}'=\bsy_{\cS}, \bsy_{\cS^c}' > \bsy_{\cS^c}\}$. Then, for each $\bsy' \in H$, by defining $A_{\bsy,\bsy'} := \{\tilde{\bsy}: \tilde{\bsy}_{\cS} = \bsy_{\cS}, \bsy_{\cS^c} < \tilde{\bsy}_{\cS^c} < \bsy_{\cS^c}'\}$, we have $\tilde{\mu}_{\by}(A_{\bsy,\bsy'})>0$. This means that $A_{\bsy,\bsy'}$ must intersect with $A$. Now, we choose a sequence $\bsy_n' \in H$ approaching $\bsy$, and pick $\bs{r}_n \in A_{\bsy,\bsy_n'} \cap A$ for each $n$. This gives a sequence $\bs{r}_n \in H \cap A$ such that $\bs{r}_n \to \bsy$. Then, for each $i \in [m]$, we must have almost surely,
\[
0
\leq G_{\bs{V}_i}(\bsy) - G_{\bs{V}_i}(\bs{r}_n)
\leq \sum_{j \in \cS^c} \mu_{\by_j\mid\bs{V}_i}([\bsy_j,\bs{r}_{n,j}))
\to \sum_{j \in \cS^c} \mu_{\by_j\mid\bs{V}_i}(\{\bsy_j\})
= 0,
\]
where the last step is because for each $j \in \cS^c$ we have $\E \mu_{\by_j\mid\bs{V}_i}(\{\bsy_j\}) = \mu_{\by_j}(\{\bsy_j\}) = 0$, which gives $\mu_{\by_j\mid\bs{V}_i}(\{\bsy_j\}) = 0$ almost surely. Now, the formula above gives for each $i \in [m]$, $G_{\bs{V}_i}(\bs{r}_n)$ converges to $G_{\bs{V}_i}(\bsy)$ a.s. Since we also have $\bs{r}_n \in A$, this means for each $\bs{r}_n$, almost surely $(G_{\bs{V}_1}(\bs{r}_n), \ldots, G_{\bs{V}_m}(\bs{r}_n)) \in B$. By the closedness of $B$, we obtain that $(G_{\bs{V}_1}(\bsy), \ldots, G_{\bs{V}_m}(\bsy)) \in B$ almost surely.

In the second case, given that all the $\by_j$'s are independent under the measure $\tilde{\mu}_{\by}$, we must have that there exists at least one $j$ such that for $\check{\bsy}_j > \bsy_j$ sufficiently close, $\mu_{\by_j}([\bsy_j, \check{\bsy}_j)) = 0$. We now redefine $\cS$ as the collection of such indices. Then for each $j \in \mathcal{S}$, we seek for the largest $\check{\bsy}_j$ such that $\mu_{\by_j}([\bsy_j, \check{\bsy}_j)) = 0$. If one of such $\check{\bsy}_j = \infty$, then we can prove that $G(\bsy) = 0$, such that for each $i \in [m]$, $\E G_{\bs{V}_i}(\bsy) = G(\bsy) = 0$. This further means that $G_{\bs{V}_i}(\bsy) = 0$ almost surely, implying that $(G_{\bs{V}_1}(\bsy), \ldots, G_{\bs{V}_m}(\bsy))= \bs{0} \in B$ almost surely. 

Otherwise, let $\check{\bsy}$ be such that for any $j \in \mathcal{S}^c$, $\check{\bsy}_j$ is equal to $\bsy_j$; for any $j \in \mathcal{S}$, $\check{\bsy}_j$ is equal to the largest $q$ such that $\mu_{\by_j}([\bsy_j, q)) = 0$. Then $\check{\bsy}$ satisfies the requirement of $\bsy$ in Case 1 and we know from there that $(G_{\bs{V}_1}(\check{\bsy}), \ldots, G_{\bs{V}_m}(\check{\bsy})) \in B$ almost surely. Now it remains to prove that almost surely, for each $i \in [m]$, $G_{\bs{V}_i}(\bsy) = G_{\bs{V}_i}(\check{\bsy})$. To prove this, for each $i \in [m]$ and $j \in \cS$, we have $\E \mu_{\by_j\mid\bs{V}_i}([\bsy_j,\check{\bsy}_j)) = \mu_{\by_j}([\bsy_j,\check{\bsy}_j)) = 0$, so almost surely $\mu_{\by_j\mid\bs{V}_i}([\bsy_j,\check{\bsy}_j)) = 0$, which gives almost surely,
\[
0
\leq G_{\bs{V}_i}(\bsy) - G_{\bs{V}_i}(\check{\bsy})
\leq \sum_{j \in \cS} \mu_{\by_j\mid\bs{V}_i}([\bsy_j,\check{\bsy}_j))
= 0.
\]

In light of our control of all cases, we prove the desired result.
\end{proof}

\begin{lemma}\label{lem:iff0}
	We have that 
    \begin{itemize}
        \item[(i)] $\int \var(\pr(\by \ge \bsy \mid \bz)) d \tilde{\mu}_{\by} (\bsy) = 0$ if and only if $\by$ and $\bz$ are independent;
        \item[(ii)] $\int \E(\var(\pr(\by \ge \bsy \mid \bx, \bz) \mid \bx)) d \tilde{\mu}_{\by} (\bsy) = 0$ if and only if $\by$ and $\bz$ are conditionally independent given $\bx$.
    \end{itemize}
\end{lemma}

\begin{proof}

Without loss of generality, we just need to prove (ii), and (i) follows immediately from taking $\bx \equiv \bs{0}$. The ``if'' is easy, since conditional independence implies that for any $\bsy \in \R^{\dy}$, almost surely, $\pr(\by \ge \bsy \mid \bx, \bz) = \pr(\by \ge \bsy \mid \bx)$. We focus on the ``only if''. Write $\bs{W} := (\bx, \bz)$. Then we directly have that
    \[
    \int \E[(G_{\bs{W}}(\bsy) - G_{\bx}(\bsy))^2] d \tilde{\mu}_{\by}(\bsy) = 0,
    \]
    so that there exists a $A \subseteq \R^{\dy}$ with $\tilde{\mu}_{\by}(A) = 1$ such that for any $\bsy \in A$,
    \[
    \pr(G_{\bs{W}}(\bsy) - G_{\bx}(\bsy) = 0) = 1.
    \]
    Then we can apply~\Cref{lem:closedset} with $m = 2$, $\bs{V}_1 = \bs{W}, \bs{V}_2 = \bx$, and $B = \{(p_1, p_2) \in \R^2: p_1 - p_2 = 0\}$ to get that the above relation holds for all $\bsy \in \R^{\dy}$.
    
    In light of the above, and write $\Q$ as a measure satisfying $\by \independent \bz \mid \bx$ and sharing the same distribution as $\pr$ on $\by$ given $\bx$ and $\bz$ given $\bx$, we have that for each fixed $\bsx \in \R^{\dx}, \bsy \in \R^{\dy}$ and $\bsz \in \R^{\dz}$, 
    \begin{align*}
    & \pr(\by \ge \bsy, \bx \ge \bsx, \bz \ge \bsz) = \int_{\bsx' \ge \bsx, \bsz' \ge \bsz} G_{\bs{W} = (\bsx', \bsz')}(\bsy) d \mu_{\bx, \bz}((\bsx', \bsz'))\\
    & = \int_{\bsx' \ge \bsx, \bsz' \ge \bsz} G_{\bx = \bsx'}(\bsy) d \mu_{\bx, \bz}((\bsx', \bsz')) = \int_{\bsx' \ge \bsx} G_{\bx = \bsx'}(\bsy) \int_{\bsz' \ge \bsz} 1 \cdot d \mu_{\bz \mid \bx = \bsx'}(\bsz') d \mu_{\bx}(\bsx') \\
    & = \int_{\bsx' \ge \bsx} \pr(\by \ge \bsy \mid \bx = \bsx') \pr(\bz \ge \bsz \mid \bx = \bsx') d \mu_{\bx}(\bsx') = \Q(\by \ge \bsy, \bx \ge \bsx, \bz \ge \bsz),
    \end{align*}
    thereby proving the conditional independence.
\end{proof}

% \begin{lemma}\label{lem:supp}
% 	Consider $d = 1$ case in~\Cref{lem:suppdef}, then $C_\mu^c$ can be expressed as a countable union of disjoint open intervals, i.e., there exists a sequence of intervals $\{(a_k, b_k)\}_{k=1}^\infty$ such that
%     \[
%     C_\mu^c = \cup_{k=1}^\infty (a_k, b_k).
%     \]
% \end{lemma}

% \begin{proof}
% As shown in~\Cref{lem:suppdef}, $C_\mu$ must be a closed set. Then $C_\mu^c$ must be an open set and the result follows directly from~\citet[Theorem 1.3]{stein2009real}.
% \end{proof}

\begin{lemma}\label{lem:const}
If for any $\bsy \in \R^{\dy}$, $\pr(G_{\bz}(\bsy) \in \{0, 1\}) = 1$, then $\by$ is almost surely a function of $\bz$.
\end{lemma}

\begin{proof}
Using the countability of $\mathbb{Q}^{\dy}$, we easily have from the condition that
\[
\pr(\forall \bsy \in \mathbb{Q}^{\dy}, G_{\bz}(\bsy) \in \{0, 1\}) = 1.
\]
This, together with the standard definition of regular conditional probability, means that there exists a $B \subseteq \R^{\dz}$ with $\mu_{\bz}(B) = 1$ such that for any $\bsz \in B$ and $\bsy \in \Q^{\dy}$, $G_{\bz = \bsz}(\bsy) \in \{0, 1\}$; moreover, $\mu_{\by \mid \bz = \bsz}$, the measure upon which $G_{\bz = \bsz}(\bsy)$ is constructed, is a valid distribution measure. Now for each $j \in [\dy]$, we construct a sequence $\bsy_1', \bsy_2', \ldots \in \mathbb{Q}^{\dy}$ such that $\bsy_{n, j}' = q \in \mathbb{Q}$ and for any $k \neq j$, $\bsy_{n, k}' \downarrow - \infty$. Then apparently for any $\bsz \in B$,
\[
\pr(\by_j \ge q \mid \bz = \bsz) = \lim_{n \to \infty} G_{\bz = \bsz}(\bsy_n') \in \{0, 1\}.
\]
Now for any $b \in \R$, we construct $q_1, q_2, \ldots \in \mathbb{Q}$ such that $q_n \uparrow b$. Then for any $\bsz \in B$,
\[
\pr(\by_j \ge b \mid \bz = \bsz) = \lim_{n \to \infty} \pr(\by_j \ge q_n \mid \bz = \bsz) \in \{0, 1\}.
\]

In light of this, it follows from the same analysis as the proof of~\citet[Theorem~9.2]{azadkia2021simple} that there exists a function $f_j(\cdot)$ such that $\by_j = f_j(\bz)$ almost surely. This proves the desired result.
\end{proof}

\begin{lemma}\label{lem:iff1}
We have that
\begin{itemize}
    \item[(i)] $\int \var(\pr(\by \ge \bsy \mid \bz)) d \tilde{\mu}_{\by} (\bsy) \le \int \var(\one\{\by \ge \bsy \}) d \tilde{\mu}_{\by} (\bsy)$, where the equality holds if and only if $\by$ is almost surely a function of $\bz$;
    \item[(ii)] $\int \E(\var(\pr(\by \ge \bsy \mid \bx, \bz) \mid \bx)) d \tilde{\mu}_{\by} (\bsy) \le \int \E(\var(\one\{\by \ge \bsy\} \mid \bx)) d \tilde{\mu}_{\by} (\bsy)$, where the equality holds if and only if $\by$ is almost surely a function of $(\bx, \bz)$.
\end{itemize}
\end{lemma}

\begin{proof}
We start by proving (i). The first result directly follows from the total variance decomposition
\begin{equation}\label{eq:totvar}
\var(\one\{\by \ge \bsy \}) = \E[\var(\one\{\by \ge \bsy \} \mid \bz)] + \var(\pr(\by \ge \bsy \mid \bz)).
\end{equation}
For the second result, ``if'' is easy, since by given $\bz$, $\one\{\by \ge \bsy\}$ becomes deterministic, which gives
\[
\var(\pr(\by \ge \bsy \mid \bz)) = \var(\E[\one\{\by \ge \bsy\} \mid \bz]) = \var(\one\{\by \ge \bsy\}).
\]
We now consider the ``only if''. From the total variance decomposition, the equality holds directly implies that
\begin{equation}\label{eq:remain}
\int \E[\var(\one\{\by \ge \bsy\} \mid \bz)] d \tilde{\mu}_{\by} (\bsy) = \int \E[G_{\bz}(\bsy) (1 - G_{\bz}(\bsy))] d \tilde{\mu}_{\by} (\bsy) = 0.
\end{equation}
This further means there exists a $A \subseteq \R^{\dy}$ with $\tilde{\mu}_{\by}(A) = 1$ such that for any fixed $\bsy \in A$, almost surely, $G_{\bz}(\bsy) \in \{0, 1\}$. Applying~\Cref{lem:closedset} with $m = 1, \bs{V}_1 = \bz$ and $B = \{0, 1\}$, we have for all $\bsy \in \R^{\dy}$, almost surely, $G_{\bz}(\bsy) \in \{0, 1\}$. This, together with~\Cref{lem:const}, implies that almost surely, $\by$ is deterministic provided $\bz$, thereby proving the desired result.
	
	We now consider (ii). Again, ``if'' is straightforward. For ``only if'', following again the total variance decomposition, we have that the equality directly implies that
    \[
    \int \E(\E[\var(\one\{\by \ge \bsy\} \mid \bx, \bz) \mid \bx]) d \tilde{\mu}_{\by}(\bsy) = \int \E(\var(\one\{\by \ge \bsy\} \mid \bx, \bz)) d \tilde{\mu}_{\by}(\bsy) = 0,
    \]
    and the proof follows from (i), except that we have replaced $\bz$ by $\bs{W} := (\bx, \bz)$.
\end{proof}

\begin{lemma}\label{lem:mea}
We have that
\begin{itemize}
    \item[(i)] 
If $\by$ is not almost surely a constant, then $\int \var(\one\{\by \ge \bsy\}) d \tilde{\mu}_{\by}(\bsy) > 0$;
\item[(ii)] If $\by$ is not almost surely a function of $\bx$, then $\int \E[\var(\one\{\by \ge \bsy\} \mid \bx)] d \tilde{\mu}_{\by}(\bsy) > 0$.
\end{itemize}
\end{lemma}

\begin{proof}
    Without loss of generality we just need to prove (ii), and (i) follows by taking $\bx$ as a fixed constant. Using total variance decomposition (see e.g.~\eqref{eq:totvar}), it is straightforward that $\int \E[\var(\one\{\by \ge \bsy\} \mid \bx)] d \tilde{\mu}_{\by}(\bsy) = 0$ if and only if $\int \var(\pr(\by \ge \bsy \mid \bx)) d \tilde{\mu}_{\by} (\bsy) = \int \var(\one\{\by \ge \bsy \}) d \tilde{\mu}_{\by} (\bsy)$. Therefore the desired result directly follows from Lemma~\ref{lem:iff1}(i).
\end{proof}

\begin{lemma}[{\citet[Chapter~2, Section~2]{par05}, see also \citet[Section~7.4]{cohn13}}]\label{lem:suppdef}
    Consider a random vector $\bs{U} \in \R^d$ with measure $\mu$. Then there exists a unique closed set $C_\mu$ satisfying (i) $\mu(C_\mu) = 1$; (ii) if $D$ is any closed set such that $\mu(D) = 1$, then $C_\mu \subseteq D$; (iii) $\bs{u} \in C_\mu$ if and only if for each open set $U$ containing $\bs{u}$, $\mu(U) > 0$.
    We call $C_\mu$ as the support of $\mu$.
\end{lemma}

\begin{lemma}\label{lem:gamma}
    Consider $\by', \bar{\by}, \bar{\by}'$ as i.i.d. replications of $\by$, then
    \[
    \tilde{F}(\by \wedge \bar{\by}) + \tilde{F}(\by' \wedge \bar{\by}') = \tilde{F}(\by \wedge \by') + \tilde{F}(\bar{\by} \wedge \bar{\by}')
    \]
    almost surely implies that $\by$ is almost surely a constant vector.
\end{lemma}

\begin{proof}
    When $\dy =1$, the argument has already been discussed in ~\citet[Lemma~10.4]{chatterjee2021new}. We now consider this on $\dy \ge 2$. For simplicity we write $\dy, \mu_{\by}, \tilde{\mu}_{\by}, \mu_{\by_i}$ as $d, \mu, \tilde{\mu}, \mu_i$, and write the CDF of $\mu_i$ as $F_i$. Also, we define $R(\bsy, \varepsilon)$ as the set of points whose coordinate-wise distance to $\bsy$ is smaller than $\varepsilon$.
    
    Define $H_0 := \R^d$, and for each $1 \le d' \le d$,
\[
H_{d'} := \{\bsy \in \R^d: \exists \text{ distinct } i_1 \ldots, i_{d'} \in [d] \text{ s.t. } \mu_{i_1}(\{\bsy_{i_1}\})>0,\ldots,\mu_{i_{d'}}(\{\bsy_{i_{d'}}\})>0\}.
\]
Note that following this definition, $H_d$ is a countable set, and
\[
H_d \subseteq H_{d-1} \subseteq \cdots \subseteq H_1 \subseteq H_0 = \R^d.
\]
We now want to prove our main result in three steps. First, we show that $\mu(H_1)=1$. Second, we apply an analogous argument to show that for each $2 \leq d' \leq d$, $\mu(H_{d'-1})=1$ implies $\mu(H_{d'})=1$. Finally, we show that $\by$ is almost surely a constant vector.

First, we show that $\mu(H_1)=1$. From the definition, $H_1$ is a (finite or infinite) countable union of $(d-1)$-dimensional hyperspaces. Suppose in contradiction that $\mu(H_1)<1$, then, none of the $\by_i$ is discrete random variable (they may have discrete fractions). Moreover, by setting $\bar{F}_1(x) := \mu(\{\bsy \in \R^d: \bsy_1 \le x\} \setminus H_1)$, we have that $\bar{F}_1(x)$ is continuous and non-decreasing, and $\lim_{x \to -\infty}\bar{F}_1(x) = 0$, $\lim_{x \to \infty}\bar{F}_1(x) = \mu(H_0 \setminus H_1) > 0$. Using the continuity of $\bar{F}_1(x)$, we have that there must exist $a < b$, such that $0 < \bar{F}_1(a) < \bar{F}_1(b) < \lim_{x \to \infty}\bar{F}_1(x)$.

Define $S := \{\bsy: \exists i\;\mathrm{s.t.}\; F_i(\bsy_i) = 0\}$; then apparently $\mu(S) = 0$. 
We now consider two sets $A :=\{\bsy \in \R^d: \bsy_1 < a\} \setminus (H_1 \cup S)$, $B :=\{\bsy \in \R^d: \bsy_1 > b\} \setminus (H_1 \cup S)$. Apparently 
\[
\mu(A) = \bar{F}_1(a) > 0 \quad\&\quad \mu(B) = \lim_{x \to \infty} \bar{F}_1(x) - \bar{F}_1(b) > 0,
\]
i.e., both $A$ and $B$ have non-zero measures. In light of this and~\Cref{lem:suppdef}, we have that there exist $\bsy_A \in A, \bsy_B \in B$ such that for any $\varepsilon > 0$, $\mu(R(\bsy_A, \varepsilon)), \mu(R(\bsy_B, \varepsilon)) > 0$. Moreover, since $\bsy_A, \bsy_B \not\in S$, we have
\begin{equation}\label{eq:fabnonzero}
    \forall i \in [d], \min\{F_i(\bsy_{A,i}), F_i(\bsy_{B,i})\} > 0.
\end{equation}

Now consider a positive sequence $\varepsilon_n \to 0$, then for each $\varepsilon_n$ in the sequence, the event $\by, \bar{\by} \in R(\bsy_A, \varepsilon_n), \by', \bar{\by}' \in R(\bsy_B, \varepsilon_n)$ occurs with non-zero probability, so that there must exist a $\bsy_n, \bar{\bsy}_n \in R(\bsy_A, \varepsilon_n)$, $\bsy_n', \bar{\bsy}_n' \in R(\bsy_B, \varepsilon_n)$ such that
\[
\tilde{F}(\bsy_n \wedge \bar\bsy_n) + \tilde{F}(\bsy_n' \wedge \bar\bsy_n') = \tilde{F}(\bsy_n \wedge \bsy_n') + \tilde{F}(\bar\bsy_n \wedge \bar\bsy_n').
\]

Since $\bsy_A, \bsy_B \not\in H_1$, $\tilde{F}$ is coordinate-wise continuous at $\bsy_A$ and $\bsy_B$. Using further decomposability of $\tilde{F}$, we have that $\tilde{F}$ is continuous at $\bsy_A, \bsy_B$ and $\bsy_A \wedge \bsy_B$. So, as $n \to \infty$, the left hand side and the right hand side of the above equation converge to
\[
\tilde{F}(\bsy_A) + \tilde{F}(\bsy_B) \quad\&\quad 2 \tilde{F}(\bsy_A \wedge \bsy_B)
\]
respectively. This raises a contradiction since 
\begin{align*}
    \tilde{F}(\bsy_A) + \tilde{F}(\bsy_B) - 2 \tilde{F}(\bsy_A \wedge \bsy_B) & \ge \tilde{F}(\bsy_B) - \tilde{F}(\bsy_A \wedge \bsy_B) \\
& \ge (F_1(\bsy_{B, 1}) - F_1(\bsy_{A, 1})) \prod_{i=2}^d \min\{F_i(\bsy_{A, i}), F_i(\bsy_{B, i})\} > 0,
\end{align*}
where for the last inequality we apply 
\[
F_1(\bsy_{B, 1}) - F_1(\bsy_{A, 1}) \ge \bar{F}_1(\bsy_{B, 1}) - \bar{F}_1(\bsy_{A, 1}) \ge \bar{F}_1(b) - \bar{F}_1(a) > 0
\]
and~\eqref{eq:fabnonzero}.

% \textcolor{red}{I got until here.}

We now prove the second claim, that is, for each $2 \leq d' \leq d$, $\mu(H_{d'-1})=1$ implies $\mu(H_{d'})=1$. We prove this via an argument analogous to the proof of $\mu(H_1) = 1$. From the definition, $H_{d'-1}$ can be expressed as a (finite or infinite) countable union of $(d-d'+1)$-dimensional hyperspaces $W_1, W_2, \ldots$. Suppose in contradiction that $\mu(H_{d'-1})=1$ but $\mu(H_{d'})<1$. Then, by the union bound, we have
\[
0 < \mu(H_{d'-1} \setminus H_{d'}) = \mu((\cup_k W_k) \setminus H_{d'}) = \mu(\cup_k (W_k \setminus H_{d'})) \leq \sum_k \mu(W_k \setminus H_{d'}).
\]
So, there must exist a $(d-d'+1)$-dimensional hyperspace $W_k$ such that $\mu(W_k \setminus H_{d'}) > 0$. For simplity we redenote $W_k$ as $W$, and write $W = \{\bsy \in \R^d: \bsy_{\cT} = \bsy_{\cT}^*\}$ for some $\cT \subseteq [d]$ with $|\cT| = d'-1$ and some $\bsy^* \in \R^d$ with $\mu_i(\{\bsy_i^*\})>0, \forall i \in \cT$. By reordering axes, we can w.l.o.g. assume $1 \not\in \cT$. Now, we apply analogous argument on space $W$ with measure $\mu_W(E) := \mu(E \cap W)$.

We reset $\bar{F}_1(x) := \mu(\{\bsy \in W: \bsy_1 \leq x\} \setminus H_{d'})$, then we have that $\bar{F}_1(x)$ is continuous and non-decreasing, $\lim_{x \to - \infty} \bar{F}_1(x)=0, \lim_{x \to \infty} \bar{F}_1(x) = \mu(W \setminus H_{d'}) > 0$, and there exists $a<b$ such that $0 < \bar{F}_1(a) < \bar{F}_1(b) < \lim_{x \to +\infty} \bar{F}_1(x)$. Recall the set $S := \{\bsy: \exists i \in [d] \text{ s.t. } F_i(\bsy_i)=0\}$ and its property $\mu(S)=0$. We consider two sets $A := \{\bsy \in W: \bsy_1 < a\} \setminus (H_{d'} \cup S), B := \{\bsy \in W: \bsy_1 > b\} \setminus (H_{d'} \cup S)$; then analogously $\mu_W(A),\mu_W(B)>0$. So, applying again~\Cref{lem:suppdef} on space $W$ and measure $\mu_W$, following exactly the same argument, we get $\bsy_A \in A, \bsy_B \in B$ such that $\mu_W(R(\bsy_A,\varepsilon)),\mu_W(R(\bsy_B,\varepsilon))>0, \forall \varepsilon > 0$, and that $\min\{F_i(\bsy_{A,i}),F_i(\bsy_{B,i})\}>0, \forall i \in [d]$. Then, analogously, for any positive sequence $\varepsilon_n \to 0$, we can find sequences $\bsy_n, \bar{\bsy}_n \in R(\bsy_A, \varepsilon_n) \cap W$, $\bsy_n', \bar{\bsy}_n' \in R(\bsy_B, \varepsilon_n) \cap W$ such that
\[
\tilde{F}(\bsy_n \wedge \bar\bsy_n) + \tilde{F}(\bsy_n' \wedge \bar\bsy_n') = \tilde{F}(\bsy_n \wedge \bsy_n') + \tilde{F}(\bar\bsy_n \wedge \bar\bsy_n').
\]

Now we argue that $\tilde{F}$ restricted on $W$ is continuous at $\bsy_A$, $\bsy_B$ and $\bsy_A \wedge \bsy_B$. Since $\bsy_A \in W$ but $\bsy_A \not\in H_{d'}$, we have that for any $i \in \cT$, $\mu_i(\{\bsy_{A,i}\})>0$, but for any distinct $i_1,\ldots,i_{d'} \in [d]$, one of $\mu_{i_1}(\{\bsy_{A,i_1}\}),\ldots,\mu_{i_{d'}}(\{\bsy_{A,i_{d'}}\})$ must be $0$. But recall that $|\cT|=d'-1$, so for any $i \not\in \cT$, we must have $\mu_i(\{\bsy_{A,i}\})=0$. This means that $F_i$ is continuous at $\bsy_{A,i}$ for any $i \not\in \cT$. Analogously, we can prove that $F_i$ is continuous at $\bsy_{B,i}$ for any $i \not\in \cT$. So, using the decomposability of $\tilde{F}$, we know that $\tilde{F}$ restricted on $W$ is continuous at $\bsy_A$, $\bsy_B$ and $\bsy_A \wedge \bsy_B$, i.e., for any sequence $\bsy_n \in W$ that converges to $\bsy_A$, $\bsy_B$ or $\bsy_A \wedge \bsy_B$, $\tilde{F}(\bsy_n)$ will converge to $\tilde{F}(\bsy_A)$, $\tilde{F}(\bsy_B)$ or $\tilde{F}(\bsy_A \wedge \bsy_B)$, respectively.

According to this continuity and the fact that $\bsy_n,\bar{\bsy}_n,\bsy_n',\bar{\bsy}_n' \in W$, we have that as $n \to \infty$, the left hand side and the right hand side of the previous equation converge to
\[
\tilde{F}(\bsy_A) + \tilde{F}(\bsy_B) \quad\&\quad 2 \tilde{F}(\bsy_A \wedge \bsy_B)
\]
respectively. This raises a contradiction since
\[
\begin{aligned}
\tilde{F}(\bsy_A) + \tilde{F}(\bsy_B) - 2 \tilde{F}(\bsy_A \wedge \bsy_B)
&\geq \tilde{F}(\bsy_B) - \tilde{F}(\bsy_A \wedge \bsy_B) \\
&\geq (F_1(\bsy_{B,1}) - F_1(\bsy_{A,1})) \prod_{i \neq 1: i \in \cT} F_i(\bsy_i^*) \prod_{i \neq 1: i \not\in \cT} F_i(\bsy_{A,i} \wedge \bsy_{B,i}) \\
&\geq (\bar{F}_1(b) - \bar{F}_1(a)) \prod_{i \neq 1: i \in \cT} \mu_i(\{\bsy_i^*\}) \prod_{i \neq 1: i \not\in \cT} \min\{F_i(\bsy_{A,i}), F_i(\bsy_{B,i})\} > 0.
\end{aligned}
\]
Note that in $d'=d$ case, the last product $\prod_{i \neq 1: i \not\in \cT} \min\{F_i(\bsy_{A,i}), F_i(\bsy_{B,i})\}$ just disappears, while other parts are not affected, so the proof is still valid.

Finally, we show that $\by$ is almost surely a constant vector. The previous arguments have already shown $\mu(H_d)=1$. By definition $H_d$ is a countable set, so $\by$ is a discrete random vector. Suppose by contradiction that $\by$ is not almost surely a constant, then there must exists $\bsy_A \neq \bsy_B$ such that $\pr(\by=\bsy_A),\pr(\by=\bsy_B)>0$. By reordering axes and exchanging $\bsy_A,\bsy_B$, we can w.l.o.g. assume $\bsy_{A,1} < \bsy_{B,1}$. Then again, since the event $\by = \bar{\by} = \bsy_A, \by' = \bar{\by}' = \bsy_B$ happens with non-zero probability, we have that
\[
\tilde{F}(\bsy_A) + \tilde{F}(\bsy_B) = 2 \tilde{F}(\bsy_A \wedge \bsy_B),
\]
which raises a contradiction since
\[
\begin{aligned}
\tilde{F}(\bsy_A) + \tilde{F}(\bsy_B) - 2 \tilde{F}(\bsy_A \wedge \bsy_B)
&\geq \tilde{F}(\bsy_B) - \tilde{F}(\bsy_A \wedge \bsy_B) \\
&\geq \mu_1(\{\bsy_{B,1}\}) \prod_{i=2}^d \min\{\mu_i(\{\bsy_{A,i}\}),\mu_i(\{\bsy_{B,i}\})\} \\
&\geq \pr(\by=\bsy_B) \prod_{i=2}^d \min\{\pr(\by=\bsy_A),\pr(\by=\bsy_B)\} > 0.
\end{aligned}
\]
Therefore, $\by$ is almost surely a constant.
\end{proof}

\subsection{Proof of~\Cref{thm:gamma}}\label{sec:pfgamma}

Write $\by'$ as an i.i.d. replication of $\by$. Define
$
h(\bsy) := \E \tilde{F}(\by \wedge \bsy) , \theta := \E \tilde{F}(\by \wedge \by') = \E[h(\by)]
$. Then apparently $\Gamma_2 = \E[h(\by)^2] - (\E[h(\by)])^2 = \var(h(\by)) \ge 0$.

For $\Gamma_1$, we have
\begin{align*}
    \Gamma_1 & = \E[(\tilde{F}(\by \wedge \by') - \theta)^2] - \E[(h(\by) - \theta)^2] - \E[(h(\by') - \theta)^2] \\
    & = \E[(\tilde{F}(\by \wedge \by') - \theta)^2] - \E[(h(\by) + h(\by') - 2\theta)^2].
\end{align*}
Since 
\begin{align*}
    & \E[(\tilde{F}(\by \wedge \by') - \theta) (h(\by) - \theta)] = \E[\E[(\tilde{F}(\by \wedge \by') - \theta) (h(\by) - \theta) \mid \by]] \\
    & = \E[\E[(\tilde{F}(\by \wedge \by') - \theta)\mid \by] \E[(h(\by) - \theta) \mid \by]] = \E[\E[(h(\by) - \theta)^2 \mid \by]] = \E[(h(\by) - \theta)^2],
\end{align*}
and we can deal with $\E[(\tilde{F}(\by \wedge \by') - \theta) (h(\by') - \theta)]$ analogously, we further have
\begin{align*}
    \Gamma_1 & = \E[(\tilde{F}(\by \wedge \by') - \theta)^2] -2 \E[(\tilde{F}(\by \wedge \by') - \theta) (h(\by) + h(\by') - 2\theta)] + \E[(h(\by) + h(\by') - 2\theta)^2] \\
    & = \E[(\tilde{F}(\by \wedge \by') - h(\by) - h(\by') + \theta)^2] \ge 0.
\end{align*}

Now our only remaining job is to prove that
\begin{equation}\label{eq:fht}
\E[(\tilde{F}(\by \wedge \by') - h(\by) - h(\by') + \theta)^2] = 0
\end{equation}
implies $\by$ is equal to a constant almost surely. Write $\bar{\by}, \bar{\by}'$ as another two i.i.d. replications of $\by$, then~\eqref{eq:fht} implies that almost surely,
\begin{align*}
    \tilde{F}(\by \wedge \bar{\by}) = h(\by) + h(\bar{\by}) - \theta, & \; \tilde{F}(\by' \wedge \bar{\by}') = h(\by') + h(\bar{\by}') - \theta; \\
    \tilde{F}(\by \wedge \by') = h(\by) + h(\by') - \theta, & \; \tilde{F}(\bar{\by} \wedge \bar{\by}') = h(\bar{\by}) + h(\bar{\by}') - \theta.
\end{align*}
This further implies that almost surely,
\[
\tilde{F}(\by \wedge \bar{\by}) + \tilde{F}(\by' \wedge \bar{\by}') = \tilde{F}(\by \wedge \by') + \tilde{F}(\bar{\by} \wedge \bar{\by}').
\]
In light of the above, we may finish the proof by applying~\Cref{lem:gamma}.

\section{Convergence analysis}
\label{sec:converganaly}

\subsection{Proofs of~\Cref{thm:multiacest} and~\Cref{prop:multiac}}

In this subsection, we prove the almost sure convergence results in~\Cref{thm:multiacest} and~\Cref{prop:multiac}. Let
	\begin{equation}
    \label{def:tildeF}
	    \tilde{F}_n(\bsy) := \frac{1}{n} \sum_{i=1}^n \one\{\tilde{\by}_i \le \bsy\} \quad\&\quad \tilde{F}(\bsy) := \int \one\{\bsy' \le \bsy\} d \tilde{\mu}_{\by}(\bsy');
	\end{equation}	
	and
    \begin{equation}
    \label{def:G_n}
        G_n(\bsy) := \frac{1}{n} \sum_{i=1}^n \one\{\by_i \ge \bsy\} \quad\&\quad G(\bsy) := \int \one\{\bsy' \ge \bsy\} d \mu_{\by}(\bsy').
    \end{equation}

\begin{lemma}\label{lem:permcdf}
	We have that
	\[
	\sup_{\bsy \in \R^{\dy}} |\tilde{F}_n(\bsy) - \tilde{F}(\bsy)| \asconv 0.
	\]
\end{lemma}

\begin{proof}
    Given a function $f: \R^d \to \R$, we denote its entry-wise left limit as
\[
f(\bsy^-) := \lim_{\bsy' \uparrow \bsy} f(\bsy') := \lim_{\substack{\bsy_j' \uparrow \bsy_j \\ j \in [d]}} f(\bsy').
\]
We will focus on these four functions:
\[
\begin{aligned}
\tilde{F}_n(\bsy) = \frac{1}{n} \sum_{i=1}^n \one\{\tby_i \leq \bsy\} \quad&\&\quad \tilde{F}(\bsy) = \int \one\{\bsy' \leq \bsy\} d\tilde{\mu}_{\by}(\bsy') \\
\tilde{F}_n(\bsy^-) = \frac{1}{n} \sum_{i=1}^n \one\{\tby_i < \bsy\} \quad&\&\quad \tilde{F}(\bsy^-) = \int \one\{\bsy' < \bsy\} d\tilde{\mu}_{\by}(\bsy'),
\end{aligned}
\]
where ``$\bsy' < \bsy$'' means that $\bsy'$ is entry-wise strictly smaller than $\bsy$.

We first show that for any fixed $\bsy$,
\begin{equation}\label{eq:pointconv}
\tilde{F}_n(\bsy) \asconv \tilde{F}(\bsy) \quad\&\quad \tilde{F}_n(\bsy^-) \asconv \tilde{F}(\bsy^-).
\end{equation}
To begin with, we have
\[
\begin{aligned}
\E[\tilde{F}_n(\bsy)] &= \E[\one\{\tby_1 \leq \bsy\}] = \int \one\{\bsy' \leq \bsy\} d\tilde{\mu}_{\by}(\bsy') = \tilde{F}(\bsy) \\
\E[\tilde{F}_n(\bsy^-)] &= \E[\one\{\tby_1 < \bsy\}] = \int \one\{\bsy' < \bsy\} d\tilde{\mu}_{\by}(\bsy') = \tilde{F}(\bsy^-),
\end{aligned}
\]
where the second equality of each line uses the fact that all the $\pi_d(i)$'s, $d=1,\ldots,\dy$ are nonidentical. Moreover, if we change one $\by_i$ into $\by_i'$ then both $\tilde{F}_n(\bsy)$ and $\tilde{F}_n(\bsy^-)$ would be altered by at most $\dy/n$, so the bounded difference inequality gives
\[
\tilde{F}_n(\bsy) - \E[\tilde{F}_n(\bsy)] \asconv 0
\quad\&\quad
\tilde{F}_n(\bsy^-) - \E[\tilde{F}_n(\bsy^-)] \asconv 0.
\]
Putting together yields~\eqref{eq:pointconv} for each fixed $\bsy$.

Now for any fixed $\varepsilon > 0$, we need to construct a grid $S = \{\bsy_1,\ldots,\bsy_L\} \subseteq (\R \cup \{\pm\infty\})^{\dy}$ such that:
\begin{equation}\label{eq:gridcond}
\forall \bsy \in \R^{\dy},\quad \exists \ell,\ell' \in [L] \;\st\quad \bsy_\ell \leq \bsy < \bsy_{\ell'} \quad\&\quad \tilde{F}(\bsy_{\ell'}^-) - \tilde{F}(\bsy_\ell) \leq \varepsilon.
\end{equation}

We first show how to construct such grid in $\dy=1$ case. Starting at $\bsy_1 := -\infty$, we recursively define $\bsy_{\ell+1} := \sup\{\bsy \geq \bsy_\ell: \tilde{\mu}_{\by}((\bsy_{\ell},\bsy)) < \varepsilon\}$. This definition immediately gives
\[
\tilde{\mu}_{\by}((\bsy_{\ell},\bsy_{\ell+1})) \leq \varepsilon
\quad\&\quad
(\tilde{\mu}_{\by}((\bsy_{\ell},\bsy_{\ell+1}]) \geq \varepsilon \quad\text{if}\quad \bsy_{\ell+1} < +\infty).
\]
The second property tells us that, after $L \leq \lfloor 1/\varepsilon \rfloor + 2$ steps, this procedure must end at $\bsy_L = +\infty$, and produce a grid $S$ containing $-\infty = \bsy_1 < \cdots < \bsy_L = +\infty$. And, the first property ensures that $S$ satisfies the grid condition~\eqref{eq:gridcond}.

We then extend our construction to $\dy>1$ case. For each coordinate $j \in [\dy]$, we use the previous procedure for $\dy=1$ case to construct a grid $S_j = \{\bsy_{1,j},\ldots,\bsy_{L_j,j}\} \subseteq \R\cup\{\pm\infty\}$ such that
\[
\forall a \in \R,\quad \exists \ell_j,\ell_j' \in [L_j]\;\st\quad \bsy_{\ell_j,j} \leq a < \bsy_{\ell_j',j} \quad\&\quad F_j(\bsy_{\ell_j',j}^-) - F_j(\bsy_{\ell_j,j}) \leq \varepsilon / \dy,
\]
where $F_j$ represents the marginal CDF of the $j$-th coordinate of the random vector $\by$. Now, our multi-dimensional grid is $S := S_1 \times \cdots \times S_{\dy} \subseteq (\R\cup\{\pm\infty\})^{\dy}$. For any $\bsy \in \R^{\dy}$, use the above condition to collect $\ell_j$ and $\ell_j'$ for each $j \in [\dy]$, such that the $j$-th coordinate of $\bsy$ is within the interval $[\bsy_{\ell_j, j}, \bsy_{\ell_j', j})$. Then consider $(\bsy_{\ell_1, 1}, \ldots, \bsy_{\ell_{\dy}, \dy}) \in S$ and $(\bsy_{\ell_1', 1}, \ldots, \bsy_{\ell_{\dy}', \dy}) \in S$, which we denote by $\bsy_{\ell}, \bsy_{\ell'}$, we can have that the grid condition~\eqref{eq:gridcond} holds for this $\bsy$:
\[
\bsy_\ell \leq \bsy < \bsy_{\ell'}
\;\&\;
\tilde{F}(\bsy_{\ell'}^-) - \tilde{F}(\bsy_\ell) = \prod_{j=1}^{\dy} F_j(\bsy_{\ell_j',j}^-) - \prod_{j=1}^{\dy} F_j(\bsy_{\ell_j,j})
\leq \sum_{j=1}^{\dy} (F_j(\bsy_{\ell_j',j}^-) - F_j(\bsy_{\ell_j,j}))
\leq \sum_{j=1}^{\dy} \frac{\varepsilon}{\dy}
= \varepsilon.
\]
So, $S$ is a valid grid.

With such a grid $S$ in hand, the main statement is easy to prove. For any $\bsy \in \R^{\dy}$, by~\eqref{eq:gridcond}, we have $\exists \bsy_{\ell}, \bsy_{\ell'} \in S$ such that
\[
\tilde{F}_n(\bsy_\ell) - \tilde{F}(\bsy_\ell) - \varepsilon \leq \tilde{F}_n(\bsy) - \tilde{F}(\bsy) \leq \tilde{F}_n(\bsy_{\ell'}^-) - \tilde{F}(\bsy_{\ell'}^-) + \varepsilon.
\]
So, using the point-wise convergence~\eqref{eq:pointconv} and the finiteness of $S$, we have that
\[
\sup_{\bsy \in \R^{\dy}} \left|\tilde{F}_n(\bsy) - \tilde{F}(\bsy)\right|
\leq \max_{\bsy \in S} \left|\tilde{F}_n(\bsy)-\tilde{F}(\bsy)\right|+\max_{\bsy \in S} \left|\tilde{F}_n(\bsy^-)-\tilde{F}(\bsy^-)\right|+\varepsilon
\asconv \varepsilon.
\]
Since this holds for arbitrary $\varepsilon > 0$, we prove the desired result.
\end{proof}

\begin{proof}[Proof of~\Cref{thm:multiacest}]
Without loss of generality, we just need to prove $1 / n^3$ times the numerator of $\htaci$:
\[
Q_n := \frac{1}{n^3} \sum_{i=1}^n (n \tilde{R}(Y_i \wedge Y_{M_{\bz}(i)}) - \check{L}_i^2)
\]
converges almost surely to $Q := \int \var(\pr(\by \ge \bsy \mid \bz)) d \tilde{\mu}_{\by}(\bsy)$, and the rest follows from analogous arguments. With the new notations~\eqref{def:tildeF},~\eqref{def:G_n}, we rewrite $Q_n$ as 
    \[
	Q_n := \frac{1}{n} \sum_{i=1}^n (\tilde{F}_n(\by_i \wedge \by_{M_{\bz}(i)}) - G_n(\tilde{\by}_i)^2).
	\]
	Using exactly the same argument as \citet[Lemma~11.9]{azadkia2021simple}, we have $Q_n - \E[Q_n] \asconv 0$.
	Thus, the only remaining question is to prove $\E[Q_n] \to Q$. Let
	\[
	Q_n' := \frac{1}{n} \sum_{i=1}^n (\tilde{F}(\by_i \wedge \by_{M_{\bz}(i)}) - G(\tilde{\by}_i)^2).
	\]
    Then
    \[
    |Q_n' - Q_n| \le \sup_{\bsy \in \R^{\dy}} |\tilde{F}_n(\bsy) - \tilde{F}(\bsy)| + 2 \sup_{\bsy \in \R^{\dy}} |G_n(\bsy) - G(\bsy)|.
    \]
    Applying~\Cref{lem:permcdf} and treating $G_n$ by analogous arguments, we have  $|Q_n' - Q_n| \asconv 0$.
    
	We now focus on $Q_n'$. Notice that $Q_n'$ can be equivalently written as 
	\[
	Q_n' = \int \frac{1}{n} \sum_{i=1}^n \one\{\bsy \le \by_i\} \one\{\bsy \le \by_{M_{\bz}(i)}\} d \tilde{\mu}_{\by}(\bsy) - \frac{1}{n} \sum_{i=1}^n G(\tilde{\by}_i)^2.
	\]
	The expectation of the second term is apparently equal to $\int G(\bsy)^2 d \tilde{\mu}_{\by}(\bsy)$. For the first term, let $\mathcal{F}$ be the $\sigma$-field generated by $(\bz_1, \ldots, \bz_n)$ and the random variables used for breaking ties in the selection of nearest neighbors, then we have
	\[
	\E[\one\{\bsy \le \by_1\} \one\{\bsy \le \by_{M_{\bz}(1)}\} \mid \mathcal{F}] = G_{\bz_1}(\bsy) G_{\bz_{M_{\bz}(1)}}(\bsy),
	\]
	where $G_{\bz}(\bsy) := \pr(\by \ge \bsy \mid \bz)$. In light of the above, following exactly the same proof as \citet[Lemma~11.8]{azadkia2021simple}, we have $\E[Q_n'] \to Q$. Putting together yields the desired result.
\end{proof}

\begin{proof}[Proof of~\Cref{prop:multiac}]
    This follows from exactly the same argument as the proof of \Cref{thm:multiacest}, except that we have replaced $\tilde{F}_n, \tilde{F}$ by
\[
	F_n(\bsy) := \frac{1}{n} \sum_{i=1}^n \one\{\by_i \le \bsy\} \quad\&\quad F(\bsy) := \E\left[\one\{\by_1 \le \bsy\}\right].
\]
\end{proof}

\subsection{Proof of~\Cref{thm:tacnormal}}

As a direct consequence of~\Cref{thm:gamma} and~\Cref{lem:probneigh} that will be proved later, $\sigma_n^2$ is always positive. Moreover, $0 < \liminf \sigma_n^2 \le \limsup \sigma_n^2 < \infty$. We now focus on the rest results.

Recall the definitions of $\tilde{F}_n(\cdot), \tilde{F}(\cdot), G_n(\cdot)$ introduced in~\eqref{def:tildeF} and~\eqref{def:G_n}. Also, write $h(\bsy) := \E [\tilde{F}(\by \wedge \bsy)], \theta := \E[h(\by)]$. Armed with these definitions, we may decompose $1 / n^3$ times the numerator of $\htaci$ as $(S_n + \Delta_n^{(1)} + \Delta_n^{(2)} + \Delta_n^{(3)}) / \sqrt{n}$, where
\[
\begin{aligned}
\Delta_n^{(1)} &:= \frac{1}{\sqrt{n}} \sum_{i=1}^n (\tilde{F}_n-\tilde{F})(\by_i \wedge \by_{M(i)}) - \frac{1}{\sqrt{n}(n-1)} \sum_{i \neq j} (\tilde{F}_n-\tilde{F})(\by_i \wedge \by_j) \\
\Delta_n^{(2)} &:= \frac{1}{\sqrt{n}(n-1)} \sum_{i \neq j} \tilde{F}_n(\by_i \wedge \by_j) - \frac{1}{\sqrt{n}} \sum_{i=1}^n G_n(\tilde{\by}_i)^2 \\
\Delta_n^{(3)} &:= \frac{1}{\sqrt{n}} \sum_{i=1}^n (2 h(\by_i) - \theta) - \frac{1}{\sqrt{n}(n-1)} \sum_{i \neq j} \tilde{F}(\by_i \wedge \by_j) \\
S_n &:= \frac{1}{\sqrt{n}} \sum_{i=1}^n (\tilde{F}(\by_i \wedge \by_{M(i)}) - 2 h(\by_i) + \theta).
\end{aligned}
\]
(From here and below, we abbreviate $M := M_{\bz}$ and $(\tilde{F}_n - \tilde{F})(\bsy) := \tilde{F}_n(\bsy) - \tilde{F}(\bsy)$.)

Then if we can prove that $\Delta_n^{(t)} \overset{\mathbb{P}}{\to} 0$ for $t \in \{1, 2, 3\}$ and $S_n / \tilde{\sigma}_n \overset{d}{\to} \mathcal{N}(0, 1)$, where $\tilde{\sigma}_n^2$ corresponds to the numerator of the $\sigma^2_n$ defined in~\Cref{thm:tacnormal}, the desired result follows directly from Slutsky's theorem. The rest of this section is organized as follows. First, we introduce some preliminary lemmas and provide proof of these lemmas; second, we use these lemmas to analyze the convergence of $\Delta_n^{(t)}$ and $S_n$. 
% in \Cref{sec:varest}, we discuss how to estimate $\sigma_n^2$ from data.

\noindent{\bf Notations.} For any vector $\bs{u} \in \R^d$, we write $\|\bs{u}\|$ as its $\ell_2$-norm. We define $\mathrm{Binom}(n, p)$ as a binomial distribution with parameters $n \ge 1$ and $p \in [0, 1]$. Finally, for simplicity of exposition, we rewrite $M_{\bz}(\cdot)$ as $M(\cdot)$, and write $M^{-1}(i) := \{j: M(j) = i\}$, write
\[
d_{\max}(M):=\underset{1\leq k\leq n}{\max}\sum_{j : j \neq k} \mathbbm{1}\{M(j)=k\}
\]

\begin{lemma}
\label{lem:maxdegree}
For any constants $\delta, K > 0$, 
\[
\E[d_{\text{max}}(M)^K] / n^\delta \to 0.
\]
\end{lemma}
\begin{proof}

First, by definition, $d_{\text{max}}(M)\leq n$. Once we further have the equation that for some constant $C > 0$,
\begin{equation}\label{eq:dmaxp}
\pr [d_{\text{max}}(M)>(\log n)^2] \le C e^{- (\log n)^2 / 2},    
\end{equation}
the desired result can be derived:
     \begin{equation*}
         \begin{aligned}
             \E d_{\text{max}}(M)^K&=\E [d_{\text{max}}(M)^K \one\{d_{\text{max}}(M)\leq (\log n)^2\}]+\E [d_{\text{max}}(M)^K \one\{d_{\text{max}}(M) > (\log n)^2\}]\\
             &\leq(\log n)^{2K} + n^K \pr[d_{\text{max}}(M)>\log^2 n] \le (\log n)^{2K} + C n^K e^{- (\log n)^2 / 2}.
         \end{aligned}
     \end{equation*}
Now our focus is to prove~\eqref{eq:dmaxp}. To prove this, first, as a direct consequence of~\citet[Lemma~11.4]{azadkia2021simple}, there exists a deterministic constant $C(\dz)$ depending only on $\dz$ such that
\begin{equation}\label{eq:maxineibor}
      \sum_{j \neq 1: \bz_j \neq \bz_1} \mathbbm{1}\{M(j)=1\}\leq C(\dz).  
\end{equation}

Second, we prove that for a random variable $X \sim \mathrm{Binom}(n,1/n),$ 
\begin{equation}\label{eq:binom}
\pr[X>k] \leq 1/k \,! ,\; \forall k > 0.
\end{equation}
To prove this, observe that $ \forall j \geq 1$,
\begin{equation*}
    \begin{aligned}
\pr(X=j)
&= \binom{n}{j} \left(\frac{1}{n}\right)^j \left(1-\frac{1}{n}\right)^{n-j}=   \frac{1}{j!}\left(1-\frac{1}{n}\right)^{n-j}\frac{n!}{(n-j)!}\left(\frac{1}{n}\right)^{j} \\
&= \frac{1}{j!} \left(1-\frac{1}{n}\right)^{n-j} \prod_{i=0}^{j-1} \left(1-\frac{i}{n}\right) \leq \frac{1}{j!} \exp\left(-\left(1-\frac{j}{n}\right)\right) \exp\left(- \sum_{i=0}^{j-1} \frac{i}{n}\right) \\
&= \frac{1}{e \cdot j!} \exp\left(\frac{\frac32 j - \frac12 j^2}{n}\right) \leq \frac{1}{j!},\\
\end{aligned}
\end{equation*}
where the first inequality comes from the basic inequality: $(1+x)^a\leq e^{ax}, \forall x > -1, a>0$.
% \begin{equation}\label{eq:basic}
    
% \end{equation}

Then for $j\geq 2$, we can have 
$$\pr(X=j) \leq \frac{1}{j!} \leq \frac{j-1}{j!} = \frac{1}{(j-1)!} - \frac{1}{j!}.$$
Hence $\forall k\geq 1$,
$$
\pr(X>k) = \sum_{k < j \leq n} \pr(X=j) \leq \sum_{k < j \leq n} \left(\frac{1}{(j-1)!}-\frac{1}{j!}\right) \leq \frac{1}{k!}.
$$
% and for $ k=0 $ the desired inequality trivially holds, so the claim holds. 

In light of~\eqref{eq:maxineibor} and~\eqref{eq:binom}, using a union bound,
\begin{equation*}
    \begin{aligned}
&\pr(d_{\text{max}}(M) > (\log n)^2)
= \pr\left(\bigcup_{k=1}^n \left\{\sum_{j : j\neq k} \mathbbm{1}\{M(j)=k\} > (\log n)^2 \right\} \right) \\
&\qquad \leq \sum_{k=1}^n \pr\left( \sum_{j : j\neq k} \mathbbm{1}\{M(j)=k\} > (\log n)^2 \right) = n \pr\left( \sum_{j = 2}^n \mathbbm{1}\{M(j)=1\} > (\log n)^2 \right), \\
\end{aligned}
\end{equation*}
where to get the last equality we apply that $\bz_i$'s are i.i.d. To control the right hand side of the above derivation, we split the sum $\sum_{j = 2}^n \mathbbm{1}\{M(j)=1\}$ into two parts:

$$
\sum_{j = 2}^n \mathbbm{1}\{M(j)=1\} = \sum_{j = 2}^n \mathbbm{1}\{M(j)=1\} \one\{\bz_j = \bz_1\} + \sum_{j = 2}^n \mathbbm{1}\{M(j)=1\} \one\{\bz_j \neq \bz_1\}.
$$

The second part can be directly controlled by~\eqref{eq:maxineibor}. For the first part, conditional on $\bz_1,\cdots,\bz_n$, if $\forall j \neq 1$, $\bz_j \neq \bz_1$, then the first part is equal to zero. Otherwise, write $N_n$ as the number of $\bz_j$ that are equal to $\bz_1$, then the first part follows Binomial distribution with parameters $(N_n, 1 / N_n)$. Now in light of~\eqref{eq:binom}, we can bound that, $\forall k>0$,
\begin{equation*}
    \begin{aligned}
& \pr\left(\sum_{j = 2}^n \mathbbm{1}\{M(j)=1\} \one\{\bz_j = \bz_1\} > k \right) \\
&\qquad = \E\left[  \pr \left( \sum_{j = 2}^n \mathbbm{1}\{M(j)=1\} \one\{\bz_j = \bz_1\} > k \mid \bz_1,\cdots,\bz_n\right)\right] \leq \frac{1}{k!}.
\end{aligned}
\end{equation*}
Putting together, we obtain that
$$
\pr[d_{\max}(M) > (\log n)^2]
\leq n \pr\left(\sum_{j \neq 1: \bz_j=\bz_1} \mathbbm{1}\{M(j)=1\} > (\log n)^2 - C(\dz)\right)
\leq \frac{n}{(\lfloor (\log n)^2 \rfloor - C(\dz))!}
$$
From above,~\eqref{eq:dmaxp} is a direct consequence of Stirling's approximation.
\end{proof}

\begin{lemma}
    \label{lem:probneigh}
  We have that 
  \[
  a_n := \pr(M(1)=2, M(2)=1) \leq \frac{1}{n - 1} \quad\&\quad b_n := \pr(M(1)=M(2)) \le \frac{C(\dz)}{n - 1},
  \]
    where $C(\dz)$ is a positive constant depending only on $\dz$.
\end{lemma}

\begin{proof}
    For $a_n$:
    \[
    a_n \le \pr(M(1) = 2) = \frac{1}{n - 1}.
    \]
    For $b_n$, using the i.i.d. property of the $\bz_i$'s, we have
\begin{equation} \label{eq:bnbnd}
\begin{aligned}
\frac{b_{n}}{n-2}
&= \frac{1}{n-2} \pr(M(1)=M(2)) = \frac{1}{n-2} \sum_{j=3}^{n} \pr(M(1)=M(2)=j) \\
&= \pr(M(1)=M(2)=3) = \frac{1}{(n - 1)(n - 2)} \sum_{i \neq j \;\&\; i, j \neq 3} \pr(M(i)=M(j)=3),
\end{aligned}
\end{equation}
where
\[
\sum_{i \neq j \;\&\; i, j \neq 3} \pr(M(i)=M(j)=3) = \E\left[\sum_{i \neq j \;\&\; i, j \neq 3} \one\{M(i)=M(j)=3\}\right] \le \E\left[|M^{-1}(3)|^2\right].
\]
Now observe that
\[
|M^{-1}(3)|^2 \le 2 |\{j: M(j) = 3 \;\&\; \bz_j = \bz_3\}|^2 + 2 |\{j: M(j) = 3 \;\&\; \bz_j \neq \bz_3\}|^2.
\]
As a direct consequence of~\citet[Lemma~11.4]{azadkia2021simple}, the second term can be bounded by a $C(\dz)$, i.e., a constant depending only on $\dz$. For the first term, let $N := \sum_{i \neq 3} \one\{\bz_i = \bz_3\}$. Then conditioning on a $N \ge 1$, $|\{j: M(j) = 3 \;\&\; \bz_j = \bz_3\}|$ follows a binomial distribution with parameters $(N, 1 / N)$, so that
\[
\E[|\{j: M(j) = 3 \;\&\; \bz_j = \bz_3\}|^2 \mid N] = 2 - 1 / N.
\]
Putting together, we have that $\E[|M^{-1}(3)|^2]$ is bounded above by $2 C(\dz) + 4$. In light of this and~\eqref{eq:bnbnd}, we prove the desired result.
\end{proof}

\begin{lemma}
\label{lem:jointdis}
We have that
\[
 \pr(M(1) = i, M(2) = j) = \left\{\begin{aligned}
 & \frac{b_n}{n-2}, & & i = j\\
 & a_n, & & i = 2, j = 1\\
 & \frac{1 - a_n (n - 1)}{(n-1)(n - 2)}, & & i = 2, j \ge 3 \;\mathrm{or}\; j = 1, i \ge 3\\
 & \frac{1 + a_n - b_n}{(n - 2)(n - 3)} - \frac{2}{(n - 1)(n - 2)(n - 3)}, & & i \neq j \;\&\; i, j \ge 3
 \end{aligned}\right.
\]
\end{lemma}

\begin{proof}
The proof of the case $i = j$ and $i = 2, j = 1$ follows directly from the definition and the i.i.d. property of all $\bz_i$'s. We now focus on the third case. Without loss of generality, we just need to prove this holds for $i = 2, j \ge 3$, and the other case follows from an analogous argument. First,
\[
\pr(M(1)=2,M(2)=j) =\pr(M(1)=2, M(2)\neq 1)\pr(M(2)=j|M(1)=2, M(2)\neq 1).
\]
Apparently, using the i.i.d. property of all the $\bz_i$'s, by conditioning on the event $M(1)=2, M(2)\neq 1$, all the $\bz_3, \ldots, \bz_n$ are still exchangeable, so that they are equally likely to be selected as $M(2)$. Then
\[
\pr(M(2)=j|M(1)=2, M(2)\neq 1) = \frac{1}{n - 2} \sum_{j' = 3}^n \pr(M(2)=j'|M(1)=2, M(2)\neq 1)  = \frac{1}{n - 2}.
\]
Putting back, we have
\begin{align*}
    \pr(M(1)=2,M(2)=j) & = \pr(M(1)=2, M(2) \neq 1) \frac{1}{n - 2} \\
    & = (\pr(M(1)=2) - \pr(M(1)=2, M(2) = 1)) \frac{1}{n - 2} = \left(\frac{1}{n - 1} - a_n\right) \frac{1}{n - 2}.
\end{align*}

For the case $i \neq j \;\&\; i, j \ge 3$, using again the i.i.d. property, the result is a direct consequence of the following equation.
\begin{align*}
    & (n - 2) (n - 3) \pr(M(1) = i, M(2) = j) = 1 - \pr(M(1) = M(2)) - \pr(M(1) = 2, M(2) = 1) \\
    &\qquad - \pr(M(1) = 2, M(2) \ge 3) - \pr(M(1) \ge 3, M(2) = 1)).
\end{align*}
\end{proof}

% \begin{lemma}
% \label{lem:gassian}
%     For all $ a,b \geq 0 $,
% $$
% \sup_{z\in \R} |\Phi(b z) - \Phi(a z)| \lesssim \min\left(1, \frac{|b-a|}{\sqrt{2 \pi e} \min(a,b)}\right),
% $$
% where $\Phi(z)$ represents the CDF of a standard normal random variable.
% \end{lemma}
% \begin{proof}
%     Without loss of generality, assume $ b \geq a $. For all $ z \in \R $, we have

%     $$
% |\Phi(b z) - \Phi(a z)| \leq (b-a) |z| \sup_{a \leq \xi \leq b} |\Phi'(\xi z)| = \frac{1}{\sqrt{2\pi}} (b-a) |z| e^{- \frac12 a^2 z^2},
% $$
% where the inequality holds through the Lagrange's mean value theorem. Since $a \geq 0$,
% $$
% \sup_z |\Phi(b z) - \Phi(a z)|
% \leq \frac{1}{\sqrt{2\pi}} (b-a) \sup_z |z| e^{- \frac12 a^2 z^2}
% = \frac{1}{\sqrt{2 \pi e}} \left(\frac{b}{a}-1\right).
% $$
% And, obviously
% $$
% \sup_z |\Phi(b z) - \Phi(a z)| \leq 1
% $$
% so the desired result follows. 
% \end{proof}
% \subsubsection{Analysis of $\Delta_n^{(t)}$'s and $S_n$}\label{sec:normalterms}

Armed with these lemmas, we are now in a position to prove~\Cref{thm:tacnormal}. As also mentioned at the beginning of this section, the only remaining job is to understand the asymptotic properties of $\Delta_n^{(t)}, t \in \{1, 2, 3\}$ and $S_n$.

\paragraph{Analysis of $\Delta_n^{(1)}$} We first show that there exists some universal constant $C > 0$ such that
\begin{equation}
\label{eq:mibound}
\E|\tilde{F}_n(\by_1 \wedge \by_2) - \tilde{F}(\by_1 \wedge \by_2)|^2 \le C \dy^2 / n.    
\end{equation}

First, as a direct consequence of bounded difference inequality, there exists some universal constant $C' > 0$ such that almost surely,
\begin{equation*}
    \E\left[\left(\tilde{F}_n(\by_1 \wedge \by_2) - \E[\tilde{F}_n(\by_1 \wedge \by_2) \mid \by_1, \by_2]\right)^2 \;\Big|\; \by_1, \by_2\right] \le C' \frac{\dy^2}{n}.
\end{equation*}

Writing $\mathcal{S}$ as the set of indices $i$ such that $\pi_{d'}(i)$ is equal to $1$ or $2$ for some $1 \le d' \le \dy$, then
\begin{align*}
    \E[\tilde{F}_n(\by_1 \wedge \by_2) \mid \by_1, \by_2] & = \frac{1}{n} \sum_{i \in \mathcal{S}}\E[\one\{\tilde{\by}_i \le \by_1 \wedge \by_2\} \mid \by_1, \by_2] + \frac{1}{n} \sum_{i \not\in \mathcal{S}}\E[\one\{\tilde{\by}_i \le \by_1 \wedge \by_2\} \mid \by_1, \by_2] \\
    & = \frac{1}{n} \sum_{i \in \mathcal{S}}\E[\one\{\tilde{\by}_i \le \by_1 \wedge \by_2\} \mid \by_1, \by_2] + \frac{n - |\mathcal{S}|}{n} \tilde{F}(\by_1 \wedge \by_2).
\end{align*}
Putting together and using $|\mathcal{S}| \le 2 \dy$, we have almost surely,
\begin{align*}
& \E[|\tilde{F}_n(\by_1 \wedge \by_2) - \tilde{F}(\by_1 \wedge \by_2)|^2 \mid \by_1, \by_2] = \E[|\tilde{F}_n(\by_1 \wedge \by_2) - \E[\tilde{F}_n(\by_1 \wedge \by_2) \mid \by_1, \by_2]|^2 \mid \by_1, \by_2] \\
& \qquad + \E[|\E[\tilde{F}_n(\by_1 \wedge \by_2) \mid \by_1, \by_2] - \tilde{F}(\by_1 \wedge \by_2)|^2 \mid \by_1, \by_2] \le (C' + 4) \dy^2 / n,
\end{align*}
thereby proving the desired result. 

We next turn to the analysis of $\E(\Delta_n^{(1)})^2$. For simplicity, we define $A_{i,j}:=\tilde{F}_n(\by_i \wedge \by_{j})-\tilde{F}(\by_i \wedge \by_{j})$, then we can express $\Delta_n^{(1)}$ as
$$\Delta_n^{(1)}=\frac{1}{\sqrt{n}}\sum_{i = 1}^n A_{i,M(i)}-\frac{1}{\sqrt{n}(n-1)}\sum_{i\neq j} A_{i,j}.$$
We expand $(\Delta_n^{(1)})^2$ into three parts, 
  $$(\Delta_n^{(1)})^2=T_2-2T_1+T_0,$$
  where 
  $$\begin{aligned}
      &T_2:=\frac{1}{n}\left(\sum_{i=1}^n A_{i,M(i)}\right)^2,\,\,T_1:=\frac{1}{n(n-1)}\left(\sum_iA_{i,M(i)}\right)\left(\sum_{i\neq j}A_{i,j}\right),\,\,T_0:=\frac{1}{n(n-1)^2}\left(\sum_{i\neq j}A_{i,j}\right)^2.
  \end{aligned}$$
Using the i.i.d. property of the $\bz_i$'s and the independence between $\by_i$ and $\bz_i$, 
\begin{align*}
	\E[T_1 \mid \by_1, \ldots, \by_n] & = \frac{1}{n(n-1)}\left(\E\left[\sum_iA_{i,M(i)} \mid \by_1, \ldots, \by_n\right]\right)\left(\sum_{i\neq j}A_{i,j}\right) \\
	& = \frac{1}{n(n-1)}\left(\frac{1}{n - 1}\sum_{i \neq j}A_{i, j}\right)\left(\sum_{i\neq j}A_{i,j}\right) = T_0,
\end{align*}
which yields $\E[(\Delta_n^{(1)})^2] = \E[T_2] - \E[T_0]$. We first consider $\E T_0$:
$$\begin{aligned}
T_0
& =\frac{1}{n(n-1)^2}\left(\sum_{i\neq j}A_{i,j}\right)\left(\sum_{i'\neq j'}A_{i',j'}\right)\\
&= \frac{1}{n(n-1)^2} \sum_{i \neq j} A_{i,j} \Big( A_{i,j} + A_{j,i} + \sum_{j' \not\in \{i,j\}} A_{i,j'} + \sum_{j' \not\in \{i,j\}} A_{j,j'} \\ &\qquad\qquad\qquad\qquad\qquad\qquad\qquad+  \sum_{i' \not\in \{i,j\}} A_{i',i} + \sum_{i' \not\in \{i,j\}} A_{i',j} + \sum_{i',j' \not\in \{i,j\}; i' \neq j'} A_{i',j'} \Big) \\
&= \frac{1}{n(n-1)^2} \left( 2\sum_{i\neq j} A_{i,j}^2 + 4 \sum_{i\neq j} \left(\sum_{k \not\in \{i,j\}}A_{i,j} A_{i,k}\right) + \sum_{i\neq j} \left(\sum_{k, l \not\in\{i,j\}; k \neq l} A_{i,j} A_{k,l}\right) \right),\\
\end{aligned}$$
where for the last inequality we apply $A_{i,j} = A_{j, i}$. Now for convenience, we write 
  $$\begin{aligned}
      &S_0:= \frac{1}{n(n-1)(n-2)(n-3)}\E\sum_{i\neq j} \left(\sum_{k, l \not\in\{i,j\}; k \neq l} A_{i,j} A_{k,l}\right),\\
      &S_1:=\frac{1}{n(n-1)(n-2)}\E \sum_{i\neq j} \left(\sum_{k \not\in \{i,j\}}A_{i,j} A_{i,k}\right), \qquad S_2 :=\frac{1}{n(n-1)}\E\sum_{i\neq j} A^2_{i,j}.
  \end{aligned}$$
Then, 
$$\begin{aligned}
\E T_0
&= \frac{1}{n(n-1)^2} (2n(n-1)S_2 + 4n(n-1)(n-2)S_1 + n(n-1)(n-2)(n-3)S_0) \\
&= \frac{2}{n-1} S_2 + \frac{4(n-2)}{n-1} S_1 + \frac{(n-2)(n-3)}{n-1} S_0.\\
\end{aligned}$$
For $\E T_2$, by definition, 
$$
\begin{aligned}
T_2
&= \frac{1}{n}\left(\sum_{i=1}^n A_{i, M(i)}\right)\left(\sum_{i'=1}^n A_{i',M(i')}\right) = \frac{1}{n} \left( \sum_{i=1}^n A_{i, M(i)}^2 + \sum_{i \neq j} A_{i, M(i)} A_{j, M(j)} \right).
\end{aligned}
$$
Then applying~\Cref{lem:jointdis}, we get that
$$
\begin{aligned}
    \E[T_2 \mid \;&\;\by_1, \ldots, \by_n] \\
    &= \frac{1}{n (n-1)} \sum_{i \neq j} A_{i, j}^2 + \frac{1}{n} \sum_{i \neq j} \Bigg(  a_n A_{i, j} A_{j, i} +\frac{b_n}{n-2} \sum_{k \neq i, j} A_{i,k} A_{j,k} \\
    &\qquad\qquad + \left(\frac{1}{n-1}-a_n\right)\frac{1}{n-2} \sum_{k \neq i, j} A_{i,j} A_{j,k} + \left(\frac{1}{n-1}-a_n\right)\frac{1}{n-2} \sum_{k \neq i, j} A_{i,k} A_{j,i}  \\
&\qquad\qquad + \left(\frac{1 + a_n - b_n}{(n - 2)(n - 3)} - \frac{2}{(n - 1)(n - 2)(n - 3)}\right) \sum_{k, l \not\in \{i, j\}; k \neq l} A_{i,k} A_{j,l}\Bigg) \\
&= \frac{1}{n} \left( \left(\frac{1}{n-1}+a_n\right) \sum_{i \neq j} A_{i,j}^2 + \left(b_n+\frac{2}{n-1}-2a_n\right) \frac{1}{n-2} \sum_{i \neq j} \left(\sum_{k \neq i, j} A_{i,j} A_{i,k} \right)\right.\\
&\qquad\qquad + \left.\left(1-\frac{2}{n-1}+a_n-b_n\right)\frac{1}{(n-2)(n-3)} \sum_{i \neq j} \left(\sum_{k, l \not\in \{i, j\}; k \neq l} A_{i,j} A_{k,l}\right)\right),\\
\end{aligned}
$$
where the second equality holds because $A_{i,j}=A_{j,i}$. Then, we take expectation on both side to get
$$
\begin{aligned}
\E T_2
&= (n-1) \left( \left(\frac{1}{n-1}+a_n\right)S_2 + \left(b_n+\frac{2}{n-1}-2a_n\right)S_1
+ \left(1-\frac{2}{n-1}+a_n-b_n\right)S_0 \right) \\
&= (1+(n-1)a_n)S_2 + (2+(n-1)b_n-2(n-1)a_n)S_1 + (n-3+(n-1)a_n-(n-1)b_n)S_0.\\
\end{aligned}
$$
Combining our analysis of $\E T_0$ and $\E T_2$, we obtain
\[
	\begin{aligned}
	\E (\Delta_n^{(1)})^2
	&= \E T_2 - \E T_0 = \left(1+(n-1)a_n-\frac{2}{n-1}\right)S_2 + \left(\frac{4}{n-1}-2+(n-1)b_n-2(n-1)a_n\right)S_1 \\
	&\qquad + \left(1-\frac{2}{n-1}+(n-1)a_n-(n-1)b_n\right)S_0.
\end{aligned}
\]
From the definitions of $A_{i,j}, S_0, S_1, S_2$,~\eqref{eq:mibound}, and the basic inequality that $|A_{i,j} A_{k, l}| \le (A_{i,j}^2 + A_{k, l}^2) / 2$, we know $|S_0|, |S_1|, |S_2| \le C(\dy) / n$ for some constant $C(\dy)$ depending only on $\dy$. Combining this with the above equality and~\Cref{lem:probneigh}, we obtain $\E(\Delta_n^{(1)})^2 \to 0$, thus proving the desired result.

% In light of the above inequality,~\Cref{lem:probneigh},~\eqref{eq:mibound}, and the basic inequality that $A_{i,j} A_{k, l} \le (A_{i,j}^2 + A_{k, l}^2) / 2$, we obtain the desired result.

\paragraph{Analysis of $\Delta_n^{(2)}$ and $\Delta_n^{(3)}$} For $\Delta_n^{(2)}$, expanding $G_n(\tilde{\by}_i)$, we have
\begin{align*}
    \frac{1}{\sqrt{n}} \sum_{i=1}^n G_n(\tilde{\by}_i)^2 & = \frac{1}{\sqrt{n}} \sum_{i=1}^n \frac{1}{n^2} \sum_{j, k}\one\{\by_j \ge \tilde{\by}_i\}\one\{\by_k \ge \tilde{\by}_i\} = \frac{1}{\sqrt{n}} \sum_{i=1}^n \frac{1}{n^2} \sum_{j, k}\one\{\tilde{\by}_i \le \by_j \wedge \by_k\} \\
    & = \frac{1}{n^{3/2}} \sum_{j, k} \tilde{F}_n(\by_j \wedge \by_k).
\end{align*}
So
$$
\begin{aligned}
\Delta_n^{(2)}
&= \frac{1}{\sqrt{n}(n-1)} \sum_{i \neq j} \tilde{F}_n(\by_i \wedge \by_j) - \frac{1}{n^{3/2}} \sum_{i,j} \tilde{F}_n(\by_i \wedge \by_j) \\
&= \frac{1}{\sqrt{n}(n-1)} \sum_{i \neq j} \tilde{F}_n(\by_i \wedge \by_j) - \frac{1}{n^{3/2}} \sum_j \tilde{F}_n(\by_j) - \frac{1}{n^{3/2}} \sum_{i \neq j} \tilde{F}_n(\by_i \wedge \by_j) \\
&= \frac{1}{n^{3/2}(n-1)} \sum_{i \neq j} \tilde{F}_n(\by_i \wedge \by_j) - \frac{1}{n^{3/2}} \sum_j \tilde{F}_n(\by_j), \\
\end{aligned}
$$
which implies that almost surely,
$$
|\Delta_n^{(2)}| \leq \max\left(\frac{1}{n^{3/2}(n-1)} \cdot n(n-1) ,\; \frac{1}{n^{3/2}} \cdot n\right) = \frac{1}{\sqrt{n}}.
$$
This proves the desired result.

For $\Delta_n^{(3)}$, as a direct consequence of~\citet[Lemma~D.4]{deb2020measuring}, we have that for some universal constant $C > 0$, $\E|\Delta_n^{(3)}|^2 \le C (\E[\tilde{F}(\by_1 \wedge \by_2)^2] + \theta^2) / n$, which proves the desired result.

\paragraph{Analysis of $S_n$} We first show that almost surely, with $n \ge 5$,
\begin{equation}\label{eq:sns1}
    \var(S_n \mid M) \ge \Gamma_1 / 2, \quad\&\quad \var(S_n \mid M) - \tilde{\sigma}_n^2 \to 0.
\end{equation}

To show the first result, we write
\[
S_n' = S_n + \Delta_n^{(3)},
\]
then it suffices to show $\var[S_n' | M] \ge \Gamma_1 /2$ a.s., since we already have $\var[S_n | M] = \var[S_n' | M] + \var[\Delta_n^{(3)}]$ from~\citet[Equation~(C.21)]{deb2020measuring}. 
Just as our analysis of $\Delta_n^{(1)}$, we redefine $A_{i,j} := \tilde{F}(\by_i \wedge \by_j) - \theta$ and define $T_0, T_1, T_2, S_0, S_1, S_2$ as our analysis of $\Delta_n^{(1)}$, but with the new $A_{i,j}$. Then we have 
\begin{equation}\label{eq:snt}
    (S_n')^2=T_2-2T_1+T_0.
\end{equation}

For $S_0, S_1, S_2$, it is straightforward that 
\begin{equation}\label{eq:s012}
    S_0 = 0, S_1 = \var[h(\by_1)], S_2 = \var[\tilde{F}(\by_1 \wedge \by_2)].
\end{equation}
We next consider $\E[T_0 | M]$. Apparently we still have
\[
\E[T_0 | M] = \frac{2}{n-1}S_2+\frac{4(n-2)}{n-1}S_1+\frac{(n-2)(n-3)}{n-1}S_0,
\]
since $T_0 \independent M$. For $\E[T_1 | M]$, using again that $\by_i$'s and $M$ are independent, and all the $\by_i$'s are i.i.d.,
\[
\E[T_1 | M] = \E\left[\frac{1}{(n - 1)} A_{1, 2} \left( \sum_{i \neq j} A_{i,j}\right)\right] = \E[T_0 | M].
\]
Thus, the only challenge is to understand $\E[T_2 | M]$. Let $\rho := [n] \to [n]$ be a completely random permutation that is independent of other randomness, and let $M' := \rho \circ M \circ \rho^{-1}$. This means that if $M$ is a mapping that takes $\by_j$ as the nearest neighbor of $\by_i$, then $M'$ takes $\by_{\rho(j)}$ as the nearest neighbor of $\by_{\rho(i)}$. Also, we define $\tilde{T}_2$ in the same way as before, but with $M$ replaced by $M'$. Then using the exchangeability of $\by_i$'s, 
\[
\E[T_2 | M] = \E[\tilde{T}_2 | M].
\]

We now investigate the distribution of $M'$ given $M$. Using the completely at random property of $\rho$, we have for any $i \neq j \neq k \neq l$, by defining
\begin{equation}
\label{def:hatab}
    \begin{aligned}
\hat{b}_n & := \frac{1}{n(n-1)} \sum_{i \neq j} \one\{M(i) = M(j)\}; \\
\hat{a}_n & := \frac{1}{n (n - 1)} \sum_{i \neq j} \one\{M(i) = j, M(j) = i\},
\end{aligned}
\end{equation}
it follows from the proof of~\Cref{lem:jointdis} that
\begin{align*}
\pr(M'(i) = k, M'(j) = k \mid M) & = \frac{\hat{b}_n}{n - 2};\\
\pr(M'(i) = j, M'(j) = i \mid M) & = \hat{a}_n; \\
\pr(M'(i) = j, M'(j) = k \mid M) &= \pr(M'(i) = k, M'(j) = i \mid M) = \frac{1 - \hat{a}_n (n - 1)}{(n - 1) ( n - 2)};\\
\pr(M'(i) = k, M'(j) = l \mid M) &= \frac{1 + \hat{a}_n - \hat{b}_n}{(n - 2)(n - 3)} - \frac{2}{(n - 1)(n - 2)(n - 3)}.
\end{align*}

Then it follows from the same analysis as our calculation of the term ``$\E[T_2]$'' in the analysis of $\Delta_n^{(1)}$ that
\begin{align*}
\E[T_2 \mid M] = \E[\tilde{T}_2 \mid M] = &  (1+(n-1)\hat{a}_n)S_2 + (2+(n-1)\hat{b}_n-2(n-1)\hat{a}_n)S_1 \\
&\qquad + (n-3+(n-1)\hat{a}_n-(n-1)\hat{b}_n)S_0.
\end{align*}

Putting together, noticing $\E[S_n' \mid M] = 0$ and recalling~\eqref{eq:snt},~\eqref{eq:s012}, we obtain
\[
\begin{aligned}
    \var(S_n' \mid M) & = \left(1+(n-1)\hat{a}_n-\frac{2}{n-1}\right) \var(\tilde{F}(\by_1 \wedge \by_2)) \\
    &\qquad + \left(\frac{4}{n-1}-2+(n-1)\hat{b}_n-2(n-1)\hat{a}_n\right)\var(h(\by_1)) \\
    & = \left(1+(n-1)\hat{a}_n-\frac{2}{n-1}\right) \Gamma_1 + (n - 1) \hat{b}_n \Gamma_2,
\end{aligned}
\]
which is at least $\Gamma_1 /2$ for $n \ge 5$.

Now for the second result in~\eqref{eq:sns1}, since we have already shown $\var[\Delta_n^{(3)}] \to 0$, we only need to show $(n - 1)\hat{a}_n - (n - 1)\hat{b}_n$ and $(n - 1) a_n - (n - 1) b_n$ converge to zero almost surely.

We first consider an equivalent characterization of the generating process of $M$. For each $i$, we associate it with a $\rho_i$, which is a uniformly at random permutation of the sequence $(1, \ldots, n)$. By denoting $\mathcal{N}(i)$ as the set of nearest neighbors of $\bz_i$, i.e.,
\[
\mathcal{N}(i) := \{j \neq i: \|\bz_j - \bz_i\| = \min_{k \neq i} \|\bz_k - \bz_i\|\},
\]
we can express
\[
M(i) := \underset{j \in \mathcal{N}(i)}{\argmin} \rho_i(j).
\]
In other words, $\hat{a}_n$ and $\hat{b}_n$ can be expressed as deterministic functions of $\bz_i$'s and $\rho_i$'s. We now discuss the change of the two terms when there is a single $\bz_i$ or $\rho_i$ replaced by a $\bz_i'$ or a $\rho_i'$. Below with a slight abuse of notation, we redefine $M'$ as the $M$ after this replacement.

\begin{enumerate}
    \item Change of $n(n - 1) \hat{a}_n$ and $n(n - 1) \hat{b}_n$ if $\rho_i$ is replaced by a $\rho_i'$: Apparently, only $M(i)$ is affected through this replacement. For $n(n - 1) \hat{a}_n$, we apply the decomposition
    \begin{align*}
    n(n - 1) \hat{a}_n = & \sum_{j: j \neq i} \one\{M(i) = j, M(j) = i\} + \sum_{j: j \neq i} \one\{M(j) = i, M(i) = j\} \\
    & \quad + \sum_{j, \ell: \; j, \ell \neq i \;\&\; j \neq \ell} \one\{M(j) = \ell, M(\ell) = j\}.
    \end{align*}
    Apparently, the replacement of $\rho_i$ only affects the first two terms, whose magnitudes are not larger than $1$. So in total, the replacement of $\rho_i$ alters $n(n - 1) \hat{a}_n$ by at most $2$. For $n(n - 1) \hat{b}_n$, analogously,
    \[
    n(n - 1) \hat{b}_n = \sum_{j: j \neq i} \one\{M(i) = M(j)\} + \sum_{j: j \neq i} \one\{M(j) = M(i)\} + \sum_{j, \ell: \; j, \ell \neq i \;\&\; j \neq \ell} \one\{M(j) = M(\ell)\}.
    \]
    Again, the replacement of $\rho_i$ only affects the first two terms. Since now the magnitudes of the two terms are no greater than $d_{\max}(M)$, we easily have that this replacement alters $n(n - 1) \hat{b}_n$ by at most $2 (d_{\max}(M) + d_{\max}(M'))$.
    
    \item  Change of $n (n - 1) \hat{a}_n$ if $\bz_i$ is replaced by a $\bz_i'$: first, we prove the claim: 
    \begin{equation}\label{eq:anzchange}
        \forall k \neq i, M(k) \neq M'(k) \Rightarrow M(k) = i \;\mathrm{or}\; M'(k) = i.
    \end{equation}
    Suppose in contradiction they are both not equal to $i$. Recall that $\mathcal{N}(k)$ can be re-expressed as 
    \[
    \mathcal{N}(k) := \{j \neq k: \|\bz_j - \bz_k\| = D(k)\} \quad\mathrm{where}\quad D(k) := \min_{j \neq k} \|\bz_j - \bz_k\|.
    \]
    Write $\mathcal{N}', D'$ as the corresponding terms when $\bz_i$ is replaced by a $\bz_i'$. If $ \Vert \bz_i'-\bz_k \Vert < D(k) $, then $ \mathcal{N}'(k)=\{i\}$, which further implies $ M'(k)=i $, contradicting with our assumption. So $ \Vert \bz_i'-\bz_k \Vert \geq D(k) $. The definitions of $ D,D' $ give $ D'(k) \geq D(k) $. Analogously, we also get $ D(k) \geq D'(k) $. Hence, $ D(k) = D'(k) $. Now, for all $ j \neq i,k $, we know $ \Vert \bz_j-\bz_k \Vert = D(k) $ if and only if $ \Vert \bz_j-\bz_k \Vert = D'(k) $, i.e. $ j \in \mathcal{N}(k) $ if and only if $ j \in \mathcal{N}'(k) $. Hence we have that either $\mathcal{N}(k)$ and $\mathcal{N}'(k)$ are equal, or they differ only by an element $i$. Apparently they cannot be equal, as otherwise $M(k) = M'(k)$. For the other case, we assume with loss of generality $i \in \mathcal{N}(k)$ while $i \not\in \mathcal{N}'(k)$. Writing $\ell := M(k)$, then apparently $\ell \in \mathcal{N}'(k)$; moreover, $\rho_k(\ell)$ is the smallest among all the $\rho_k(\ell')$'s for $ \ell' \in \mathcal{N}'(k)$. This means $M'(k) = \ell$, which raises a contradiction. This proves~\eqref{eq:anzchange}. 
    
    Through~\eqref{eq:anzchange}, the number $k$'s with $M(k) \neq M'(k)$ is at most $d_{\max}(M) + d_{\max}(M') + 1$. Then by our analysis of $n (n - 1) \hat{a}_n$ in Point 1, the change of $\bz_i$ alters $ n(n-1)\hat{a}_n $ by at most $2 (d_{\max}(M) + d_{\max}(M') + 1) \le 4 (d_{\max}(M) + d_{\max}(M'))$.

    \item Change of $n (n - 1) \hat{b}_n$ if $\bz_i$ is replaced by a $\bz_i'$: as discussed in Point 2, this replacement causes at most $4 (d_{\max}(M) + d_{\max}(M'))$ switches of $M$'s output. By our analysis of $n (n - 1) \hat{b}_n$ in Point 1, each such change alters $n (n - 1) \hat{b}_n$ by at most $2 (d_{\max}(M) + d_{\max}(M'))$. Putting together, the total change is bounded by $8(d_{\max}(M) + d_{\max}(M'))^2$.
\end{enumerate}

Write $\bz_1', \ldots, \bz_n'$ as an i.i.d. copy of $\bz_1, \ldots, \bz_n$ and $\rho_1', \ldots, \rho_n'$ as an i.i.d. copy $\rho_1, \ldots, \rho_n$. Write $\hat{a}_{n, i}', \hat{b}_{n, i}'$ as the $\hat{a}_n, \hat{b}_n$ if we replace $\bz_i$ with $\bz_i'$ and define $\hat{a}_{n, i}'', \hat{b}_{n, i}''$ analogously when we replace $\rho_i$ with $\rho_i'$. Then using a Generalized Efron-Stein Inequality established by~\citet[Theorem~2]{boucheron2005moment}, we have for some universal constant $C$ that varies from line to line,
\begin{align*}
    & \E |(n - 1) \hat{a}_n - (n - 1) \E[\hat{a}_n]|^4 \\
    & \quad \le C \E\left( \sum_{i = 1}^n \E[((n - 1) \hat{a}_{n,i}' - (n - 1) \hat{a}_n)^2 + ((n - 1) \hat{a}_{n,i}'' - (n - 1) \hat{a}_n)^2 \mid \bz_1, \ldots, \bz_n, \rho_1, \ldots, \rho_n]\right)^2.
\end{align*}
Using Points 1-3 discussed above and a Jensen's inequality, we further have
\begin{align*}
    \E |(n - 1) \hat{a}_n - (n - 1) \E[\hat{a}_n]|^4 & \le C \E \left( \E\left[\frac{d_{\max}(M)^2 + d_{\max}(M')^2}{n} \mid \bz_1, \ldots, \bz_n, \rho_1, \ldots, \rho_n\right]\right)^2 \\
    & \le C \E \left(\frac{d_{\max}(M)^2 + d_{\max}(M')^2}{n}\right)^2.
\end{align*}
Applying~\Cref{lem:maxdegree}, and the fact that $M$ and $M'$ are equal in distribution, we have 
\[
n^{3/2} \E |(n - 1) \hat{a}_n - (n - 1) \E[\hat{a}_n]|^4 \to 0.
\]
With this, the Borel-Cantelli lemma, and that $(n - 1) \E[\hat{a}_n] = (n - 1) a_n$, we have 
\begin{equation}\label{eq:anconv}
 (n - 1) \hat{a}_n - (n - 1) a_n \asconv 0.
\end{equation}
Applying again above Points 1-3, using an analogous argument, we may prove that 
\begin{equation}\label{eq:bnconv}
(n - 1) \hat{b}_n - (n - 1) b_n \asconv 0.    
\end{equation} 
In light of our analysis of $(n - 1) \hat{a}_n$, $(n - 1) \hat{b}_n$, we finish our proof of the second result in~\eqref{eq:sns1}.

Armed with the first result of~\eqref{eq:sns1}, we are now in a position to prove:
\begin{equation}\label{eq:sns2}
    \E\sup_{z}\left|\pr\left(\frac{S_n}{\sqrt{\var(S_n \mid M)}} \le z \mid M\right) - \Phi(z)\right|\to 0
\end{equation}

To prove~\eqref{eq:sns2}, we first recall~\citet[Theorem~2.7]{chen2004normal}: Consider $n$ random variables $U_1, \ldots, U_n$ whose dependency structure can be described by a dependency graph $G:=([n],\mathcal{E})$. Namely, for any disjoints sets $\mathcal{V}_1, \mathcal{V}_2 \subseteq [n]$, if for any $i \in \mathcal{V}_1, j \in \mathcal{V}_2$, the undirected edge $(i,j) \notin \mathcal{E}$, then the two sets of random variables $\{U_i, i\in \mathcal{V}_1\}$ and $\{U_i,i\in \mathcal{V}_2\}$ are independent. Assuming further that $\E[U_i] = 0$, $\E|U_i|^p \le \theta^p$ for some $p \in (2, 3]$, and writing $W := \sum_{i=1}^n U_i$, $\E[W^2] = 1$, we have
$$\sup_z \vert \pr(W\leq z)-\Phi(z)\vert \le 75 d(\mathcal{E})^{5(p-1)}n \theta^p,$$
where $d(\mathcal{E})$ is the maximal degree of $\mathcal{E}$.

Back to the proof of~\eqref{eq:sns2}. We may set $p = 3$,
\[
U_i := \frac{\frac{1}{\sqrt{n}}(\tilde{F}(\by_i \wedge \by_{M(i)})-2h(\by_i)+\theta)}{\sqrt{\var(S_n|M)}},
\]
and the edge set
$$\mathcal{E} :=\{\{i,M(i)\}: i \in [n]\} \cup \{\{j,k\}: M(j)=M(k), j,k \in [n]\}.$$

Then one can verify that $\E[U_i \mid M] = 0$, $\E[W^2 \mid M]= 1$, and from the first result of~\eqref{eq:sns1}, almost surely, $\E[|U_i|^3 \mid M] \le (2 / \sqrt{\Gamma_1 / (2n)})^3$. In addition, $\mathcal{E}$ is consistent with the dependency structure of $U_i$'s conditional on $M$. This allows us to apply~\citet[Theorem~2.7]{chen2004normal} to get that for some universal constant $C > 0$ that varies from line to line, almost surely,
$$\sup_z \left| \pr\left[ \frac{S_n}{\sqrt{\var[S_n|M]}} \leq z \mid M\right] - \Phi(z) \right|
\le C d(\mathcal{E})^{10} n \left(\frac{2}{\sqrt{\frac{\Gamma_1}{2} n}}\right)^3
\le \frac{C d_{\max}(M)^{10}}{\Gamma_1^{3/2} \sqrt{n}}.$$
Here the last inequality follows from
\[
    d(\mathcal{E}) \leq \max_k \left\{\sum_{j \neq k}\one\{M(j)=k\}+\sum_{j \neq k} \one\{M(j)=M(k)\} + 1\right\} \le 2 d_{\max}(M) + 1 \le 3 d_{\max}(M).
\]

Taking expectation on both sides and using~\Cref{lem:maxdegree}, we obtain
$$\E\sup_z \left| \pr\left( \frac{S_n}{\sqrt{\var[S_n|M]}} \leq z \mid M\right) - \Phi(z) \right|
\le \frac{C}{\Gamma_1^{3/2} \sqrt{n}} \E d_{\max}(M)^{10} \to 0,$$
which finishes the proof of~\eqref{eq:sns2}. 

In light of~\eqref{eq:sns2}, we can show that
\begin{align*}
& \sup_{z} \left|\pr\left(\frac{S_n}{\sqrt{\var(S_n \mid M)}} \le z \right) - \Phi(z)\right| = 
\sup_{z} \left|\E\left[\pr\left(\frac{S_n}{\sqrt{\var(S_n \mid M)}} \le z \mid M \right)\right] - \Phi(z)\right| \\
& \qquad \le
\E \sup_{z} \left|\pr\left(\frac{S_n}{\sqrt{\var(S_n \mid M)}} \le z \mid M \right) - \Phi(z)\right| \to 0,
\end{align*}
which proves that $\frac{S_n}{\sqrt{\var(S_n \mid M)}}$ converges in distribution to a standard normal random variable. Then as a direct consequence of the second result of~\eqref{eq:sns1} (which, together with $\liminf \tilde{\sigma}_n > 0$, yields $\sqrt{\var(S_n \mid M)} / \tilde{\sigma}_n \to 1$ a.s.) and Slutsky's theorem, we obtain the desired result.

\subsection{Proof of~\Cref{prop:estvar}}\label{sec:varest}

In this subsection, we prove~\Cref{prop:estvar}. First, recall $\tilde{F}_n(\bsy)$ and $G_n(\bsy)$ defined in~\eqref{def:tildeF} and~\eqref{def:G_n}, and $\hat{a}_n$ and $\hat{b}_n$ defined in~\eqref{def:hatab}. We may use these terms to reformulate $\hat{\Gamma}_1, \hat{\Gamma}_2, \hat{\sigma}_n^2$ as
$$\begin{aligned}
        \hat{\Gamma}_1&:= \frac{1}{n-1}\sum_{i = 1}^{n - 1}\left(\tilde{F}_n(\by_i \wedge \by_{i + 1})\right)^2-\left(\frac{1}{n-1}\sum_{i = 1}^{n - 1}\tilde{F}_n(\by_i \wedge \by_{i + 1})\right)^2-2\hat{\Gamma}_2;\\
        \hat{\Gamma}_2&:=\frac{1}{n - 2}\sum_{i = 1}^{n - 2}\tilde{F}_n(\by_i \wedge \by_{i + 1})\tilde{F}_n(\by_i \wedge \by_{i + 2})- \left(\frac{1}{n - 1}\sum_{i = 1}^{n - 1}\tilde{F}_n(\by_i \wedge \by_{i + 1})\right)^2,
    \end{aligned}$$
and
$$\hat{\sigma}_n^2:=\frac{\hat{\Gamma}_1(1+(n - 1)\hat{a}_n)+\hat{\Gamma}_2(n - 1)\hat{b}_n}{\left( \frac{1}{n} \sum_{k=1}^n (G_n(\tilde{\by}_k))(1-G_n(\tilde{\by}_{k}))\right)^2}.$$
The other reformulations are straightforward. For terms concerning $\hat{b}_n$, this is because
\[
\begin{aligned}
(n - 1) n \hat{b}_n & = \sum_{k = 1}^n \sum_{i \neq j} \one\{M(i) = M(j) = k\} = \sum_{k = 1}^n |\{(i,j): i \neq j \;\&\; i, j \in M^{-1}(k)\}|\\
& = \sum_{k = 1}^n |M^{-1}(k)|(|M^{-1}(k)| - 1).
\end{aligned}
\]

Our proof is based on these reformulations:
\begin{proof}
According to the proof of~\Cref{thm:multiacest},~\eqref{eq:anconv} and~\eqref{eq:bnconv}, we only need to prove that $\hat{\Gamma}_1, \hat{\Gamma}_2$ converge to $\Gamma_1, \Gamma_2$ almost surely. To prove this, we only need to prove the following three convergence results:
\begin{align*}
    & \frac{1}{n-1}\sum_{i = 1}^{n - 1}\left(\tilde{F}_n(\by_i \wedge \by_{i + 1})\right)^2 \asconv \E[\tilde{F}(\by_1 \wedge \by_2)^2], \; \frac{1}{n-1}\sum_{i = 1}^{n - 1}\tilde{F}_n(\by_i \wedge \by_{i + 1}) \asconv \E[\tilde{F}(\by_1 \wedge \by_2)],\\
    & \frac{1}{n - 2}\sum_{i = 1}^{n - 2}\tilde{F}_n(\by_i \wedge \by_{i + 1})\tilde{F}_n(\by_i \wedge \by_{i + 2}) \asconv \E[\tilde{F}(\by_1 \wedge \by_2) \tilde{F}(\by_1 \wedge \by_3)].
\end{align*}
Without loss of generality, we just prove the first one, and the other two can be derived via analogous arguments. Applying \Cref{lem:permcdf}, we have
\[
\frac{1}{n-1}\sum_{i = 1}^{n - 1}\left(\tilde{F}_n(\by_i \wedge \by_{i + 1})\right)^2 - \frac{1}{n-1}\sum_{i = 1}^{n - 1} \left(\tilde{F}(\by_i \wedge \by_{i + 1})\right)^2 \asconv 0.
\]
Then as a direct consequence of bounded difference inequality and Borel-Cantelli Lemma, we have
\[
\frac{1}{n-1}\sum_{i = 1}^{n - 1} \left(\tilde{F}(\by_i \wedge \by_{i + 1})\right)^2 \asconv \E \tilde{F}(\by_1 \wedge \by_2)^2,
\]
which proves the desired result.
\end{proof}

\bibliographystyle{plainnat}
\bibliography{reference}

\newpage
\appendix

\begin{center}
    {\huge Supplementary material for ``A multivariate extension of Azadkia-Chatterjee's rank coefficient''}
\end{center}

\section{Proof of additional results in~\Cref{sec:normal}}

\subsection{Proof of~\Cref{thm:limit}}

\Cref{thm:limit} is a direct consequence of~\Cref{lem:limit}. In this proof, we write $C(\dz)$ as a constant depending only on $\dz$ that may vary from line to line, and write $M^{-1}(i) := \{j: M(j) = i\}$.

\begin{lemma}\label{lem:nngmm} % nngmm: Nearest Neighbor Graph on Mixed Measure
Consider a $\bz$ with distribution $\mu_{\bz}$. Let $S_D := \{\bsz \in \R^{\dz}: \mu_{\bz}(\{\bsz\})>0\}, \eta := \mu_{\bz}(S_D)$ and $N := \sum_{j=2}^n \one\{\bz_j = \bz_1\}$. Then, we have the following results:
\begin{align}
& \pr(\bz_1 \in S_D, N=0) \to 0 \quad\&\quad \E\left[\frac{1}{N} \one\{\bz_1 \in S_D, N \geq 1\}\right] \to 0, \label{eq:nngmmn} \\
& \E[|M^{-1}(1)| \one\{ \bz_1 \in S_D\}] \to \eta \quad\&\quad \E(|M^{-1}(1)|^2 \one\{\bz_1 \in S_D\}) \to 2 \eta, \label{eq:nngmmminv} \\
& \pr(\bz_1 \in S_D, \bz_{M(1)} \in S_D^c) \to 0 \quad\&\quad \pr(\bz_1 \in S_D^c, \bz_{M(1)} \in S_D) \to 0. \label{eq:nngmmnncd}
\end{align}
% nngmmn: the N in nngmm. nngmmminv: the M_inverse in nngmm. nngmmnncd: the Nearest Neighbor in C or D in nngmm
\end{lemma}

\begin{proof}
We first prove~\eqref{eq:nngmmn}. Fix any $\bsz \in S_D$, by definition $\mu_{\bz}(\{\bsz\}) > 0$. Conditioning on $\bz_1=\bsz$, the random variables $\one\{\bz_j=\bsz\}, j=2,\ldots,n$ follow i.i.d. $\mathrm{Bernoulli}(\mu_{\bz}(\{\bsz\}))$, so from the strong law of large numbers we know $\frac{1}{n-1} N = \frac{1}{n-1} \sum_{j=2}^n \one\{\bz_j=\bsz\} \asconv \mu_{\bz}(\{\bsz\}) > 0$. This implies $N \asconv \infty$ and hence gives $\one\{N = 0\} \asconv 0$, $\one\{N \geq 1\} / N \asconv 0$, conditioned on $\bz_1=\bsz$. Notice that we also have $\one\{N = 0\}, \one\{N \geq 1\} / N \leq 1$ conditioned on $\bz_1=\bsz$, so by the dominated convergence theorem,
\[
\begin{aligned}
& \pr(N = 0 \mid \bz_1 = \bsz) = \E(\one\{N = 0\} \mid \bz_1 = \bsz) \to 0, \\
& \E(\one\{ N \geq 1\} / N \mid \bz_1 = \bsz) = \E(\one\{N \geq 1\} / N \mid \bz_1 = \bsz) \to 0.
\end{aligned}
\]
Moreover, since $\pr(N=0 \mid \bz_1=\bsz), \E(\one\{ N \geq 1\} / N \mid \bz_1=\bsz) \leq 1$ for any $\bsz \in S_D$, applying again the dominated convergence theorem yields
\[
\begin{aligned}
& \pr(N=0, \bz_1 \in S_D) = \int_{\bsz \in S_D} \pr(N = 0 \mid \bz_1 = \bsz) d\mu_{\bz}(\bsz) \to 0, \\
& \E(\one\{N \geq 1, \bz_1 \in S_D\} / N) = \int_{\bsz \in S_D} \E(\one\{N \geq 1\} / N \mid \bz_1 = \bsz) d\mu_{\bz}(\bsz) \to 0,
\end{aligned}
\]
which proves the desired result.

    Next, we prove~\eqref{eq:nngmmminv}. We apply the decomposition
    \[
    M^{-1}(1) = \underset{=: \cI_1}{\underbrace{\{i: \bz_i \neq \bz_1 \;\&\; M(i)\ = 1\}}} \cup \underset{=: \cI_2}{\underbrace{\{i: \bz_i = \bz_1 \;\&\; M(i)\ = 1\}}},
    \]
    where the two sets $\cI_1$ and $\cI_2$ are disjoint.
    For $\cI_1$, we have
    \[
    \E[|\cI_1| \one\{\bz_1 \in S_D\}] = \E[|\cI_1| \one\{\bz_1 \in S_D, N = 0\}] + \E[|\cI_1| \one\{\bz_1 \in S_D, N \ge 1\}].
    \]
    From~\citet[Lemma~11.4]{azadkia2021simple}, $|\cI_1| \leq C(\dz)$. Then using the first part of~\eqref{eq:nngmmn}, the first term can be controlled by $C(\dz) \pr(\bz_1 \in S_D, N = 0) \to 0$, so we just need to focus on the second term. Following the random tie-breaking mechanism, we have that by conditioning on a $\{\bz_i\}_{i = 1}^n$ satisfying that $N \ge 1$, if there exists a $\bz_j \neq \bz_1$ while $1 \in \{i: \|\bz_i - \bz_j\| = \min_{i' \neq j} \|\bz_{i'} - \bz_j\|\}$, then the probability that $M(j) = 1$ is at most $1 / (N + 1)$. From~\citet[Lemma~11.4]{azadkia2021simple} we have that there are at most $C(\dz)$ such $\bz_j$. In light of both, and the second part of~\eqref{eq:nngmmn}, we have
    \[
    \E[|\cI_1| \one\{\bz_1 \in S_D, N \ge 1\}] \le \E\left[\frac{C(\dz)}{N} \one\{\bz_1 \in S_D, N \geq 1\}\right] \to 0.
    \]
    Putting together yields $\E[|\cI_1| \one\{\bz_1 \in S_D\}] \to 0$. Moreover, since $|\cI_1| \leq C(\dz)$, we also have
    \begin{equation}\label{eq:i1sq}
    \E[|\cI_1|^2 \one\{\bz_1 \in S_D\}] \leq C(\dz) \E[|\cI_1| \one\{\bz_1 \in S_D\}] \to 0.
    \end{equation}

    For $\cI_2$, conditioning on a $\{\bz_i\}_{i=1}^n$ satisfying that $\bz_1 \in S_D$ and $N \ge 1$, $|\cI_2|$ follows a Binomial distribution with parameters $(N, 1 / N)$. Then we have $\E[|\cI_2| \mid \{\bz_i\}_{i=1}^n] = 1$ and $\E[|\cI_2|^2 \mid \{\bz_i\}_{i=1}^n] = 2 - 1 / N$. If $N = 0$, then $\cI_2$ must be an empty set. Therefore, using~\eqref{eq:nngmmn} we have
    \begin{align*}
    \E[|\cI_2| \one\{\bz_1 \in S_D\}] & = \E[|\cI_2| \one\{\bz_1 \in S_D, N \ge 1\}] = \pr(\bz_1 \in S_D, N \ge 1) \to \pr(\bz_1 \in S_D) = \eta, \\
    \E[|\cI_2|^2 \one\{\bz_1 \in S_D\}] & = \E[|\cI_2|^2 \one\{\bz_1 \in S_D, N \ge 1\}] = 2 \pr(\bz_1 \in S_D, N \ge 1)  - \E\left[\frac{1}{N} \one\{\bz_1 \in S_D, N \ge 1\}\right] \\
    & \to 2 \pr(\bz_1 \in S_D) = 2 \eta.
    \end{align*}
    Moreover, using Cauchy-Schwartz inequality with~\eqref{eq:i1sq} and the above limit, we have
    \[
    \E[|\cI_1| \cdot |\cI_2| \one\{\bz_1 \in S_D\}] \leq \sqrt{\E[|\cI_1|^2 \one\{\bz_1 \in S_D\}]} \sqrt{\E[|\cI_2|^2 \one\{\bz_1 \in S_D\}]} \to 0.
    \]
    In light of our control of $|\cI_1|$ and $|\cI_2|$, we prove~\eqref{eq:nngmmminv}.

Finally, we prove~\eqref{eq:nngmmnncd}.
For the first part of~\eqref{eq:nngmmnncd}, notice that $(\bz_1 \in S_D, \bz_{M(1)} \in S_D^c)$ implies $\bz_{M(1)} \neq \bz_1$, which further implies $N=0$. So, the first part of~\eqref{eq:nngmmnncd} directly follows from the first part of~\eqref{eq:nngmmn}:
\[
\pr(\bz_1 \in S_D, \bz_{M(1)} \in S_D^c) \leq \pr(\bz_1 \in S_D, N=0) \to 0.
\]
For the second part of~\eqref{eq:nngmmnncd}, rewrite the probability:
\[
\pr(\bz_1 \in S_D^c, \bz_{M(1)} \in S_D) = \pr(\bz_{M(1)} \in S_D) - \pr(\bz_1 \in S_D, \bz_{M(1)} \in S_D)
\]
The first term can be analyzed using the i.i.d. property and the first part of~\eqref{eq:nngmmminv}:
\[
\begin{aligned}
\pr(\bz_{M(1)} \in S_D) &= \sum_{k \neq 1} \pr(M(1)=k, \bz_k \in S_D) = (n-1) \pr(M(1)=2, \bz_2 \in S_D) \\
&= \E\left[\sum_{j \neq 2} \one\{M(j)=2\} \one\{\bz_2 \in S_D\}\right] = \E\left[|M^{-1}(2)| \one\{ \bz_2 \in S_D\}\right] \to \eta.
\end{aligned}
\]
The second term can be analyzed using the first part of~\eqref{eq:nngmmnncd}, which we have just proved before:
\[
\pr(\bz_1 \in S_D, \bz_{M(1)} \in S_D) = \pr(\bz_1 \in S_D)- \pr(\bz_1 \in S_D, \bz_{M(1)} \in S_D^c) \to \pr(\bz_1 \in S_D) = \eta.
\]
Putting together proves the second part of~\eqref{eq:nngmmnncd}.
\end{proof}

\begin{lemma}\label{lem:limit}
	Consider the set up of~\Cref{thm:limit} and recall the $a_n, b_n$ defined in~\Cref{lem:probneigh}. We have
	\[
	\lim_{n \to \infty} (n - 1) a_n = (1 - \eta) A_{\dz} \quad \& \quad \lim_{n \to \infty} (n - 1) b_n = (1 - \eta) B_{\dz} + \eta.
	\]
\end{lemma}

\begin{proof}
    In this proof, we adopt the notations introduced in Lemma~\ref{lem:nngmm}. In addition, we define
\[
N_a := \sum_j \one\{\bz_j \in S_D^c\}
\quad\&\quad
M_a(k) := \argmin_{j: j \neq k, \bz_j \in S_D^c} \Vert \bz_j - \bz_k \Vert
\]
Notice that in $M_a$ there are almost surely no ties, so $M_a(k)$ is well-defined as long as $\{j: j \neq k, \bz_j \in S_D^c\}$ is non-empty. If it is empty, define $M_a(k):=0$. The definition of $M_a$ immediately gives us a useful relation:
\begin{equation}\label{eq:maiff}
M_a(1)=M(1) \Leftrightarrow \bz_{M(1)} \in S_D^c \qquad\text{a.s.}
\end{equation}
% We prove the $a_n$-part and the $b_n$-part of this lemma.

\paragraph{Proof of the $a_n$-part} Write $(n - 1) a_n = I_n^{(1)} + I_n^{(2)}$, where
\[
\begin{aligned}
I_n^{(1)} &:= (n - 1) \pr(M(1)=2, M(2)=1; \bz_1  \in S_D \text{ or } \bz_2 \in S_D) \\
I_n^{(2)} &:= (n - 1) \pr(M(1)=2, M(2)=1; \bz_1,\bz_2 \in S_D^c).
\end{aligned}
\]
To control $I_n^{(1)}$, we easily have
\[
\begin{aligned}
I_n^{(1)} &\leq 2 (n - 1) \pr(M(1)=2, M(2)=1, \bz_1 \in S_D) = 2 (n - 1) \pr(M(1)=2, M(M(1))=1, \bz_1 \in S_D) \\
&= 2 \sum_{j \neq 1} \pr(M(1)=j, M(M(1))=1, \bz_1 \in S_D) = 2 \pr(M(M(1))=1, \bz_1 \in S_D).
\end{aligned}
\]
Notice that when $N \geq 1$, conditioning on $N$, the event $\{M(M(1))=1\}$ happens with probability $1/N$; with this, and~\eqref{eq:nngmmn}, we have
\[
\begin{aligned}
I_n^{(1)} \le &\; 2 \pr(M(M(1))=1, \bz_1 \in S_D, N=0) + 2 \pr(M(M(1))=1, \bz_1 \in S_D, N \geq 1) \\
\leq&\; 2 \pr(\bz_1 \in S_D, N=0) + 2 \E(\one\{\bz_1 \in S_D, N \geq 1\} / N) \to 0.
\end{aligned}
\]

For $I_n^{(2)}$, we first approximate it by
\[
I_n' := (n - 1) \pr(M_a(1)=2, M_a(2)=1; \bz_1,\bz_2 \in S_D^c).
\]
On one hand,~\eqref{eq:maiff} gives $I_n^{(2)} \leq I_n'$. On the other hand, the union bound, the i.i.d. property, the~\eqref{eq:maiff}, and the second part of~\eqref{eq:nngmmnncd}, give
\[
\begin{aligned}
I_n' - I_n^{(2)} &= (n - 1) \pr(M_a(1)=2, M_a(2)=1; M(1) \neq 2 \;\mathrm{ or }\; M(2) \neq 1; \bz_1,\bz_2 \in S_D^c) \\
&\leq 2 (n - 1) \pr(M_a(1)=2, M_a(2)=1; M(1) \neq 2; \bz_1,\bz_2 \in S_D^c) \\
&\leq 2 (n - 1) \pr(M_a(1)=2, M_a(1) \neq M(1), \bz_1 \in S_D^c) \\
& = 2 \sum_{j=2}^n \pr(M_a(1)=j, M_a(1) \neq M(1), \bz_1 \in S_D^c) \\
&\leq 2 \pr(M_a(1) \neq M(1), \bz_1 \in S_D^c) = 2 \pr(\bz_{M(1)} \in S_D, \bz_1 \in S_D^c) \to 0.
\end{aligned}
\]

% \textcolor{red}{I got until here.}

Now we only need to focus on $I_n'$. Conditioning on an $N_a \ge 2$ and $\{\bz_1,\bz_2 \in S_D^c\}$, the probability of $\{M_a(1)=2,M_a(2)=1\}$ is exactly $a_{N_a}^{(a)}$: here $a_n^{(a)}$ stands for the $a_n$ with the underlying distribution replaced by $\mu_{\bz, a}$. And, conditioning on $N_a$, the probability of $\{\bz_1,\bz_2 \in S_D^c\}$ is $\frac{N_a(N_a-1)}{n(n-1)}$. So we have
\[
I_n' = (n-1) \E[\pr(M_a(1)=2, M_a(2)=1 ; \bz_1,\bz_2 \in S_D^c \mid N_a) \one\{N_a \geq 2\}] = \E\left[\frac{N_a}{n}  (N_a-1) a_{N_a}^{(a)}\right].
\]
Now~\Cref{lem:probneigh} gives $\frac{N_a}{n}  (N_a-1) a_{N_a}^{(a)} \le 1$; and the strong law of large number gives $N_a/n \asconv 1 - \eta$. Now if $\eta < 1$, then $N_a \asconv \infty$ , so \citet[Lemma~3.6]{shi2024azadkia} gives $(N_a-1) a_{N_a}^{(a)} \asconv A_{\dz}$, which means $\frac{N_a}{n}  (N_a-1) a_{N_a}^{(a)} \asconv (1 - \eta) A_{\dz}$; so by the dominated convergence theorem, $I_n' \to (1 - \eta) A_{\dz}$. Otherwise, we have $N_a = 0$ almost surely, so we still have $I_n' = 0 = (1 - \eta) A_{\dz}$. 

In light of our control of $I_n^{(1)}, (I_n'-I_n^{(2)})$ and $I_n'$, we prove the desired result.

\paragraph{Proof of the $b_n$-part} By the i.i.d. property,
\[
(n - 1) b_n = (n - 1) \sum_{k \neq 1,2} \pr(M(1)=M(2)=k) = (n - 1) (n-2) \pr(M(1)=M(2)=3).
\]
Write $(n - 1) b_n = J_n^{(1)} + J_n^{(2)} + J_n^{(3)}$, where
\[
\begin{aligned}
J_n^{(1)} &:= (n - 1) (n-2) \pr(M(1)=M(2)=3, \bz_3 \in S_D) \\
J_n^{(2)} &:= (n - 1) (n-2) \pr(M(1)=M(2)=3; \bz_1 \in S_D \;\mathrm{or}\; \bz_2 \in S_D; \bz_3 \in S_D^c) \\
J_n^{(3)} &:= (n - 1) (n-2) \pr(M(1)=M(2)=3; \bz_1,\bz_2,\bz_3 \in S_D^c).
\end{aligned}
\]
We analyze $J_n^{(1)}$, $J_n^{(2)}$ and $J_n^{(3)}$, respectively.

For $J_n^{(1)}$, \eqref{eq:nngmmminv} yields
\[
J_n^{(1)} = \E\left[\sum_{i \neq j; i, j \neq 3} \one\{M(i)=M(j)=3, \bz_3 \in S_D\}\right] = \E(|M^{-1}(3)|(|M^{-1}(3)| - 1) \one\{\bz_3 \in S_D\}) \to \eta.
\]

For $J_n^{(2)}$, we have
\[
\begin{aligned}
J_n^{(2)} & \leq 2 (n - 1)(n - 2) \pr(M(1)=M(2)=3; \bz_1 \in S_D, \bz_3 \in S_D^c) \\
&= 2 (n - 1)(n - 2) \pr(M(1)=M(2)=3; \bz_1 \in S_D, \bz_{M(1)} \in S_D^c) \\
&= 2 \E\left[\sum_{j \neq k; j, k \neq 1} \one\{M(1)=M(j)=k, \bz_1 \in S_D, \bz_{M(1)} \in S_D^c\}\right] \\
&\le 2 \E(|M^{-1}(M(1))| \one\{ \bz_1 \in S_D, \bz_{M(1)} \in S_D^c\}).
\end{aligned}
\]
Under the event $\bz_{M(1)} \in S_D^c$, almost surely, $\bz_{M(1)}$ is not equal to any other $\bz_i$, which, together with~\citet[Lemma~11.4]{azadkia2021simple}, implies that almost surely, $|M^{-1}(M(1))| \le C(\dz)$. This, together with the first part of~\eqref{eq:nngmmnncd}, allows us to further get
\[
J_n^{(2)} \le 2 C(\dz) \pr(\bz_1 \in S_D, \bz_{M(1)} \in S_D^c) \to 0.
\]

For $J_n^{(3)}$, we approximate it by
\[
J_n' := (n - 1) (n - 2) \pr(M_a(1)=M_a(2)=3; \bz_1,\bz_2,\bz_3 \in S_D^c).
\]
On one hand,~\eqref{eq:maiff} gives $J_n^{(3)} \leq J_n'$. On the other hand, the union bound, the~\eqref{eq:maiff}, the i.i.d. property, give
\[
\begin{aligned}
J_n' - J_n^{(3)} &= (n - 1) (n - 2) \pr(M_a(1)=M_a(2)=3; M(1) \neq 3 \text{ or } M(2) \neq 3; \bz_1,\bz_2,\bz_3 \in S_D^c) \\
&\leq 2 (n - 1) (n - 2) \pr(M_a(1)=M_a(2)=3; M(1) \neq 3; \bz_1,\bz_2,\bz_3 \in S_D^c) \\
&\leq 2 (n - 1) (n - 2) \pr(M_a(1)=M_a(2)=3; \bz_{M(1)} \in S_D; \bz_1, \bz_2 \in S_D^c) \\
&= 2 \E\left[ \sum_{j \neq 1} \sum_{k \not\in \{0, 1, j\}} \one\{M_a(1)=M_a(j)=k; \bz_{M(1)} \in S_D; \bz_1,\bz_j \in S_D^c \} \right] \\
&= 2 \E\left[ \sum_{j \neq 1} \one\{M_a(j)=M_a(1) \neq 0, \bz_j \in S_D^c\} \one\{\bz_{M(1)} \in S_D; \bz_1 \in S_D^c\} \right].
\end{aligned}
\]
Applying~\citet[Lemma~11.4]{azadkia2021simple}, we have on the event $\{\bz_1 \in S_D^c\}$, almost surely, $\sum_{j \neq 1} \one\{M_a(j)=M_a(1) \neq 0, \bz_j \in S_D^c\} \leq C(\dz)$. So we get
\[
J_n' - J_n^{(3)} \le 2C(\dz) \pr(\bz_{M(1)} \in S_D, \bz_1 \in S_D^c) \to 0,
\]
where the last step comes from the second part of~\eqref{eq:nngmmnncd}.

So now we only need to focus on $J_n'$. First, by the i.i.d. property,
\[
J_n' = (n - 1) \sum_{j=3}^n \pr(M_a(1)=M_a(2)=j; \bz_1,\bz_2 \in S_D^c) = (n - 1) \pr(M_a(1)=M_a(2) \neq 0; \bz_1,\bz_2 \in S_D^c).
\]
Now, observe that, conditioning on a $N_a \ge 2$ and $\{\bz_1,\bz_2 \in S_D^c\}$, the probability of $\{M_a(1)=M_a(2) \neq 0\}$ is exactly $b_{N_a}^{(a)}$ -- here $b_n^{(a)}$ stands for the $b_n$ with the underlying distribution replaced by $\mu_{\bz, a}$. And, conditioning on $N_a$, the probability of $\{\bz_1,\bz_2 \in S_D^c\}$ is $\frac{N_a(N_a-1)}{n(n-1)}$. So, we have
\[
J_n' = (n - 1) \E[\pr(M_a(1)=M_a(2); \bz_1,\bz_2 \in S_D^c \mid N_a)\one\{N_a \ge 2\}] = \E\left[\frac{N_a(N_a-1)}{n} b_{N_a}^{(a)}\right].
\]
Now~\Cref{lem:probneigh} gives $\frac{N_a(N_a-1)}{n} b_{N_a}^{(a)} \le C(\dz)$. Following exactly the same analysis as our control of $I_n'$, we have $ J_n' \to (1 - \eta) B_{\dz}$.

Combining our analysis of $J_n^{(1)}, J_n^{(2)}, (J_n'-J_n^{(3)})$ and $J_n'$, the desired result is proved.
\end{proof}

\subsection{Proof of~\Cref{thm:nolimit}}

\noindent{\bf Notations.} First introduce some notations that will be used only in this subsection. Here $\dz=1$, so we write $\bsz, \bz, \Vert\cdot\Vert$ as $z, Z, |\cdot|$, respectively. When $a$ is a complex number, $|a|$ represents the modulus of $a$. For $c \in \R$ and $r > 0$, define $B(c,r) := \{x \in \R: |x-c| < r\}$ as the open ball centered at $c$ with radius $r$. For any $z \in [0,1)$, write its ternary expansion as $(0.z^{(1)}z^{(2)}\cdots)_3$, where $z^{(j)} := \lfloor 3^j z \rfloor \bmod 3$. So $z$ can be written as $z = \sum_{j=1}^\infty 3^{-j} z^{(j)}$. Under the notation of the ternary expansion, the Cantor ternary set is given as $\cC := \{z \in [0,1): \forall j, z^{(j)} \in \{0,2\} \}$, and our random variable $Z$ that follows the Cantor distribution can be expressed as $Z = (0.Z^{(1)}Z^{(2)}\cdots)_3$ where $\{Z^{(j)}\}_{j=1}^{\infty}$ follow i.i.d. $\mathcal{U}(\{0,2\})$.

Let $F_Z$ denote the CDF of $Z$. Notice that $\forall z \in [0,1): \pr(Z=z) = \prod_{j=1}^{\infty} \pr(Z^{(j)}=z^{(j)}) \leq \prod_{j=1}^{\infty} \frac12 = 0$, so $F_Z$ is continuous. Also, given a positive sequence $a_1, a_2, \ldots$, we write $O(a_n)$ to represent a sequence $b_n$ such that $\limsup |b_n| / a_n < \infty$, and write $o(a_n)$ when instead $\lim |b_n| / a_n = 0$. Given a set $A \subseteq \R$, which can be either finite or with positive and finite Lebesgue measure, we write $V \sim \mathcal{U}(A)$ if $V$ is a random variable following uniform distribution on $A$.

\begin{lemma}\label{lem:expexp}
Consider a random variable $X \in [0,1]$, an event $\mathcal{Q}$, and a constant $c \in \R$. We have as $n \to \infty$,
\[
\E((1-X)^{n+c} \one\{ \mathcal{Q}\}) = (1 + O((\ln n)^4 / n)) \E(e^{- n X} \one\{ \mathcal{Q}\}) + O(e^{-(\ln n)^2}).
\]
\end{lemma}

\begin{proof}
We first study the function $f_n(x) := (1-x)^{n+c} - e^{- n x}, x \in [0,1]$. When $x \leq (\ln n)^2 / n$, we have
\[
(1-x)^{n+c}
= e^{(n+c) \ln(1-x)}
= e^{(n+c)(-x+O(x^2))}
= e^{- n x + O(\ln^4 n / n)}
= e^{O(\ln^4 n / n)} e^{- n x},
\]
where ``$\ln^4 n$'' means $(\ln n)^4$.
So $|f_n(x)| = |e^{O(\ln^4 n / n)} - 1| e^{- n x} = O(\ln^4 n / n) e^{- n x}$. When $x > \ln^2 n / n$, we have
\[
(1-x)^{n+c}
\leq \left(1-\frac{\ln^2 n}{n}\right)^{n+c}
= e^{(n+c)\ln\left(1-\frac{\ln^2 n}{n}\right)}
\leq e^{(n+c)\left(-\frac{\ln^2 n}{n}\right)}
= O(e^{- \ln^2 n}),
\]
so $|f_n(x)| \leq (1-x)^{n+c} + e^{- n x} = O(e^{- \ln^2 n})$. Combining two parts, we get
\[
\begin{aligned}
&\; |\E((1-X)^{n+c} \one\{\mathcal{Q}\}) - \E(e^{- n X} \one\{\mathcal{Q}\})| = |\E(f_n(X) \one\{\mathcal{Q}\})| \leq \E(|f_n(X)| \one\{\mathcal{Q}\}) \\
=&\; \E(|f_n(X)| \one\{ \mathcal{Q}, X \leq \ln^2 n / n\}) + \E(|f_n(X)| \one\{ \mathcal{Q}, X > \ln^2 n / n\})
\leq O(\ln^4 n / n) \E(e^{- n X} \one\{ \mathcal{Q}\}) + O(e^{- \ln^2 n}),
\end{aligned}
\]
which provides the desired result.
\end{proof}

\begin{lemma}\label{lem:limitpoint}
Given a sequence $\{x_n\}_{n \geq 1} \in \R$ and a point $x \in \R$, we call $x$ a subsequential limit of the sequence $\{x_n\}_{n \geq 1}$ if there exists a subsequence $x_{n_k} \to x$. We have the following results.
\begin{itemize}
    \item[(i)] Let $\theta$ be any irrational number. Then, $0$ is a subsequential limit of $\{(n \theta) - \lfloor n \theta \rfloor\}_{n \geq 1}$.
    \item[(ii)] Every point in $[1,2)$ is a subsequential limit of $\{n / 2^{\lfloor \log_2 n \rfloor}\}_{n \geq 1}$.
\end{itemize}
\end{lemma}

\begin{proof}
First, we prove (i). By taking $\alpha=\epsilon<1/2$ in~\citet[Theorem~438, Chapter~23]{hardy2008number}, we get that $\forall \epsilon \in (0,1/2), N>0$, there exists an $n>N$ such that $n\theta \in \cup_{p \in \mathbb{Z}} (p,p+2\epsilon)$, so $x_n := (n\theta) - \lfloor n \theta \rfloor \in (0,2\epsilon)$. So, we can recursively find $n_k>n_{k-1}$ such that $x_{n_k} \in (0,2/k)$, hence providing a subsequence $x_{n_k} \to 0$. Therefore, $0$ is a subsequential limit of $\{(n \theta) - \lfloor n \theta \rfloor\}_{n \geq 1}$.

Second, we prove (ii). Fix arbitrary $x \in [1,2)$, and consider an $n_k := \lfloor 2^k x \rfloor$. We have
\[
2^k - 1 \leq 2^k x - 1 < n_k \leq 2^k x < 2^{k+1},
\]
so $\lfloor \log_2 n_k \rfloor = k$. Now, we have
\[
\frac{n_k}{2^{\lfloor \log_2 n_k \rfloor}} = \frac{\lfloor 2^k x \rfloor}{2^k} \in (x-2^{-k},x],
\]
hence $n_k / 2^{\lfloor \log_2 n_k \rfloor} \to x$. Therefore, every point $x \in [1,2)$ is a subsequential limit of $\{n / 2^{\lfloor \log_2 n \rfloor}\}_{n \geq 1}$.
\end{proof}

\begin{lemma}\label{lem:cantorrefl}
$F_Z(\cdot)$ satisfies the following properties:
    \begin{itemize}
        \item[(i)] for any $z \in [0, 1)$, $F_Z(z)$ can be expressed as
        \[
        F_Z(z) = \sum_{j=1}^r \frac{z^{(j)}/2}{2^j} + \frac{1}{2^{r+1}},
        \]
        where $r := \min\{j \geq 1: z^{(j)} = 1\} - 1$ if $z \not\in \cC$, and $r := \infty$ otherwise;
        \item[(ii)] for any $r \geq 0$ and $z \in \cC$ satisfying $3^r z \in \mathbb{Z}$, we have
        \[
        F_Z(z - 3^{-r}) = F_Z(z) \quad\&\quad F_Z(z + 3^{-r}) = F_Z(z + 2 \cdot 3^{-r}) = F_Z(z) + 2^{-r};
        \]
        \item[(iii)] for any $x,y \in \cC$ such that $x<y$,  
    \[
    F_Z(2x-y) = F_Z(z) \quad\&\quad F_Z(2y-x)=F_Z(z)+2^{-r},
    \]
    \end{itemize}
    where $r := \min\{j \geq 1: x^{(j)} \neq y^{(j)}\} - 1$ and $z := (0.x^{(1)} \cdots x^{(r)} 0 0 \cdots)_3$.
\end{lemma}

\begin{proof}[Proof of~\Cref{lem:cantorrefl}(i)]
    We consider $z \in \cC$, and $z \in [0,1) \setminus \cC$, respectively. For $z \in \cC$, consider the mapping $\phi(z) := \sum_{j = 1}^\infty 2^{-j} z^{(j)} / 2$. Then apparently $\phi$ is strictly increasing. Also, the random variable $U := \phi(Z) \in [0,1]$ satisfies $U \sim \mathcal{U}([0,1])$. In light of this, we have
\begin{equation}\label{eq:cantorcdfin}
F_Z(z) = \pr(Z \leq z) = \pr(\phi(Z) \leq \phi(z)) = \pr(U \leq \phi(z)) = \phi(z) = \sum_{j=1}^{\infty} \frac{z^{(j)}/2}{2^j}.
\end{equation}

For $z \in [0,1) \setminus \cC$, apparently $z^{(r + 1)} = 1$ and all $z^{(1)}, \ldots, z^{(r)} \in \{0, 2\}$. Consider the number $z' \in \cC$ obtained from $z$ by replacing $z^{(r+1)}$ with $2$ and $z^{(r+2)},z^{(r+3)},\cdots$ with $0$ in its ternary expansion. This means that
\[
\begin{aligned}
z &= (0\;.\;z^{(1)}\;\cdots\;z^{(r)}\;1\;z^{(r+2)}\;z^{(r+3)}\;\cdots)_3, \\
z' &= (0\;.\;z^{(1)}\;\cdots\;z^{(r)}\;2\;0\;0\;\cdots)_3,
\end{aligned}
\]
where $z^{(1)},\cdots,z^{(r)} \in \{0,2\}$. We have that $z \leq z'$ and that $z < Z < z'$ implies $Z^{(r+1)} = 1$. So,
\[
0 \leq F_Z(z') - F_Z(z) = \pr(z < Z < z') \leq \pr(Z^{(r+1)}=1) \leq \pr(Z \not\in \cC) = 0.
\]
Therefore, it follows from~\eqref{eq:cantorcdfin} that
\[
F_Z(z) = F_Z(z') = \sum_{j=1}^r \frac{z^{(j)}/2}{2^j} + \frac{1}{2^{r+1}}
\]
where $r = \min\{j \geq 1: z^{(j)}=1\} - 1$.
\end{proof}

\begin{proof}[Proof of \Cref{lem:cantorrefl}(ii)]
    The ternary expansion of every number we will meet at this part has zeros starting from the $(r+1)$-th digit, so we only list the first $r$ digits. To prove the first equation, suppose the last $2$ in the ternary expansion of $z$ appears at the $(s+1)$-th digit, then
\[
\begin{aligned}
z &= (0\;.\;z^{(1)}\;\cdots\;z^{(s)}\;2\;0\;\cdots\;0)_3, \\
z - 3^{-r} &= (0\;.\;z^{(1)}\;\cdots\;z^{(s)}\;1\;2\;\cdots\;2)_3.
\end{aligned}
\]
Applying~\Cref{lem:cantorrefl}(i) to $F_Z(z)$ and $F_Z(z-3^{-r})$, we obtain
\[
\begin{aligned}
F_Z(z) &= \sum_{j=1}^s \frac{z^{(j)}/2}{2^j} + \frac{1}{2^{s+1}}, \\
F_Z(z-3^{-r}) &= \sum_{j=1}^s \frac{z^{(j)}/2}{2^j} + \frac{1}{2^{s+1}}.
\end{aligned}
\]
Hence $F_Z(z-3^{-r})=F_Z(z)$. If such $s$ does not exist, then all the first $r$ digits of the ternary expansion of $z$ are $0$, so $z=0$, then we also have $F_Z(z-3^{-r})=F_Z(z)$ because they are both equal to $0$. Putting together we prove the first equation.

To prove the second equation, suppose the last $0$ in the ternary expansion of $z$ appears at the $(s+1)$-th digit. Then,
\[
\begin{aligned}
z &= (0\;.\;z^{(1)}\;\cdots\;z^{(s)}\;0\;2\;\cdots\;2)_3, \\
z + 3^{-r} &= (0\;.\;z^{(1)}\;\cdots\;z^{(s)}\;1\;0\;\cdots\;0)_3, \\
z + 2 \cdot 3^{-r} &= (0\;.\;z^{(1)}\;\cdots\;z^{(s)}\;1\;0\;\cdots\;1)_3.
\end{aligned}
\]
Again, applying~\Cref{lem:cantorrefl}(i), we obtain
\[
\begin{aligned}
F_Z(z) &= \sum_{j=1}^s \frac{z^{(j)}/2}{2^j} + \sum_{j=s+2}^r \frac{1}{2^j}, \\
F_Z(z + 3^{-r}) &= \sum_{j=1}^s \frac{z^{(j)}/2}{2^j} + \frac{1}{2^{s+1}}, \\
F_Z(z + 2 \cdot 3^{-r}) &= \sum_{j=1}^s \frac{z^{(j)}/2}{2^j} + \frac{1}{2^{s+1}}.
\end{aligned}
\]
Hence $F_Z(z) + 2^{-r} = F_Z(z + 3^{-r}) = F_Z(z + 2 \cdot 3^{-r})$. If such $s$ does not exist, then all the first $r$ digits of the ternary expansion of $z$ are $2$, so $z=1-3^{-r}$ and~\Cref{lem:cantorrefl}(i) gives $F_Z(z)=1-2^{-r}$, then we also have $F_Z(z) + 2^{-r} = F_Z(z + 3^{-r}) = F_Z(z + 2 \cdot 3^{-r})$ because they are all equal to $1$. Putting together, we prove the second equation.
\end{proof}

\begin{proof}[Proof of~\Cref{lem:cantorrefl}(iii)]The definitions of $r,z$ and the condition $x<y$ immediately imply the ternary expansions
\[
\begin{aligned}
x &= (0\;.\;z^{(1)}\;\cdots\;z^{(r)}\;0\;x^{(r+1)}\;\cdots)_3, \\
y &= (0\;.\;z^{(1)}\;\cdots\;z^{(r)}\;2\;y^{(r+1)}\;\cdots)_3, \\
z &= (0\;.\;z^{(1)}\;\cdots\;z^{(r)}\;0\;0\;\cdots)_3.
\end{aligned}
\]
So we have
\[
z \leq x \leq z + 3^{-r-1} \leq z + 2 \cdot 3^{-r-1} \leq y \leq z + 3^{-r},
\]
which implies
\[
z - 3^{-r} \leq 2 x - y \leq z \quad\&\quad z + 3^{-r} \leq 2 y - x \leq z + 2 \cdot 3^{-r}.
\]
Therefore,
\[
F_Z(z - 3^{-r}) \leq F_Z(2 x - y) \leq F_Z(z) \quad\&\quad F_Z(z + 3^{-r}) \leq F_Z(2 y - x) \leq F_Z(z + 2 \cdot 3^{-r}).
\]
Since $3^r z \in \mathbb{Z}$, using~\Cref{lem:cantorrefl}(ii), we get
\[
F_Z(z) \leq F_Z(2 x - y) \leq F_Z(z) \quad\&\quad F_Z(z) + 2^{-r} \leq F_Z(2 y - x) \leq F_Z(z) + 2^{-r},
\]
which proves the desired result.
\end{proof}

Armed with the above lemmas, we are ready to prove~\Cref{thm:nolimit}. We first analyze the convergence behavior of $(n - 1)a_n$ and $(n - 1)b_n$ defined in~\Cref{lem:probneigh} when $Z$ follows the Cantor distribution.

\vspace*{0.1cm}

\noindent{\bf Analysis of $a_n$}
    Since $F_Z$ is continuous, there are almost surely no ties. Then, $a_n$ can be simplified:
\[
\begin{aligned}
a_n &= \pr[\forall j \in \{3,\ldots,n\}: Z_j \not\in B(Z_1,|Z_2-Z_1|) \cup B(Z_2,|Z_1-Z_2|)] \\
&= \E[\pr[\forall j \in \{3,\ldots,n\}: Z_j \not\in B(Z_1,|Z_2-Z_1|) \cup B(Z_2,|Z_1-Z_2|) \mid Z_1,Z_2]] \\
&= \E[(1-\mu_Z(B(Z_1,|Z_2-Z_1|) \cup B(Z_2,|Z_1-Z_2|)))^{n-2}].
\end{aligned}
\]
Using~\Cref{lem:expexp}, we have that
\[
\begin{aligned}
a_n &= (1+O(\ln^4 n / n)) I_n + O(e^{- \ln^2 n}) \\
\quad\text{where}\quad
I_n &:= \E \exp(- n \mu_Z(B(Z_1,|Z_2-Z_1|) \cup B(Z_2,|Z_1-Z_2|))).
\end{aligned}
\]
Now we use~\Cref{lem:cantorrefl}(iii) to analyze the term $\mu_Z(B(Z_1,|Z_2-Z_1|) \cup B(Z_2,|Z_1-Z_2|))$. Define $R := \min\{j \geq 1: Z_1^{(j)} \neq Z_2^{(j)}\} - 1$; when $Z_1<Z_2$, we have
\[
\begin{aligned}
&\; \mu_Z(B(Z_1,|Z_2-Z_1|) \cup B(Z_2,|Z_1-Z_2|)) = \mu_Z((2Z_1-Z_2,2Z_2-Z_1)) \\
=&\; F_Z(2Z_2-Z_1)-F_Z(2Z_1-Z_2) = 2^{-R}.
\end{aligned}
\]
When $Z_1>Z_2$, by the i.i.d. property, we also have
\[
\mu_Z(B(Z_1,|Z_2-Z_1|) \cup B(Z_2,|Z_1-Z_2|)) = 2^{-R}.
\]
Since $Z_1 \neq Z_2$ almost surely, the above equality almost surely holds. Substituting it back, we obtain
\[
I_n = \E \exp(- n \cdot 2^{-R}).
\]
By definition, $(R+1)$ follows the Geometry distribution with success probability $1/2$, which means that $\pr(R=k) = 2^{-k-1}, k \geq 0$. So we get
\[
I_n = \sum_{k \geq 0} 2^{-k-1} e^{- 2^{-k} n}.
\]
Notice also that 
\[
\sum_{k < 0} 2^{-k-1} e^{- 2^{-k} n} = \frac12 \sum_{k \geq 1} 2^k e^{- 2^k n} \leq \frac12 \sum_{k \geq 1} 2^k e^{- 2^k - n + 1} = \left(\frac{e}{2} \sum_{k \geq 1} 2^k e^{- 2^k}\right) e^{-n} = O(e^{-n}),
\]
where the second step uses the inequality $a b \geq a + b - 1$ for $a,b \geq 1$, which can be derived from the fact $(a-1)(b-1) \geq 0$.

In light of both formulas, we can write
\[
I_n = \sum_{k \in \mathbb{Z}} 2^{-k-1} e^{- 2^{-k} n} + O(e^{-n}).
\]
Now, putting together, we get
\[
a_n = (1 + O(\ln^4 n / n)) \left(\sum_{k \in \mathbb{Z}} 2^{-k-1} e^{- 2^{-k} n} + O(e^{-n})\right) + O(e^{- \ln^2 n}).
\]
In other words, 
\begin{align}
n a_n &= (1+O(\ln^4 n / n)) \cdot \frac12 A(n) + O(n e^{- \ln^2 n}) \label{eq:nantemp} \\
\text{where}\quad A(x) &:= \sum_{k \in \mathbb{Z}} 2^{-k} x e^{- 2^{-k} x} \quad,\quad x \in (0,+\infty). \label{eq:axdef}
\end{align}
We observe some basic properties of $A(x)$:
\begin{itemize}
\item[(A1)] For any fixed $x \in (0,+\infty)$, the summation $\sum_{k = -n}^n 2^{-k} x e^{- 2^{-k} x}$ converges as $n \to \infty$. So $A: (0,+\infty) \to \R$ is well-defined.
\item[(A2)] By replacing $k$ with $(k+1)$, we have that $A(2x)=A(x)$ for any $x \in (0,+\infty)$.
\item[(A3)] For any $x \in [1,2]$, we have $2^{-k} x e^{- 2^{-k} x} \leq 2^{1-k} e^{- 2^{-k}}$ and $\sum_{k \in \mathbb{Z}} 2^{1-k} e^{- 2^{-k}} < \infty$. So it follows from dominated convergence theorem that $A(x)$ is continuous on $x \in [1,2]$; moreover, $A(x)$ is uniformly bounded on $x \in [1,2]$.
\end{itemize}
From (A3) we know that $A$ is uniformly bounded on $x \in [1,2]$. Combining this with (A2), we know that $A$ is uniformly bounded on the full domain $x \in (0,+\infty)$. Therefore, our result~\eqref{eq:nantemp} can be further simplified as
\begin{equation}\label{eq:nanfinal}
(n - 1) a_n = \frac12 A(n) + o(1).
\end{equation}

\vspace*{0.1cm}

\noindent{\bf Analysis of $b_n$.} The analysis of $b_n$ is similar to that of $a_n$. Since $F_Z$ is continuous, there are almost surely no ties. Using this fact, together with the i.i.d. property, we have
\[
\begin{aligned}
\frac{b_n}{n-2} &= \frac{1}{n-2} \sum_{j=3}^n \pr(M_Z(1)=M_Z(2)=j) = \pr(M_Z(1)=M_Z(2)=3) \\
&= \pr[\forall j \in \{4,\ldots,n\}: Z_j \not\in B(Z_1,|Z_3-Z_1|) \cup B(Z_2,|Z_3-Z_2|); |Z_1-Z_3|,|Z_2-Z_3|<|Z_1-Z_2|] \\
&= \E[(1-\mu_Z(B(Z_1,|Z_3-Z_1|) \cup B(Z_2,|Z_3-Z_2|)))^{n-3} \one\{ |Z_1-Z_3|,|Z_2-Z_3|<|Z_1-Z_2|\}] \\
&= \E[(1-\mu_Z(B(Z_1,|Z_3-Z_1|) \cup B(Z_2,|Z_3-Z_2|)))^{n-3} \one\{ Z_1<Z_3<Z_2 \text{ or } Z_2<Z_3<Z_1\}] \\
&= 2 \E[(1-\mu_Z(B(Z_1,|Z_3-Z_1|) \cup B(Z_2,|Z_3-Z_2|)))^{n-3} \one\{ Z_1<Z_3<Z_2\}] \\
&= 2 \E[(1-\mu_Z((2Z_1-Z_3,2Z_2-Z_3)))^{n-3} \one\{ Z_1<Z_3<Z_2\}].
\end{aligned}
\]
Using~\Cref{lem:expexp}, we have that
\[
\begin{aligned}
\frac{b_n}{2(n-2)} &= (1 + O(\ln^4 n / n)) J_n + O(e^{- \ln^ 2 n}) \\
\text{where}\quad J_n &:= \E[e^{- n \mu_Z((2Z_1-Z_3,2Z_2-Z_3))} \one\{ Z_1<Z_3<Z_2\}].
\end{aligned}
\]
Now we use~\Cref{lem:cantorrefl}(iii) to analyze the term $\mu_Z((2Z_1-Z_3,2Z_2-Z_3))$. Define corresponding terms
\[
\begin{aligned}
R_1 := \min\{j \geq 1: Z_1^{(j)} \neq Z_3^{(j)}\} - 1 \quad&\&\quad \bar{Z}_1 := (0.Z_3^{(1)} \cdots Z_3^{(R_1)} 0 0 \cdots)_3, \\
R_2 := \min\{j \geq 1: Z_2^{(j)} \neq Z_3^{(j)}\} - 1 \quad&\&\quad \bar{Z}_2 := (0.Z_3^{(1)} \cdots Z_3^{(R_2)} 0 0 \cdots)_3.
\end{aligned}
\]
Different from the analysis of $a_n$ previously, here we first study the symmetry. Rewrite $R_1=r(Z_1,Z_3)$ and $R_2=r(Z_2,Z_3)$, where $r(x,y) := \min\{j \geq 1: x^{(j)} \neq y^{(j)}\} - 1$. Notice that almost surely, there does not exist an $n > 0$ satisfying that $\forall j > n, Z^{(j)} = 0$, which gives that the ternary expansion of $(1-Z)$ can be obtained from the ternary expansion of $Z$ by interchanging the digits $0$ and $2$. So, by the definition of $r(\cdot,\cdot)$, we have $r(Z_1,Z_3)=r(1-Z_1,1-Z_3),r(Z_2,Z_3)=r(1-Z_2,1-Z_3)$ almost surely. Now, according to the fact that $Z$ and $(1-Z)$ are identically distributed, we have
\[
\begin{aligned}
&\; \E[e^{- n \mu_Z((2Z_1-Z_3,2Z_2-Z_3))} \one\{ Z_1<Z_3<Z_2, r(Z_1,Z_3) < r(Z_2,Z_3)\}] \\
=&\; \E[e^{- n \mu_Z((1-2Z_1+Z_3,1-2Z_2+Z_3))} \one\{ 1-Z_1<1-Z_3<1-Z_2, r(1-Z_1,1-Z_3) < r(1-Z_2,1-Z_3)\}] \\
=&\; \E[e^{- n \mu_Z((2Z_2-Z_3,2Z_1-Z_3))} \one\{ Z_2<Z_3<Z_1, r(Z_1,Z_3) < r(Z_2,Z_3)\}] \\
=&\; \E[e^{- n \mu_Z((2Z_1-Z_3,2Z_2-Z_3))} \one\{ Z_1<Z_3<Z_2, r(Z_2,Z_3) < r(Z_1,Z_3)\}].
\end{aligned}
\]
Moreover, notice that under $Z_1<Z_3<Z_2$, the definitions of $R_1,R_2$ imply that $Z_3^{(R_1+1)}=2, Z_3^{(R_2+1)}=0$, so $R_1 \neq R_2$. Therefore, we have that
\[
\begin{aligned}
J_n &= \E[e^{- n \mu_Z((2Z_1-Z_3,2Z_2-Z_3))} \one\{ Z_1<Z_3<Z_2\} \one\{ R_1<R_2 \text{ or } R_2<R_1\}] \\
&= 2 \E[e^{- n \mu_Z((2Z_1-Z_3,2Z_2-Z_3))} \one\{ Z_1<Z_3<Z_2, R_1<R_2\}].
\end{aligned}
\]
Then, we use~\Cref{lem:cantorrefl}(iii) under the event $\{Z_1<Z_3<Z_2, R_1<R_2\}$:
\[
\begin{aligned}
\mu_Z((2Z_1-Z_3,2Z_2-Z_3)) &= F_Z(2Z_2-Z_3) - F_Z(2Z_1-Z_3)
= F_Z(\bar{Z}_2) + 2^{-R_2} - F_Z(\bar{Z}_1) \\
&= \sum_{j=1}^{R_2} \frac{Z_3^{(j)}/2}{2^j} + 2^{-R_2} - \sum_{j=1}^{R_1} \frac{Z_3^{(j)}/2}{2^j}
= 2^{-R_2}+ \sum_{j=R_1+1}^{R_2} 2^{-j-1} Z_3^{(j)}.
\end{aligned}
\]
Substituting it back, we obtain
\[
\frac12 J_n = \E\left[\exp\left(- n \left(2^{-R_2} + \sum_{j=R_1+1}^{R_2} 2^{-j-1} Z_3^{(j)}\right)\right) \one\{ Z_1<Z_3<Z_2, R_1<R_2\}\right].
\]

Now, we calculate the value of $J_n/2$. First, we may reexpress $J_n / 2$ as:
\[
\frac12 J_n = \sum_{0 \leq i < k} \E\left[\exp\left(- n \left(2^{-k} + \sum_{j=i+1}^k 2^{-j-1} Z_3^{(j)}\right)\right) \one\{ Z_1<Z_3<Z_2, R_1=i, R_2=k\}\right].
\]
For any $0 \leq i < k$, the event $\{Z_1<Z_3<Z_2, R_1=i, R_2=k\}$ is equivalent to the event
\[
\{
(Z_1^{(j)})_{1 \leq j \leq i} = (Z_3^{(j)})_{1 \leq j \leq i}, Z_1^{(i+1)}=0, Z_3^{(i+1)}=2,
(Z_3^{(j)})_{1 \leq j \leq k} = (Z_2^{(j)})_{1 \leq j \leq k}, Z_3^{(k+1)}=0, Z_2^{(k+1)}=2
\},
\]
so conditioning on $Z_3$, it happens with probability $2^{-i-k-2} \one\{Z_3^{(i+1)}=2, Z_3^{(k+1)}=0\}$. So, we get
\[
\begin{aligned}
\frac12 J_n &= \sum_{0 \leq i < k} \E\left[\exp\left(- n \left(2^{-k} + \sum_{j=i+1}^k 2^{-j-1} Z_3^{(j)}\right)\right) \pr[Z_1<Z_3<Z_2, R_1=i, R_2=k \mid Z_3]\right] \\
&= \sum_{0 \leq i < k} 2^{-i-k-2} \E\left[\exp\left(- n \left(2^{-k} + \sum_{j=i+1}^k 2^{-j-1} Z_3^{(j)}\right)\right) \one\{ Z_3^{(i+1)}=2, Z_3^{(k+1)}=0\}\right] \\
&= \sum_{0 \leq i < k} 2^{-i-k-4} \E\left[\exp\left(- n \left(2^{-k} + 2^{-i-1} + \sum_{j=i+2}^k 2^{-j-1} Z_3^{(j)}\right)\right)\right].
\end{aligned}
\]
We calculate the expectation:
\[
\begin{aligned}
&\; \E\left[\exp\left(- n \left(2^{-k} + 2^{-i-1} + \sum_{j=i+2}^k 2^{-j-1} Z_3^{(j)}\right)\right)\right]
= e^{- 2^{-k} n - 2^{-i-1} n} \E\left[\prod_{j=i+2}^k e^{- 2^{-j-1} Z_3^{(j)} n}\right] \\
=&\; e^{- 2^{-k} n - 2^{-i-1} n} \prod_{j=i+2}^k \frac12 (1 + e^{- 2^{-j} n})
= 2^{i+1-k} e^{- 2^{-k} n - 2^{-i-1} n} \prod_{j=i+2}^k (1 + e^{- 2^{-j} n}).
\end{aligned}
\]
Substituting it back, we get
\[
\frac12 J_n = \sum_{0 \leq i < k}2^{-2k-3} e^{- 2^{-k} n - 2^{-i-1} n} \prod_{j=i+2}^k (1 + e^{- 2^{-j} n})
= \sum_{k=1}^{\infty} 2^{-2k-3} e^{- 2^{-k} n} \sum_{i=0}^{k-1} e^{- 2^{-i-1} n} \prod_{j=i+2}^k (1 + e^{- 2^{-j} n}).
\]
We calculate the inner summation, denoted as $S_k$. Notice that we can do telescoping:
\[
S_k = \sum_{i=0}^{k-1} \left(\prod_{j=i+1}^k (1 + e^{- 2^{-j} n}) - \prod_{j=i+2}^k (1 + e^{- 2^{-j} n})\right) = \prod_{j=1}^k (1 + e^{- 2^{-j} n}) - 1.
\]
Then, since
\[
\prod_{j=0}^{k-1} (1 + x^{2^j}) = \frac{(1-x) (1+x) (1+x^2) \cdots (1+x^{2^{k-1}})}{1 - x} = \frac{1 - x^{2^k}}{1 - x}
\]
from repeated application of $a^2-b^2=(a-b)(a+b)$, we have by plugging $x = e^{- 2^{-k} n}$ that
\[
S_k = \prod_{j=0}^{k-1} (1 + e^{- 2^{-k} n \cdot 2^j}) - 1 = \frac{1 - e^{-n}}{1 - e^{- 2^{-k} n}} - 1 = \frac{e^{- 2^{-k} n} - e^{-n}}{1 - e^{- 2^{-k} n}}.
\]
Substituting back and using a similar extension of summation range as that in the analysis of $a_n$:
\[
\frac12 J_n
= \sum_{k \geq 1} 2^{-2k-3} \frac{e^{- 2^{-k} n} - e^{-n}}{e^{2^{-k} n} - 1}
= \sum_{k \in \mathbb{Z}} \frac{2^{-2k-3}}{e^{2^{-k} n} (e^{2^{-k} n} - 1)} - \sum_{k \leq 0} \frac{2^{-2k-3}}{e^{2^{-k} n} (e^{2^{-k} n} - 1)} - \sum_{k \geq 1} \frac{2^{-2k-3} e^{-n}}{e^{2^{-k} n} - 1}.
\]
The second summation and the third summation can be bounded:
\[
\begin{aligned}
\sum_{k \leq 0} \frac{2^{-2k-3}}{e^{2^{-k} n} (e^{2^{-k} n} - 1)} &= \sum_{k \geq 0} \frac{2^{2k-3}}{e^{2^k n} (e^{2^k n} - 1)} \leq \sum_{k \geq 0} \frac{2^{2k-3} e^{-2^k-n+1}}{e^n-1} = \left(\frac{e}{8} \sum_{k \geq 0} 2^{2k} e^{-2^k}\right) \frac{e^{-n}}{e^n-1} = O(e^{-2n}), \\
\sum_{k \geq 1} \frac{2^{-2k-3} e^{-n}}{e^{2^{-k} n} - 1} &\leq \sum_{k \geq 1} \frac{2^{-2k-3} e^{-n}}{2^{-k} n} = \frac{e^{-n}}{8n} = O(e^{-n}),
\end{aligned}
\]
where the second step in the first line again uses the inequality $a b \geq a + b - 1$ for $a,b \geq 1$. So, we obtain
\[
\frac12 J_n = \sum_{k \in \mathbb{Z}} \frac{2^{-2k-3}}{e^{2^{-k} n} (e^{2^{-k} n} - 1)} + O(e^{-n}).
\]

Now, putting together, we get
\[
\frac{b_n}{4(n-2)}
= (1 + O(\ln^4 n / n)) \left(\sum_{k \in \mathbb{Z}} \frac{2^{-2k-3}}{e^{2^{-k} n} (e^{2^{-k} n} - 1)} + O(e^{-n})\right) + O(e^{- \ln^2 n}).
\]
Using the fact that $(1-2/n)(1 + O(\ln^4 n / n)) = 1 + O(\ln^4 n / n)$, we can simplify the result as
\begin{align}
n b_n &= (1 + O(\ln^4 n / n)) \cdot \frac12 B(n) + O(n^2 e^{- \ln^2 n}) \label{eq:nbntemp} \\
\text{where}\quad B(x) &:= \sum_{k \in \mathbb{Z}} \frac{2^{-2k} x^2}{e^{2^{-k} x} (e^{2^{-k} x} - 1)} \quad,\quad x \in (0,+\infty). \label{eq:bxdef}
\end{align}
We observe some basic properties of $B(x)$:
\begin{itemize}
\item[(B1)] For any fixed $x \in (0,+\infty)$, the summation $\sum_{k = -m}^n \frac{2^{-2k} x^2}{e^{2^{-k} x} (e^{2^{-k} x} - 1)}$ converges as $m, n \to \infty$. So $B: (0,+\infty) \to \R$ is well-defined.
\item[(B2)] By replacing $k$ with $(k+1)$, we have that $B(2x)=B(x)$ for any $x \in (0,+\infty)$.
\item[(B3)] For any $x \in [1,2]$, we have $\frac{2^{-2k} x^2}{e^{2^{-k} x} (e^{2^{-k} x} - 1)} \leq \frac{2^{2-2k}}{e^{2^{-k}} (e^{2^{-k}} - 1)}$ and $\sum_{k \in \mathbb{Z}} \frac{2^{2-2k}}{e^{2^{-k}} (e^{2^{-k}} - 1)} < \infty$. So it follows from dominated convergence theorem that $B(x)$ is continuous on $x \in [1,2]$; moreover, $B(x)$ is uniformly bounded on $x \in [1,2]$.
\end{itemize}
From (B3) we know that $B$ is uniformly bounded on $x \in [1,2]$. Combining this with (B2), we know that $B$ is uniformly bounded on the full domain $(0,+\infty)$. Therefore, our result~\eqref{eq:nbntemp} can be further simplified as
\begin{equation}\label{eq:nbnfinal}
(n - 1) b_n = \frac12 B(n) + o(1).
\end{equation}

\paragraph{Proof of the main result} We prove by contradiction. Suppose there exist coefficients $(C_a,C_b) \neq (0,0)$ and a constant $C$ such that
\[
C_a (n - 1) a_n + C_b  (n - 1) b_n \to C, \quad\mathrm{as}\; n \to \infty.
\]
Then,~\eqref{eq:nanfinal} and~\eqref{eq:nbnfinal} give
\[
C_a A(n) + C_b B(n) \to 2C.
\]
Using the log-periodicity $A(2 x) = A(x), B(2 x) = B(x)$ established in (A2) and (B2) before, we have that
\[
C_a A(n / 2^{\lfloor \log_2 n \rfloor}) + C_b B(n / 2^{\lfloor \log_2 n \rfloor}) \to 2C.
\]
According to part (ii) of~\Cref{lem:limitpoint}, for any $x \in [1,2)$, there exists a subsequence $n_k / 2^{\lfloor \log_2 n_k \rfloor} \to x$. Substituting $n=n_k$ in the above formula, taking the limit $k \to \infty$, and using the continuity of $A,B$ on $[1,2]$ established in (A3) and (B3) before, we get
\[
C_a A(x) + C_b B(x) = 2C, \quad \forall x \in [1,2).
\]

Now we consider the ``Fourier transform'' of the above equation w.r.t. $\log x$. Formally speaking, for any integer $m \geq 1$, we multiply $x^{2 \pi i m / \ln 2} / x$ in the above equation and take integral over $x \in [1,2)$:
\begin{equation}\label{eq:abint}
C_a \int_1^2 A(x) \frac{x^{2 \pi i m / \ln 2}}{x} d x + C_b \int_1^2 B(x) \frac{x^{2 \pi i m / \ln 2}}{x} d x
= 2C \int_1^2 \frac{x^{2 \pi i m / \ln 2}}{x} d x.
\end{equation}
We calculate the three integrals respectively. The last one is easy:
\[
\int_1^2 \frac{x^{2 \pi i m / \ln 2}}{x} d x
= \left.\frac{x^{2 \pi i m / \ln 2}}{2 \pi i m / \ln 2}\right|_1^2
= 0.
\]
Denote the first two integrals as $\hat{A}_m,\hat{B}_m$, respectively. For $\hat{A}_m$, using the definition of $A(x)$ given in~\eqref{eq:axdef}, we have that
\[
\begin{aligned}
\hat{A}_m
&= \int_1^2 \sum_{k \in \mathbb{Z}} 2^{-k} e^{- 2^{-k} x} x^{\frac{2 \pi i m}{\ln 2}} d x
= \sum_{k \in \mathbb{Z}} \int_1^2 2^{-k} e^{- 2^{-k} x} x^{\frac{2 \pi i m}{\ln 2}} d x \\
&= \sum_{k \in \mathbb{Z}} \int_{2^{-k}}^{2^{1-k}} t^{\frac{2 \pi i m}{\ln 2}} e^{-t} d t
= \int_0^{+\infty} t^{\frac{2 \pi i m}{\ln 2}} e^{-t} d t
= \Gamma\left(1+\frac{2 \pi i m}{\ln 2}\right).
\end{aligned}
\]

Here, the $\Gamma$ represents Euler's Gamma function; the second step uses the Fubini theorem; and the third step uses a change of variable $t := 2^{-k} x$ together with the log-periodicity $(2x)^{2 \pi i m / \ln 2} = x^{2 \pi i m / \ln 2}$. The condition for the Fubini theorem applied in the second step holds because for any $m \ge 1$,
\[
\int_1^2 \sum_{k \in \mathbb{Z}} \left|2^{-k} e^{- 2^{-k} x} x^{\frac{2 \pi i m}{\ln 2}}\right| d x = \int_1^2 \sum_{k \in \mathbb{Z}} \left|2^{-k} e^{- 2^{-k} x} e^{\frac{2 \pi i m \ln x}{\ln 2}}\right| d x
= \int_1^2 \sum_{k \in \mathbb{Z}} 2^{-k} e^{- 2^{-k} x} d x
< \infty.
\]
For $\hat{B}_m$, using the definition of $B(x)$ given in~\eqref{eq:bxdef}, we calculate it with an analogous method:
\[
\begin{aligned}
\hat{B}_m
&= \int_1^2 \sum_k \frac{2^{-2k} x}{e^{2^{-k}x}(e^{2^{-k}x}-1)} x^{\frac{2 \pi i m}{\ln 2}} d x
= \sum_k \int_1^2 \frac{2^{-2k} x}{e^{2^{-k}x}(e^{2^{-k}x}-1)} x^{\frac{2 \pi i m}{\ln 2}} d x \\
&= \sum_k \int_{2^{-k}}^{2^{1-k}} \frac{t}{e^t(e^t-1)} t^{\frac{2 \pi i m}{\ln 2}} d t
= \int_0^{+\infty} \frac{t}{e^t(e^t-1)} t^{\frac{2 \pi i m}{\ln 2}} d t.
\end{aligned}
\]
Here similarly, the second step uses the Fubini theorem, the condition of which follows from an analogous argument; and the third step uses the change of variable $t := 2^{-k} x$ together with the log-periodicity $(2x)^{2 \pi i m / \ln 2} = x^{2 \pi i m / \ln 2}$. Now we further calculate the integral on the far right of the above equation:
\[
\begin{aligned}
\hat{B}_m
&= \int_0^{+\infty} \frac{t^{1 + 2 \pi i m / \ln 2} d t}{e^{2t} (1-e^{-t})}
= \int_0^{+\infty} e^{-2t} \left(\sum_{n \geq 0} e^{- n t}\right) t^{1 + \frac{2 \pi i m}{\ln 2}} d t
= \int_0^{+\infty} \sum_{n \geq 2} t^{1 + \frac{2 \pi i m}{\ln 2}} e^{- n t} d t \\
&= \sum_{n \geq 2} \int_0^{+\infty} t^{1 + \frac{2 \pi i m}{\ln 2}} e^{- n t} d t
= \sum_{n \geq 2} \int_0^{+\infty} n^{- 2 - \frac{2 \pi i m}{\ln 2}} u^{1 + \frac{2 \pi i m}{\ln 2}} e^{-u} d u
= \sum_{n \geq 2} n^{- 2 - \frac{2 \pi i m}{\ln 2}} \Gamma\left(2 + \frac{2 \pi i m}{\ln 2}\right) \\
&= \left(\zeta\left(2 + \frac{2 \pi i m}{\ln 2}\right) - 1\right) \Gamma\left(2 + \frac{2 \pi i m}{\ln 2}\right).
\end{aligned}
\]
Here, the $\zeta$ represents Riemann's zeta function; the fourth step uses the Fubini theorem, the condition of which follows from similar arguments; and the fifth step uses a change of variable $ u := n t$.

Substituting the three integrals into~\eqref{eq:abint}, we get
\[
C_a \Gamma\left(1 + \frac{2 \pi i m}{\ln 2}\right) + C_b \left(\zeta\left(2 + \frac{2 \pi i m}{\ln 2}\right) - 1\right) \Gamma\left(2 + \frac{2 \pi i m}{\ln 2}\right) = 0,\quad \forall m \geq 1.
\]
It is well known that $\Gamma(z+1) = z \Gamma(z)$ and $\Gamma$ has no zeros (see~\citet[Definition~1.1.1 and Theorem~1.1.2]{andrews1999gamma}). So we can eliminate a $\Gamma(1 + 2 \pi i m / \ln 2)$ in the above equation to get
\[
C_a + C_b \left(\zeta\left(2 + \frac{2 \pi i m}{\ln 2}\right) - 1\right) \left(1 + \frac{2\pi i m}{\ln 2}\right) = 0.
\]
From the equation we know $C_b \neq 0$, otherwise $C_a$ will also be $0$, contradicting with our requirement $(C_a,C_b) \neq (0,0)$. So divide $C_b$ and solve $\zeta(2 + 2 \pi i m / \ln 2)$:
\[
\zeta\left(2 + \frac{2 \pi i m}{\ln 2}\right) = 1 - \frac{C_a / C_b}{1 + \frac{2 \pi i m}{\ln 2}}.
\]
Using the definition $\zeta(s) := \sum_{n \geq 1} n^{-s}$, we have that
\[
\frac14 + \frac19 e^{- 2 \pi i m \log_2 3} = - \frac{C_a / C_b}{1 + \frac{2 \pi i m}{\ln 2}} - \sum_{n \geq 4} \frac{1}{n^{2 + 2 \pi i m / \ln 2}}.
\]
Using the periodicity $e^{2 \pi i m (x + 1)} = e^{2 \pi i m x}$ and the triangle inequality, we have that
\[
\left|\frac14 + \frac19 e^{- 2 \pi i ((m \log_2 3) - \lfloor m \log_2 3\rfloor)}\right| \leq \left|\frac{C_a / C_b}{1 + \frac{2 \pi i m}{\ln 2}}\right| + \sum_{n \geq 4} \frac{1}{n^2}
\quad,\quad \forall m \geq 0.
\]
Apparently $\log_2 3$ is an irrational number, so using part (i) of~\Cref{lem:limitpoint}, we obtain a strictly increasing sequence $\{m_{\ell}\}_{\ell=1}^{\infty} \in \mathbb{Z}^+$ such that $(m_{\ell} \log_2 3) - \lfloor m_{\ell} \log_2 3 \rfloor \to 0 ,\; \ell \to \infty$. Substituting $m = m_{\ell}$ in the above inequality and taking the limit $\ell \to \infty$, we get
\[
\frac14 + \frac19 \leq \sum_{n \geq 4} \frac{1}{n^2},
\]
which is a contradiction since $\sum_{n \geq 4} \frac{1}{n^2} \leq \sum_{n \geq 4} \frac{1}{(n-1)n} = \frac13 < \frac14 + \frac19$.

\section{Proof of results in~\Cref{sec:algorithm}}\label{sec:pfalg}

\subsection{Proof of~\Cref{thm:complexity}}

\begin{proof}[Proof of time complexity]
    We first discuss time complexity. For Algorithm~\ref{alg:2drank}, Step 1 has time complexity $O(n_a \log n_a)$. For Step 2, the computation of each $k_i$ has time complexity at most $O(\log n_a)$, so in total it is $O(n_b \log n_a)$. For Step 4, in the $j$-th round, the construction of each $\bs{e}_{j'}^j$ has time complexity $O(2^{j - 1} \log 2^{j - 1})$, so this step has time complexity $O(J_j' \cdot 2^{j - 1} \log 2^{j - 1}) = O(n_a \log n_a)$. Similarly, the time complexity for Steps 5 - 7 is $O(n_b \log 2^{j - 1}) = O(n_b \log n_a)$. Since Steps 4 - 7 are repeated $J$ times, its total time complexity is $O(n (\log n)^2)$. Putting together, we have that the time complexity of Algorithm~\ref{alg:2drank} is $O(n (\log n)^2)$.

For Algorithm~\ref{alg:drank}, we prove it by induction. When $d = 3$, following analogous analysis we can prove that Steps 1-2 and Steps 4-8 have time complexity at most $O(n \log n)$. Then the complexity depends only on Step 9. For each $(\mathcal{A}_{j'}, \cB_{j'})$, the time complexity to run Algorithm~\ref{alg:2drank} is $O((|\mathcal{A}_{j'}| + |\cB_{j'}|) (\log n)^2)$. So the total complexity for Step 9 is $O((n_a + n_b) (\log n)^2)$. This means that Algorithm~\ref{alg:drank} has complexity $O(n (\log n)^3)$.

For the higher-dimensional case, suppose in the $(d-1)$-dimensional regime the complexity is $O(n (\log n)^{d - 1})$, then applying an analogous argument, we have that for the $d$-dimensional regime, Step 9 has complexity $O((n_a + n_b) (\log n)^{d - 1})$, so that the entire algorithm has complexity $O(n (\log n)^d)$, thereby proving the desired result.
\end{proof}

We second discuss the correctness. We start by introducing the following lemma:

\begin{lemma}\label{lem:jl}
    Consider a $k \ge 2$. For any $1 \le k' < k$, there exists a unique even number $\ell \ge 2$ and a unique integer $j \ge 1$ such that
    \begin{equation}\label{eq:jl}
        2^{j - 1}(\ell - 2) < k' \le 2^{j - 1} \cdot (\ell - 1) < k \le 2^{j - 1} \cdot \ell.
    \end{equation}
\end{lemma}

\begin{proof}
Intuitively,~\eqref{eq:jl} means that if we divide integers $\{1, 2, \ldots\}$ into blocks of size $2^{j - 1}$ (i.e., the first block contains quantities $\{1, \ldots, 2^{j - 1}\}$, the second block contains $\{2^{j - 1} + 1, \ldots, 2^j\}$, ...), then $k$ and $k'$ would belong to the $\ell$-th and $(\ell -1)$-th blocks, respectively, and $\ell$ must be an even number.

We first prove such $j$ and $\ell$ exist. Consider the binary representations of $k - 1$ and $k' - 1$:
\[
k - 1 = \sum_{i = 1}^m 2^{i - 1} a_i, \quad k' - 1 = \sum_{i = 1}^m 2^{i - 1} a_i', \quad\textrm{where}\; a_i, a_i' \in \{0, 1\}.
\]
Let $i^\star$ be the largest $i$ such that $a_i \neq a_i'$. Obviously, we have $a_{i^\star} = 1$ and $a_{i^\star}' = 0$ since $k > k'$. Then by setting $j = i^\star$, $\ell = \sum_{i = i^\star}^m 2^{i - i^\star} a_i + 1$ and $\ell' = \sum_{i = i^\star}^m 2^{i - i^\star} a_i' + 1$, it is easy to discover that by dividing $\{1,2, \ldots\}$ into blocks of size $2^{j - 1}$, $k$ and $k'$ would belong to the $\ell$-th and the $\ell'$-th block, respectively. Since $a_{i^\star} = 1$, $\ell$ is an even number. Since for all $i > i^\star, a_i = a_i'$, we have $\ell' = \ell - 1$. This finishes our construction.

We second show uniqueness. Notice that~\eqref{eq:jl} is equivalent to
\[
2^{j-1} (\ell-2) \leq k'-1 < 2^{j-1} (\ell-1) \leq k-1 < 2^{j-1} \ell,
\]
which is further equivalent to
\begin{equation}\label{eq:jltrans}
\ell-2 = \left\lfloor\frac{k'-1}{2^{j-1}}\right\rfloor = \sum_{i = j}^m 2^{i - j} a_i'
\quad\&\quad
\ell-1 = \left\lfloor\frac{k-1}{2^{j-1}}\right\rfloor = \sum_{i = j}^m 2^{i - j} a_i.
\end{equation}

Then in order to prove uniqueness, we only need to prove that for any $j \neq i^\star$, we cannot find an even $\ell$ satisfying~\eqref{eq:jltrans}. We prove it by contradiction.

Suppose there exists some $j>i^\star$ such that we can find an $\ell$ satisfying~\eqref{eq:jltrans}. The definition of $i^\star$ tells that $\forall i>i^\star: a_i'=a_i$. So, here we have
\[
\ell-2 = \sum_{i = j}^m 2^{i - j} a_i' = \sum_{i = j}^m 2^{i - j} a_i = \ell-1,
\]
which is a contradiction.

Suppose there exists some $j<i^\star$ such that we can find an even $\ell$ satisfying~\eqref{eq:jltrans}. Recall that $\forall i>i^\star, a_i'=a_i$ and the fact $a_{i^\star}'=0, a_{i^\star}=1$ deduced in the proof of the existence previously. Then in order to make $\sum_{i = j}^m 2^{i - j} a_i'$ and $\sum_{i = j}^m 2^{i - j} a_i$ to differ only by $1$, we must require $a_{i^\star-1}'=\cdots=a_j'=1$ and $a_{i^\star-1}=\cdots=a_j=0$. Then, $\ell-1$ must be an even number, which contradicts with our requirement that $\ell$ must be even.

Putting together, we prove uniqueness.
\end{proof}

\begin{proof}[Proof of correctness]
We first prove correctness of Algorithm~\ref{alg:2drank}. For each $i$, if $k_i \le 1$, $c_i$ is apparently correct. Otherwise, let $\mathcal{J}_i$ be the set of $j$ where $k_i$ satisfies the criterion in Step 6. We denote the corresponding even integer by $\ell_{j,i}$. Then we may express $c_i$ constructed by the algorithm as
    \[
    c_i = \sum_{j \in \mathcal{J}_i} \sum_{x \in \bs{e}_{\ell_{j,i} - 1}^j}\one\{x \le \bsb_{i,2}\} + \one\{\bsa_{k_i, 2} \le \bsb_{i,2}\}.
    \]
    Then in order to prove correctness, i.e., it is equal to $\sum_{k=1}^{k_i} \one\{\bsa_{k, 2} \le \bsb_{i, 2}\}$, it suffices to show: i) for any $k < k_i$, there exists a \emph{unique} $j \in \mathcal{J}_i$ such that $\bsa_{k, 2} \in \bs{e}_{\ell_{j,i} - 1}^j$; ii) for any $k \ge k_i$, such $j$ does not exist. The first point is a direct consequence of~\Cref{lem:jl}. The second point is based on the fact that, for all $\bs{e}_{\ell_{j,i} - 1}^j$, their corresponding $k$ must be strictly smaller than $k_i$.

    We second prove correctness of Algorithm~\ref{alg:drank}. If $k_i \le 1$, the result is straightforward. Otherwise, we first prove its correctness for $d = 3$. Using $\mathcal{J}_i, \ell_{j, i}$ defined before, we may express $c_i$ as
    \[
    c_i = \sum_{j \in \mathcal{J}_i} \sum_{\bs{x} \in \mathcal{A}_{\ell_{j,i} - 1}^j}\one\{\bs{x} \le \bsb_{i,\{2, \ldots, d\}}\} + \one\{\bsa_{k_i, \{2, \ldots, d\}} \le \bsb_{i,\{2, \ldots, d\}}\}.
    \]
    Then following exactly the same argument as for Algorithm~\ref{alg:2drank} and the fact that Algorithm~\ref{alg:2drank} is correct, we prove the correctness for $d = 3$. Using exactly the same argument, we may also prove that when the algorithm is correct for $d - 1$ ($d \ge 4$), the algorithm is also correct for $d$. Putting together, the desired result follows by induction.
\end{proof}

\subsection{Proof of~\Cref{prop:mixture}}

Consider two random variables $\by', \by''$ such that $(\by', \bz) \sim \pr_0$ and $(\by'', \bz) \sim \pr_1$. Let $\Lambda \sim \mathrm{Bernoulli}(\eta)$ and is independent from other randomness. Finally, let $\by = (1 - \Lambda)\by' + \Lambda \by''$. Then it is straightforward that $(\by, \bz) \sim (1 - \eta) \pr_0 + \eta \pr_1$. Moreover, the marginal distributions of $\by, \by', \by''$ are all the same.

We now have
\[
\taci = \frac{\int \var(\pr(\by \ge \bsy \mid \bz)) d \tilde{\mu}_{\by}(\bsy)}{\int \var(\one\{\by \ge \bsy\}) d \tilde{\mu}_{\by}(\bsy)} = \frac{\int \var(\E(\Lambda \one\{\by'' \ge \bsy\} + (1 - \Lambda) \one\{\by' \ge \bsy\} \mid \bz)) d \tilde{\mu}_{\by}(\bsy)}{\int \var(\one\{\by \ge \bsy\}) d \tilde{\mu}_{\by}(\bsy)}.
\]
Using the independence between $\by'$ and $\bz$, we further have
\[
\taci = \frac{\int \var(\eta \E(\one\{\by'' \ge \bsy\} \mid \bz)) d \tilde{\mu}_{\by}(\bsy)}{\int \var(\one\{\by \ge \bsy\}) d \tilde{\mu}_{\by}(\bsy)} = \frac{\eta^2 \int \var(\pr(\by'' \ge \bsy \mid \bz)) d \tilde{\mu}_{\by''}(\bsy)}{\int \var(\one\{\by'' \ge \bsy\}) d \tilde{\mu}_{\by''}(\bsy)}.
\]
where for the last equality we apply that the marginal distribution of $\by$ and $\by''$ are the same. Since $\by''$ is a function of $\bz$, we further have $\taci = \eta^2$, which proves the desired result.

\subsection{Proof of~\Cref{prop:additive}}

Consider $(\bz_1, \varepsilon_1), \ldots, (\bz_3, \varepsilon_3)$ as three i.i.d. copies of $(\bz, \varepsilon)$ and write $Y_i := \eta h(\bz_i) + \varepsilon_i$ for $i \in \{1,2,3\}$. Write $U := h(\bz_1) - h(\bz)$ and $V := \varepsilon_2 \wedge \varepsilon_3 - \varepsilon_1$. 

First, for any measurable set $S \subseteq (0, \infty)$, it is obvious that the event $\{\varepsilon_2 \wedge \varepsilon_3 - \varepsilon_1 \in S \;\&\; \varepsilon_2 \le \varepsilon_3\}$ must be a subset of the event $\{\varepsilon_2 - \varepsilon_1 \wedge \varepsilon_3 \in S \;\&\; \varepsilon_1 \le \varepsilon_3\}$. This means
\[
\pr(\varepsilon_2 \wedge \varepsilon_3 - \varepsilon_1 \in S \;\&\; \varepsilon_2 \le \varepsilon_3) \le \pr(\varepsilon_2 - \varepsilon_1 \wedge \varepsilon_3 \in S \;\&\; \varepsilon_1 \le \varepsilon_3).
\]
Using the i.i.d. property of the $\varepsilon_i$'s and their continuity, we have that the left hand side and right hand side are equal to $\frac{1}{2} \pr(\varepsilon_2 \wedge \varepsilon_3 - \varepsilon_1 \in S)$ and $\frac{1}{2}\pr(\varepsilon_2 - \varepsilon_1 \wedge \varepsilon_3 \in S)$, respectively, which further implies that
\begin{equation}\label{eq:vsym}
    \pr(V \in S) \le \pr(-V \in S).
\end{equation}

Second, define the function $p(\eta) := \pr(V \ge \eta U)$, then for any $0 \le \eta_1 < \eta_2$, 
\[
    p(\eta_2) - p(\eta_1) = \pr(- \eta_1 U < - V < -\eta_2 U \;\&\; U < 0) - \pr( \eta_1 U < V < \eta_2 U \;\&\; U > 0).
\]
Using further $V \independent U$ and that $U$ is symmetrically distributed around zero, we have
\begin{equation}\label{eq:umonotone}
p(\eta_2) - p(\eta_1) = \pr(\eta_1 U < - V < \eta_2 U \;\&\; U > 0) - \pr( \eta_1 U < V < \eta_2 U \;\&\; U > 0) \ge 0,
\end{equation}
where for the last inequality we apply~\eqref{eq:vsym}.

Armed with~\eqref{eq:vsym},~\eqref{eq:umonotone}, we are now ready to prove the desired result. First, write
\[
\taci = \frac{\int (\E[\pr(Y \ge y \mid \bz)^2] - \pr(Y \ge y)^2) d \mu_Y(y)}{\int (\pr(Y \ge y) - \pr(Y \ge y)^2) d \mu_Y(y)} = \frac{\int \E[\pr(Y \ge y \mid \bz)^2] d \mu_Y(y) - \frac{1}{3}}{\frac{1}{6}}.
\]
Expanding $Y$, we have
\[
\begin{aligned}
& \int \E[\pr(Y \ge y \mid \bz)^2] d \mu_Y(y) = \int \E[\pr(\eta h(\bz) + \varepsilon \ge y \mid \bz)^2] d \mu_Y(y) \\
& \qquad = \int \E[\pr(\varepsilon \ge y - \eta h(\bz) \mid \bz)^2] d \mu_Y(y).
\end{aligned}
\]
Writing $G(t) := \pr(\varepsilon \ge t)$ and using $\varepsilon \independent \bz$, we have
\begin{align*}
    & \int \E[\pr(Y \ge y \mid \bz)^2] d \mu_Y(y) = \int \E[G(y - \eta h(\bz))^2] d \mu_Y(y) = \E[G(Y_1 - \eta h(\bz))^2] \\
    & \qquad = \pr(\varepsilon_2, \varepsilon_3 \ge Y_1 - \eta h(\bz)) = \pr(\varepsilon_2 \wedge \varepsilon_3 - \varepsilon_1 \ge \eta (h(\bz_1) - h(\bz))) = p(\eta).
\end{align*}
The desired result is a direct consequence of the monotonicity property of $p(\eta)$ as displayed in~\eqref{eq:umonotone}.

\section{Additional simulation analysis}

This section examines the monotonicity of $\taci$ under the additive model $\by = \eta \bz + (1 - \eta) \bs{\varepsilon}$ as $\eta$ increases from $0$ to $1$. We consider $\dy = \dz = d$ for $d = 2, 5$. The variables $\bz$ and $\bs{\varepsilon}$ are generated as $\bz = B_{\bz} \bz'$ and $\bs{\varepsilon} = B_{\bs{\varepsilon}} \bs{\varepsilon}'$, respectively, using the same data-generating process described in \Cref{sec:normal}. The only modification is that for each $d$, we generate only 5 independent matrix pairs $(B_{\bz}, B_{\bs{\varepsilon}})$. Since the true $\taci$ lacks a closed-form expression, we approximate it via an empirical estimate using $n = 1000000$ samples. 

\Cref{fig:monotone} presents the approximated $\taci$ values for $\eta = 0.1, 0.2, \ldots, 0.9$ across various simulation settings. The results indicate that $\taci$ increases monotonically under all tested conditions, regardless of the choices for $(B_{\bz}, B_{\bs{\varepsilon}})$, or the distribution types of $\bz'$ and $\bs{\varepsilon}'$. In general, the rate of increase is slow for small $\eta$ and accelerates as $\eta$ grows. This suggests that the coefficient is more effective at distinguishing moderate and strong dependence than distinguishing independence and weak dependence, an effect that is magnified with increasing dimensionality.

\begin{figure}
    \centering
    \includegraphics[width=1\linewidth]{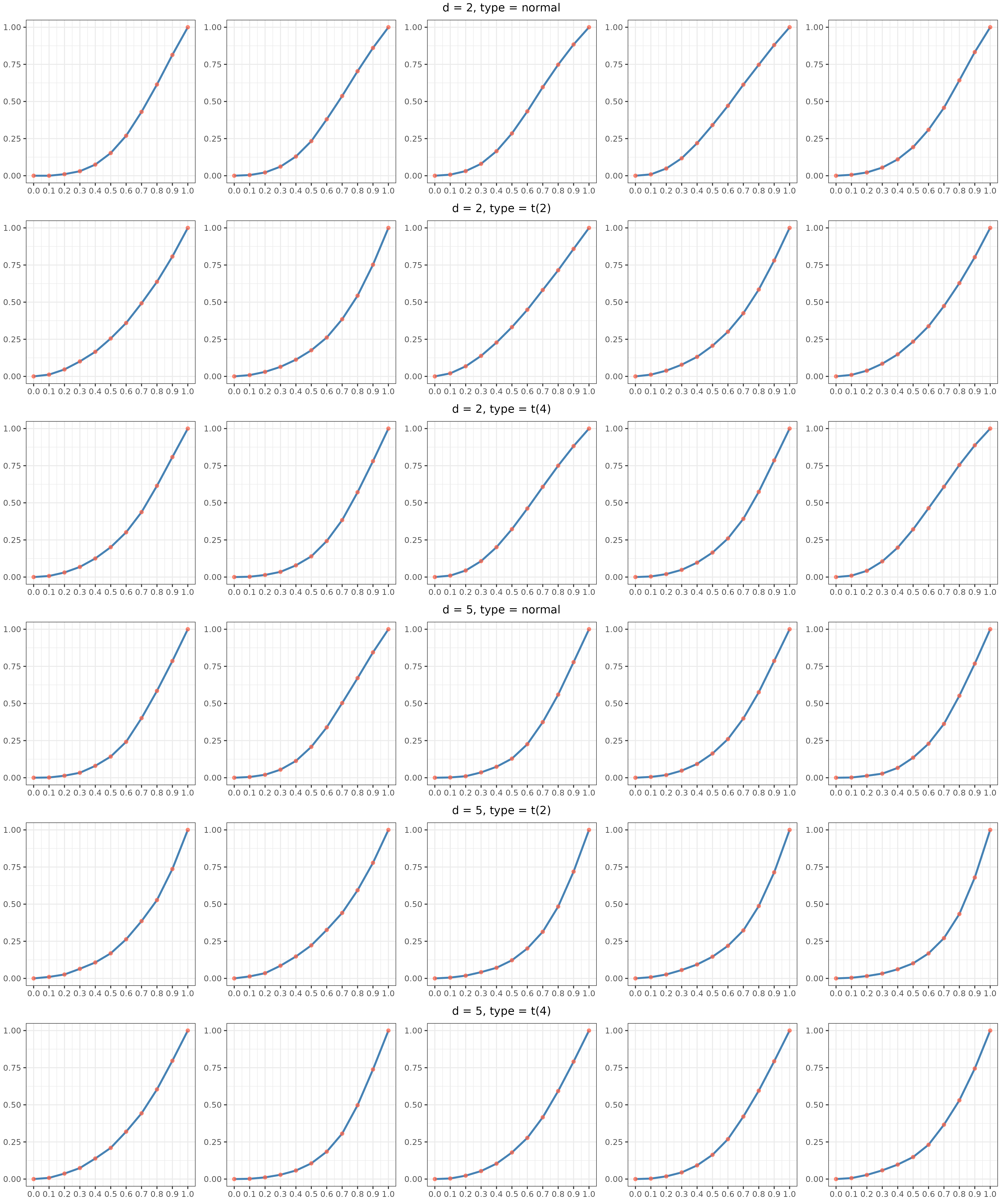}
    \caption{Values of $\taci$ for $\eta$ between $0$ and $1$. Here, $\bz, \bs{\varepsilon}$ are generated under various dimensions and distribution types prescribed at the top of each figure. The x-axis of each figure is the value of $\eta$; the y-axis is an approximation of $\taci$ constructed via $1000000$ samples. For each specific dimension and distribution type, we plot results for five different randomly generated matrix pairs $(B_{\bz}, B_{\bs{\varepsilon}})$.}
    \label{fig:monotone}
\end{figure}

\end{document}